\documentclass[10pt,a4paper]{amsart}

\usepackage{amsmath,yhmath}
\usepackage{amssymb}
\usepackage{amsfonts}
\usepackage{mathrsfs} 
\usepackage{bbm} 
\usepackage{mathtools}
\usepackage{tikz}
\usepackage{tikz-cd}
\usepackage[inline]{enumitem}
\allowdisplaybreaks
\usepackage{stmaryrd} 

\numberwithin{equation}{section}

\usepackage{caption}
\usepackage{geometry}
\usepackage{parskip}
\usepackage{setspace}
\usepackage{tikz}
\tikzset{Equal/.style={-,double line with arrow={-,-}}} 
\usetikzlibrary{decorations.markings}
\tikzset{double line with arrow/.style args={#1,#2}{decorate,decoration={markings,
mark=at position 0 with {\coordinate (ta-base-1) at (0,1pt);
\coordinate (ta-base-2) at (0,-1pt);},
mark=at position 1 with {\draw[#1] (ta-base-1) -- (0,1pt);
\draw[#2] (ta-base-2) -- (0,-1pt);
}}}}

\usepackage{aliascnt}
\usepackage[colorlinks, linkcolor=black,citecolor=blue, urlcolor=cyan]{hyperref} 
\newcommand*{\doi}[1]{\href{https://dx.doi.org/#1}{doi: #1}}

\newtheoremstyle{note}
{\topsep}
{\topsep}
{}
{0pt}
{\bfseries}
{.}
{0.5em }
{}
\theoremstyle{plain}

\newtheorem{Thm}{Theorem}[section]

\newtheorem*{Thm*}{Theorem}

\newaliascnt{prop}{Thm}
\newtheorem{Prop}[prop]{Proposition}
\aliascntresetthe{prop}

\newaliascnt{lemma}{Thm}
\newtheorem{Lemma}[lemma]{Lemma}
\aliascntresetthe{lemma}

\newaliascnt{coro}{Thm}
\newtheorem{Coro}[coro]{Corollary}
\aliascntresetthe{coro}

\newaliascnt{conjecture}{Thm}

\aliascntresetthe{conjecture}

\newtheorem{Lemma*}{Lemma}
\newtheorem*{Prop*}{Proposition}
\newtheorem*{Coro*}{Corollary}

\theoremstyle{definition}

\newaliascnt{def}{Thm}
\newtheorem{Def}[def]{Definition}
\aliascntresetthe{def}

\newtheorem*{Def*}{Definition}

\newaliascnt{Eg}{Thm}
\newtheorem{eg}[Eg]{Example}
\aliascntresetthe{Eg}

\newtheorem*{eg*}{Example}

\theoremstyle{remark}

\newaliascnt{rmk}{Thm}
\newtheorem{RMK}[rmk]{Remark}
\aliascntresetthe{rmk}

\newtheorem*{RMK*}{Remark}

\theoremstyle{plain}

\newtheorem{iThm}{Theorem}[section]

\newtheorem*{iThm*}{Theorem}

\newaliascnt{iprop}{iThm}

\aliascntresetthe{iprop}

\newtheorem{iconjecture}[iThm]{Conjecture}

\newaliascnt{iconjecture}{iThm}
\newtheorem{iConjecture}[iconjecture]{Conjecture}

\aliascntresetthe{iconjecture}

\theoremstyle{definition}

\newaliascnt{idef}{iThm}

\aliascntresetthe{idef}

\newtheorem*{iDef*}{Definition}

\newaliascnt{iEg}{iThm}

\aliascntresetthe{iEg}

\newtheorem*{ieg*}{Example}

\theoremstyle{remark}

\newaliascnt{irmk}{iThm}
\newtheorem{iRMK}[irmk]{Remark}
\aliascntresetthe{irmk}

\newtheorem*{iRMK*}{Remark}

\newcommand{\bC} { {\mathbb{C}}}
\newcommand{\bF} { {\mathbb{F}}}

\newcommand{\bZ} { {\mathbb{Z}}}
\newcommand{\bR} { {\mathbb{R}}}
\newcommand{\bK} { {\mathbb{K}}}
\newcommand{\bN} { {\mathbb{N}}}
\newcommand{\bO} { {\mathbb{N}_{\geq 2}}}

\newcommand{\cD} { {\mathcal{D}}}

\newcommand{\cK} { {\mathcal{K}}}
\newcommand{\sC} { {\mathscr{C}}}
\newcommand{\sP} { {\mathscr{P}}}
\newcommand{\sQ} { {\mathscr{Q}}}
\newcommand{\supp}{{\textnormal{supp}}}
\newcommand{\RHOM}{{ \textnormal{RHom}  }}
\newcommand{\HOM}{{ \textnormal{Hom}  }}

\newcommand{\RHHOM}{{ \textnormal{R}\mathcal{H}om  }}
\newcommand{\dotimes}{{\overset{L}\otimes }}
\newcommand{\dboxtimes}{{\overset{L}\boxtimes }}
\newcommand{\pt}{{\textnormal{pt}}}
\newcommand{\bq} { {\textbf{q}}}
\newcommand{\bp} { {\textbf{p}}}
\newcommand{\bx} { {\textbf{x}}}
\newcommand{\tnT}{{\textnormal{T}}}
\newcommand{\tnR}{{\textnormal{R}}}
\newcommand{\id}{{\textnormal{Id}}}
\newcommand{\Mod}{{\textnormal{Mod}}}

\makeatletter
\newcommand\boxstar{\mathrel{\boxcls@{\star}}}
\newcommand{\boxcls@}[1]{%
  \vphantom{\Box}
  \mathpalette\boxcls@@{#1}%
}
\newcommand{\boxcls@@}[2]{%
  \ooalign{$\m@th#1\Box$\cr
  \hidewidth\boxcls@fix{#1}\hbox{$\m@th#1#2\mkern 2.4mu$}\cr}}

\newcommand\boxcls@fix[1]{%
  \ifx#1\displaystyle
    \raise.225ex
  \else
    \ifx#1\textstyle
      \raise.225ex
    \else
      \ifx#1\scriptstyle
        \raise.180ex
      \else
        \raise.150ex
      \fi
    \fi
  \fi
}
\makeatother

\title{Capacities from the Chiu-Tamarkin complex}
\author{Bingyu Zhang}
\date{\today}

\begin{document}
\dedicatory{To Claude Viterbo on the Occasion of his Sixtieth Birthday}
\maketitle

\begin{abstract}
    In this paper, we construct a sequence $(c_k)_{k\in\bN}$ of symplectic capacities based on the Chiu-Tamarkin complex $C^{\bZ/\ell}_T$, a $\bZ/\ell$-equivariant invariant coming from the microlocal theory of sheaves. We compute $(c_k)_{k\in\bN}$ for convex toric domains, which are the same as the Gutt-Hutchings capacities. Our method also works for the prequantized contact manifold $T^*X\times S^1$. We define a sequence of ``contact capacities" $([c]_k)_{k\in\bN}$ on the prequantized contact manifold $T^*X\times S^1$, and we compute them for prequantized convex toric domains.
\end{abstract}

\section{Introduction}

\subsection{Symplectic Embedding}

A symplectic manifold $(X,\omega)$ is a manifold with a non-degenerate closed $2$-form $\omega$. Classically it appears naturally as phase spaces in Hamiltonian Mechanics. An embedding $\varphi:(X,\omega) \hookrightarrow (X',\omega')$ is called symplectic if $\varphi^*\omega'=\omega$. A basic question in symplectic geometry is to decide when a symplectic embedding between two symplectic manifolds exists. The first result, the origin of the question, is the Gromov non-squeezing theorem:
\begin{iThm}[\cite{gromov1985}]Equip $\bR^{2d}$ with the linear symplectic form. Let $B_{\pi r^2}=\{(x,p)\in \bR^{2d} : \|x \|^2+\|p\|^2 < r^2\}$, and $Z_{\pi R^2}=\{(x,p)\in \bR^{2d} : x_1^2+p_1^2< R^2 \}$. 

If there is a symplectic embedding $\varphi: B_{\pi r^2}\rightarrow \bR^{2d}$ such that $\varphi(B_{\pi r^2})\subset Z_{\pi R^2}$, then $r\leq R$.
\end{iThm}
A structure related to the embedding question is the so-called symplectic capacity. One example is the Gromov width, which is hard to compute. There are other capacities defined by generating functions \cite{viterbo1992symplectic}, Hamiltonian dynamics \cite{EkelandHoferII}, and $J$-holomorphic curves \cite{ECHhutchings2011,gutthutchings2018capacities, SiegelHigherCapacities}. A great survey about symplectic capacities is \cite{quantitivesymplectic.2010}. When the dimension is 4, the ECH capacity is a very effective tool. When the dimension is greater than 4, we know fewer results.

The Ekeland-Hofer capacity $(c^{\text{EH}}_k)_{k\in\bN}$ is a sequence of symplectic capacities defined for compact star-shaped domains in $T^*\bR^{d}$ for all $d$, which is defined using Hamiltonian dynamics. The computation of $c^{\text{EH}}_k$ is known for ellipsoids and poly-disks, say:
\begin{align*}
    &c_k^{\text{EH}}(E(a_1,\dots,a_d))=\min\left\{T: \sum_{i=1}^d \big\lfloor {\frac{T}{a_i}} \big\rfloor \geq k\right\}\qquad
    c_k^{\text{EH}}(D(a_1,\dots,a_d))=ka_1,
\end{align*}
where 
\begin{align*}
    &E(a_1,\dots,a_d)=\left\lbrace u\in \bC^d:\sum_{i=1}^d \frac{\pi|u_i|^2}{a_i} < 1 \right\rbrace,\\
    &D(a_1,\dots,a_d)=\left\lbrace u\in \bC^d:{\pi|u_i|^2} < a_i, \forall i=1,\dots,d \right\rbrace,\\
    &\text{with }0<a_1\leq \cdots \leq a_d.
\end{align*}
On the other hand, Gutt and Hutchings constructed a sequence of capacities $(c_k^{
\text{GH}})_{k\in\bN}$ in \cite{gutthutchings2018capacities} using the positive $S^1$-equivariant symplectic homology. For an open set $\Omega \subset \bR^d$, the open set \[X_{\Omega}=\left\{u\in \bC^d: (\pi|u_1|^2,\dots,\pi|u_d|^2)\in \Omega\right\}\]
is called a toric domain. We say $X_{\Omega}$ is convex if $\widehat{\Omega}=\left\{(x_1,\dots,x_d): (|x_1|,\dots,|x_d|)\in \Omega\right\}$ is convex, and is concave if $\bR^d_{\geq 0}\setminus \Omega$ is convex. The toric domain $X_\Omega$ is determined by $\Omega\cap\bR^d_{ \geq 0}$. So it is free to choose a suitable $\Omega$. In particular, we always assume
$ \bR^d_{ \leq 0} \subset\Omega.$
If $X_{\Omega}$ is a convex or a concave toric domain, one can indeed take $\Omega$ to be convex or concave (in the usual sense) and satisfying the condition $ \bR^d_{ \leq 0} \subset\Omega$. For example, ellipsoids $E(a)=X_{\Omega_{E(a)}}$ and poly-disks $D(a)=X_{\Omega_{D(a)}}$ are convex toric domains, where
\[\Omega_{E(a)}=\left\{(x_1,\dots,x_d):\sum_{i=1}^d \frac{x_i}{a_i} < 1\right\},\quad \Omega_{D(a)}=\left\{(x_1,\dots,x_d):x_i < a_i,\,\forall i=1,\dots,d\right\}.\]
Gutt and Hutchings computed $c_k^{
\text{GH}}$ for both convex and concave toric domains. For example, when $X_{\Omega}$ is convex, they showed that 
\begin{equation}
  c_k^{\text{GH}}(X_{\Omega})=\min\left\{\|v\|_\Omega^*: v\in \mathbb{N}^d, \sum_{i=1}^dv_i=k\right\}=\inf\left\{T\geq 0:\exists  {z}\in \Omega_T^\circ, I( {z})\geq k\right\} \label{capacities of convex toric domains},
\end{equation}
where
\begin{align}\label{convex toric domain notation}
\begin{aligned}
&\|v\|_\Omega^*=\max\{\langle v,w\rangle: w\in \Omega\},\\
&\Omega_T^\circ=\{z\in \bR^d: T+\langle z, \zeta \rangle \geq 0, \forall \zeta \in \Omega\},\quad I(z)=\sum_{i=1}^d \big\lfloor {-z_i} \big\rfloor. 
\end{aligned}
\end{align}
So one can observe that for ellipsoids and poly-disks, $c_k^{
\text{GH}}=c_k^{\text{EH}}$.

Unfortunately, even for ellipsoids, we know that the obstructions, given by Ekeland-Hofer capacities, and then the Gutt-Hutchings capacities, are not sharp. One new progress on higher dimensional embeddings is given by Siegel in \cite{SiegelHigherCapacities, CompSiegelHigherCapacities}. Siegel gave sharp obstructions for embeddings between some stabilized ellipsoids. 

In this paper, we construct a sequence of symplectic capacities $(c_k)_{k\in\bN}$ for open sets in a cotangent bundle $T^*X$ with an orientable base $X$. We denote $ \text{Open}(T^*X)$ the set of open sets in $T^*X$. Our main ingredient is the complex $C^{\bZ/\ell}_T(U,\bK)$ defined by Tamarkin and Chiu in \cite{mc-Tamarkin,chiu2017}, where $U$ is admissible (for example bounded open sets), $T\geq 0$, and $\ell\in \bO$. There exists a structure of $\bK[u]$-module on $H^*C^{\bZ/\ell}_T(U,\bK)$, and a fundamental class $\eta^{\bZ/\ell}_T(U,\bK) \in H^0C^{\bZ/\ell}_T(U,\bK) $. For admissible open sets $U$, we define (see \autoref{definition of capacities})
\begin{equation*} \textnormal{Spec}(U,k) \coloneqq
\left\lbrace
  T \geq 0:\begin{aligned} &\exists p\text{ prime such that }\forall\ell \in \bO, \, p_\ell \geq p,\, \\&\eta^{\bZ/\ell}_T(U,\bF_{p_\ell}) \in u^kH^{*}C^{\bZ/\ell}_T(U,\bF_{p_\ell})
  \end{aligned}
\right\rbrace,
\end{equation*}
and
\[c_k(U)\coloneqq \inf \textnormal{Spec}(U,k) \in [0,+\infty].\]For a general open set $U$, we define
$c_k(U)=\sup\{c_k(O): O\subset U,\,O\text{ is admissible}\}$.
Then we show
\begin{iThm}[\autoref{capacity property symplectic}] The functions $c_k:\text{Open}(T^*X)\rightarrow [0,\infty]$ satisfy the following:
\begin{enumerate}[fullwidth]
    \item $c_k \leq c_{k+1}$ for all $k\in \bN$.
    \item For two open sets $U_1 \subset U_2$, we have $c_k(U_1) \leq c_k(U_2)$.
    \item For a compactly supported Hamiltonian isotopy $\varphi:I\times T^*X \rightarrow T^*X$, we have 
    $c_k(U)=c_k(\varphi_z(U)).$ 
       \item If $X=\bR^d$, then $c_k(rU)=r^2c_k(U)$ for all $k\in \bN$ and $r>0$.
    \item Suppose $U=\{H<1\}$ is admissible such that $\partial U=\{H=1\}$ is a non-degenerated hypersurface of restricted contact type defined by a Hamiltonian function $H$. If $c_k(U) < \infty $, then $c_k(U)$ is represented by the action of a closed characteristic in the boundary  $\partial U$.
    
    \item $c_k(U)>0$ for all open sets $U$.
\end{enumerate}
\end{iThm}
Moreover, based on the structural theorem (\autoref{structure toric domain module}) of $H^*C^{\bZ/\ell}_T(X_{\Omega},\bF_{p_\ell})$, where $X_{\Omega}$ is a convex toric domain, we can compute $c_k(X_{\Omega})$ as follows: \begin{equation}c_k(X_{\Omega})=\inf\left\{T\geq 0:\exists  {z}\in \Omega_T^\circ, I( {z})\geq k\right\}.\label{capacities of convex toric domains2}\end{equation}
Therefore, $c_k(X_{\Omega})=c_k^{\text{GH}}(X_{\Omega})$ by \eqref{capacities of convex toric domains} and \eqref{capacities of convex toric domains2}. 
 
On the other hand, one may ask the concave case. It is explained in \autoref{concave} that some technical issues exist. So we cannot derive a clear structure theorem as \autoref{structure toric domain module}, and then the computation of capacities is not completely clear. However manual computation of some examples shows the coincidence with Gutt-Hutchings capacities is still true.

Based on the computation on the convex toric domains and concave toric domains, Gutt and Hutchings conjectured (\cite[Conjecture 1.9]{gutthutchings2018capacities}) that, for a bounded star-shaped domain $U$ and for all $k\in \bN$,\[c_k^{\text{EH}}(U)=c_k^{\text{GH}}(U).\] 
 In fact, the result $c_1^{\text{EH}}(U)=c_1^{\text{GH}}(U)=\text{Minimal action}$ has been proven by Irie\cite{irie2019symplectic} for convex body $U$. Comparing to our result, we hope the consistency could be extended to $c_k$ as well.
\begin{iConjecture}\label{conjecture}For a bounded star-shaped domain $U\subset \bR^{2d}$ and for all $k\in \bN$, there is\[c_k^{\textnormal{EH}}(U)=c_k^{\textnormal{GH}}(U)=c_k(U).\] 
\end{iConjecture}
\begin{iRMK}Recently, Jean Gutt and Vinicius G. B. Ramos claim that they prove $c_k^{\text{EH}}(U)=c_k^{\text{GH}}(U)$ for star-shaped domain $U\subset \bR^{2d}$.
\end{iRMK}

\subsection{Contact Embedding}
Contact geometry is the odd-dimensional cousin of symplectic geometry. A (co-oriented) contact manifold $(X,\alpha)$ consists of a manifold $X$ of dimension $2n+1$ and a $1$-form $\alpha$ such that $\alpha\wedge d\alpha^n \neq 0$. An embedding $\varphi: (X,\alpha) \hookrightarrow (X',\alpha')$ between two contact manifolds is called contact if $\varphi^*\alpha'=e^{f}\alpha$ for some function $f\in C^\infty(X)$. We also study the embedding question in contact geometry. The pioneering work of Eliashberg, Kim, and Polterovich\cite{eliashberg2006geometry} promote our understanding of the contact embedding question a lot. Let us explain here.

A naive attempt is to study the non-squeezing problem in the $1$-jet bundle $J^1\bR^d=T^*\bR^d\times \bR$ equipped with the contact form $\alpha=dt+\bq d\bp$. However the re-scaling map $(\bq,\bp,t)\mapsto (r\bq,r\bp,r^2t)$, which is a contactomorphism, squeezes any compact set into an arbitrary small neighborhood of the origin when $r$ is big enough. This conformal naturality of $1$-jet space illustrates us that we can study the prequantized space $T^*\bR^d\times S^1_\sigma$, where $S^1_\sigma$ is a circle. Here we equip $T^*\bR^d\times S^1_\sigma$ with a contact form $\alpha=d\sigma+\frac{1}{2}(\bq d\bp-\bp d\bq)$. However there is a global contactomorphism $F_N: T^*\bR^d\times S^1_\sigma\rightarrow T^*\bR^d\times S^1_\sigma$ defined as follows: We use complex coordinates $T^*\bR^d \cong \bC^d$, and then $F_N(z,\sigma)\coloneqq (\nu(\sigma)e^{2\pi N \sigma}z, \sigma)$, where $\nu(\sigma)=(1+N\pi|z|^2)^{-1/2}$.
One can compute directly that $F_N$ is still embedding any ball into arbitrary small neighborhood of $\{0\}\times S^1$ for $N$ big enough. However we notice that $F_N$ is not compactly supported. So loc. cit. proposed the following definition.
\begin{iDef*}{\cite[p1636]{eliashberg2006geometry}} Let $(W,\alpha)$ be a contact manifold. If $U_1,U_2 \subset W$ are two open subsets, we say that $U_1$ is squeezed into $U_2$ if there exists a compactly support contact isotopy $\varphi:[0,1]_s \times\overline{U_1}  \rightarrow W$ such that $\varphi_0=\textnormal{Id}$, and $\varphi_1(\overline{U_1}) \subset U_2$.
\end{iDef*}

An interesting phenomenon, which does not appear in the symplectic situation, is the scale of the ball will affect the validity of squeezing. About the large scale phenomenon, Eliashberg, Kim, and Polterovich give a very nice physical explanation using the quantization process. Two results about both squeezing and non-squeezing of prequantized balls $B_{\pi R^2}\times S^1 $ are:
\begin{iThm}
\begin{enumerate}[fullwidth]
\item{{\cite[Theorem 1.3]{eliashberg2006geometry}}} Suppose $d \geq 2$. Then for all $0 <\pi r^2, \pi R^2<1$, one can squeeze the prequantized ball $B_{\pi R^2}\times S^1 $ into $B_{\pi r^2}\times S^1 $ whatever the relation between $r$ and $R$ is.
\item{{\cite[Theorem 1.2]{eliashberg2006geometry}}} If there exists an integer $m\in[\pi r^2,\pi R^2]$, then $B_{\pi R^2}\times S^1$ cannot be squeezed into $B_{\pi r^2}\times S^1$.\end{enumerate}
\end{iThm}

Then the only case left about the contact non-squeezing is: what will happen if there is an integer $m$ such that $m<\pi r^2 < \pi R^2 <m+1$? It is solved by Chiu using the microlocal theory of sheaves\cite{chiu2017}, and by Fraser using technique of $J$-holomorphic curves\cite{fraser2016} in the spirit of \cite{eliashberg2006geometry}. They proved the following:
\begin{iThm}[{\cite{chiu2017,fraser2016}}]If $1\leq \pi r^2<\pi R^2$, then $B_{\pi R^2}\times S^1$ cannot be squeezed into $B_{\pi r^2}\times S^1$.
\end{iThm}

The second purpose of the paper is to explain how Chiu's work could be used to define ``contact capacities" on prequantization $T^*X\times S^1$ for orientable $X$. The notion of admissible open sets still makes sense. The presence of the scale feature makes us consider the so-called big admissible open sets. Concretely, we say a contact admissible open set ${\mathscr{U}} \subset T^*\bR^d\times S^1$ is big if there is a ball $B_a\times S^1$ for $a>1$ that can be embedded in ${\mathscr{U}}$ by a compactly supported contact isotopy on $T^*\bR^d\times S^1$. For a prequantized convex toric domain $X_{\Omega}\times S^1$, where $X_\Omega$ is a toric domain, it is big if $\|\Omega^\circ_1\|_{\infty}=\max_{z\in \Omega^\circ_1}\|z\|_\infty<1$. Besides, to obtain the contact invariance, we need to restrict to $T/\ell\in \bN$ situation in the contact case. Here, we denote $ \mathbb{P}$ the set of all prime numbers. Then we can define
\begin{equation*} [\textnormal{Spec}]({\mathscr{U}},k) \coloneqq
\{p \in \mathbb{P}:\eta^{\bZ/p,c}_{p}({\mathscr{U}},\bF_{p})\in u^kH^{*}\sC^{\bZ/p}_{p}({\mathscr{U}},\bF_{p})\}
\end{equation*}
and \[[c]_k({\mathscr{U}})\coloneqq \min [\textnormal{Spec}]({\mathscr{U}},k) \in  \mathbb{P}.\]
For a general open set ${\mathscr{U}}$, we define $[c]_k({\mathscr{U}})=\sup\{[c]_k(O): O\subset {\mathscr{U}},\,O\text{ is admissible}\}$.

Then we have, \autoref{contact capacities}, \autoref{computation of capacities of contact convex toric domian}:
\begin{iThm} The functions $[c]_k:\text{Open}(T^*X\times S^1)\rightarrow  \mathbb{P}$ satisfy the following:
\begin{enumerate}[fullwidth]
    \item $[c]_k \leq [c]_{k+1}$ for all $k\in \bN$.
    \item For two open sets ${\mathscr{U}}_1 \subset {\mathscr{U}}_2$, we have $[c]_k({\mathscr{U}}_1) \leq [c]_k({\mathscr{U}}_2)$.
    \item For a compactly supported contact isotopy $\varphi:I\times T^*X\times S^1   \rightarrow T^*X\times S^1$, we have $[c]_k({\mathscr{U}})=[c]_k(\varphi_z({\mathscr{U}}))$.
    \item For a big prequantized convex toric domain $X_\Omega \times S^1 \subset T^*\bR^d \times S^1$, we have
\begin{align*}
c_k(X_\Omega\times S^1)=\min\left\{p\in  \mathbb{P}: \exists  {z}\in \Omega_{p}^\circ, I( {z})\geq k\right\}
    =\min\left\{p\in \mathbb{P}: p\geq c_k(X_\Omega) \right\}.
\end{align*}
\end{enumerate}
\end{iThm}
A more concrete example is as follows. Suppose $X_\Omega\times S^1 = E(3,4)\times S^1$, we have.
\begin{center}
\begin{tabular}{c|cccccccccccc}
\hline
  $k$   & 1 & 2 & 3& 4& 5& 6& 7& 8& 9 &10&11\\
\hline 
 $c_k$  &3&4& 6 &  8& 9& 12& 12& 15&16&18&20\\
\hline 
$[c]_k$ & 3 & 5 & 7&  11& 11& 13& 13& 17 &17&19 &23\\
\hline 
\end{tabular}    
\end{center}

\subsection{Microlocal Theory of Sheaves and the Chiu-Tamarkin complex}
The main ingredient of our work is the microlocal theory of sheaves, introduced by Kashiwara and Schapira with motivation from algebraic analysis. We refer to \cite{KS90}. 

The main idea we use in the paper is the notion of microsupport or singular support, which is defined as follows: For a ground commutative ring $\bK$, let $D(X)$ be the derived category of complexes of sheaves of $\bK$-modules over $X$. For an object $F\in D(X)$, we can associate a set $SS(F)\subset T^*X$, which is called the microsupport of $F$. It is proved in \cite{KS90} that $SS(F)$ is always a closed conic and coisotropic subset of $T^*X$. Moreover, when $X$ is real analytic, $SS(F)$ is Lagrangian if and only if $F$ is (weakly) constructible. This result inspires us that the sheaf theory plays its role in symplectic geometry and contact geometry. 

For instance, Tamarkin develops a new method to study displacibility of Lagrangians in \cite{tamarkin2013}. Guillermou gives sheaf theoretical proofs of Gromov-Eliashberg $C^0$-rigidity and of the result by Abouzaid and Kragh that closed exact Lagrangians in cotangent bundles are homotopically equivalent to the zero section. See \cite{conicexactlagrangianGuillermou,C0rigidityGuillermou,3cuspsconjectureGuillermou} and the survey \cite{guillermou2019sheaves} about these topics. Asano-Ike develop a lower bound for the symplectic displacement energy and a lower bound for the number of intersection point of some rational Lagrangian immersions using numerical information from the Tamarkin Category (\autoref{tamarkincategory}) in \cite{asano2017persistence,asano2020sheaf}. On the other hand, there are many works studying the category of sheaves from the point of view of the Fukaya category, see the work of Nadler and Zaslow on the compact Fukaya category \cite{nadler2009constructible,nadler2009microlocal}; and the work of Nadler\cite{Nadlerwarppedfukaya2016}, and of Ganatra, Pardon, and Shende on the wrapped Fukaya category\cite{GPS3}. 

Now, let us review ideas of Tamarkin in \cite{tamarkin2013}. Tamarkin suggested studying the category of sheaves localized with respect to sheaves microsupported in non-positive direction, that is,
 the localization of $D(X\times \bR)$ with respect to the full thick subcategory $\{F:SS(F)\subset \{\tau\leq 0\} \}$.
This localization 
is equivalent to the essential image of the functor $ \bK_{[0,\infty)}\star:D(X\times \bR)\rightarrow D(X\times \bR)$,
where $\star: D(X\times \bR)\times D(X\times \bR)\rightarrow D(X\times \bR)$ is the convolution.
We denote these two equivalent categories by $\mathcal{D}(X)$ and call them the Tamarkin category of $X$.
The category $\mathcal{D}(X)$ is triangulated.

Since the microsupport is conic, Tamarkin considers the following conification
procedure: for a given closed set $A$ in the cotangent bundle $T^*X$ we set $\widehat{A}=\{(x,p,t,\tau)\in T^*(X\times \bR) :(x,p/\tau)\in A, \tau >0\}$. We are interested in the category of sheaves on $X\times \bR$ microsupported in $\widehat{A}$, that is, $F\in \mathcal{D}(X)$ such that $SS(F) \cap \{\tau >0\} \subset \widehat{A}$, and we denote the category they form by $\mathcal{D}_A(X)$. Categorically, this category and its semi-orthogonal complement $^\perp \mathcal{D}_A(X)$ are completely determined by the projectors from $\mathcal{D}(X)$ onto them. Hopefully, we could understand the geometry of $A$ from these projectors. One way to study these projectors is to represent them as integral functors defined by kernels, for example $ \bK_{[0,\infty)}\star$ introduced by Tamarkin, or the {\em cut-off} functors of Kashiwara and Schapira \cite{KS90}. 

Here, we start with the symplectic case. An open set $U\subset T^*X$ whose projector $\mathcal{D}(X) \rightarrow \mathcal{D}_{U}(X)=^\perp \mathcal{D}_{T^*X\setminus U}(X)$ is represented by a convolution functor $\star P_U: \mathcal{D}(X)\rightarrow \mathcal{D}(X)$ is called admissible. We will see later that bounded open sets and toric domains are all admissible. One particularly interesting example is the open ball $U=B_{\pi R^2}$. Chiu constructed a kernel for $B_{\pi R^2}$ using the idea of generating functions in \cite{chiu2017}, which is the main ingredient of his proof of contact non-squeezing.

Another ingredient of Chiu's proof is (a contact version of) an object $C^{\bZ/\ell}_T(U,\bK) \in D_{\bZ/\ell}(\pt)$ defined using $P_U$, where $\bZ/\ell$ is the cyclic group, $\bK$ is a field, and $D_{\bZ/\ell}(\pt)$ denotes the equivariant derived category over point (see \cite{BernsteinLunts}). Remark that Chiu denotes our $\ell$ as $N$.

Chiu did not define the symplectic version $C^{\bZ/\ell}_T(U,\bK)$ explicitly, while his idea applies directly to defining the symplectic version we presented in \autoref{def CT complex}. His discussions on contact invariance and computation work perfectly to the symplectic case as we will present in the following. We have that the object $C^{\bZ/\ell}_T(U,\bK)$ is a Hamiltonian invariant of an admissible open set $U$ for $\ell \in \bN$ and $T\geq 0$. We will define a fundamental class $ \eta^{\bZ/\ell}_T(U,\bK)\in H^0C^{\bZ/\ell}_T(U,\bK)$, and also see that $H^*C^{\bZ/\ell}_T(U,\bK)$ is a left graded module over the Yoneda algebra $A=\textnormal{Ext}^*_{\bZ/\ell}(\bK,\bK)$. If $\text{char}(\bK)|\ell$, we have $A\cong \bK[u,\theta]$ where $|u|=2$, $|\theta|=1$, and $u^2=k\theta$ ($k$ depends on the parity of $\ell$). To achieve the condition $\text{char}(\bK)|\ell$, we can take $\bK=\bF_{p_\ell}$ to be the finite field of order $p_\ell$ where $p_\ell$ is the minimal prime factor of $\ell$.

We will also discuss, in \autoref{section: contact inv}, the contact version $\sC^{\bZ/\ell}_{n\ell}({\mathscr{U}},\bK)$ that Chiu originally defines in \cite{chiu2017}, for a contact admissible open set ${\mathscr{U}}\subset T^*X\times S^1$ and $(n,\ell)\in \bN_0\times \bN$. The differences are that $\sC^{\bZ/\ell}_{n\ell}({\mathscr{U}},\bK)$ is defined for $(n,\ell)\in \bN_0\times \bN$ while $C^{\bZ/\ell}_T(U,\bK)$ is defined for $(T,\ell)\in \bR_{\geq 0}\times \bN$ and that $\sC^{\bZ/\ell}_{n\ell}({\mathscr{U}},\bK)$ is invariant under contact isotopies. The fundamental class $ \eta^{\bZ/\ell,c}_{n\ell}({\mathscr{U}},\bK)\in H^0\sC^{\bZ/\ell}_{n\ell}({\mathscr{U}},\bK)$ and the $A$-module structure can also be defined. We will see that if the open set $U\subset T^*X$ is symplectic admissible, the prequantized open set ${\mathscr{U}}=U\times S^1\subset T^*X\times S^1$ is contact admissible, and we have an isomorphism $\sC^{\bZ/\ell}_{n\ell}(U\times S^1,\bK)\cong C^{\bZ/\ell}_{n\ell}(U,\bK)$, which preserves the fundamental class. The isomorphism is helpful since even though the symplectic version is a priori not a contact invariant, it computes the contact version. In this sense, Chiu computed $C^{\bZ/\ell}_{\ell}(B_{\pi R^2},\bK)$ using $\sC^{\bZ/\ell}_{\ell}(B_{\pi R^2}\times S^1,\bK)$.

Chiu's proof for the contact non-squeezing theorem can be organized by our language in the following steps:
\begin{itemize}[fullwidth]
\item When $\ell$ is an odd prime number and $\pi r^2>1$, Chiu constructs an isomorphism of $A$-modules:
\[ H^*\sC^{\bZ/\ell}_{\ell}(B_{\pi r^2}\times S^1,\bF_\ell)\cong {u^{-d\lfloor \ell/\pi r^2\rfloor}}\bF_\ell[u,\theta], \]
and an element $\Lambda_r=k u^{-d\lfloor \ell/\pi r^2\rfloor}$ such that $\eta^{\bZ/\ell,c}_{\ell}(B_{\pi r^2}\times S^1,\bF_\ell)=u^{d\lfloor \ell/\pi r^2\rfloor}\Lambda_r \neq 0$.

\item The fundamental class is preserved under the contact invariance. Specifically, for a compactly supported contact isotopy $\varphi:I\times T^*\bR^{d}\times S^1 \rightarrow T^*\bR^{d}\times S^1$ and $z\in I$, we have an isomorphism of $A$-modules 
\[\Phi_{z\ell}^{\bZ/\ell,c}:H^*\sC^{\bZ/\ell}_{\ell}(\varphi_z(B_{\pi r^2}\times S^1),\bF_\ell)\cong H^*\sC^{\bZ/\ell}_{\ell}(B_{\pi r^2}\times S^1,\bF_\ell)\]
such that $\eta^{\bZ/\ell,c}_{\ell}(\varphi_z(B_{\pi r^2}\times S^1),\bF_\ell)$ is mapped to $\eta^{\bZ/\ell,c}_{\ell}(B_{\pi r^2}\times S^1,\bF_\ell)$.

\item If there exists an inclusion $B_{\pi R^2}\times S^1\subset B_{\pi r^2}\times S^1$, for $R>r$, we have a degree $0$ morphism of $A$-modules
\[i:H^*\sC^{\bZ/\ell}_{\ell}(B_{\pi r^2}\times S^1,\bF_\ell)\rightarrow H^*\sC^{\bZ/\ell}_{\ell}(B_{\pi R^2}\times S^1,\bF_\ell),\]
which preserves the fundamental class.
In particular, we have $\eta^{\bZ/\ell,c}_{\ell}(B_{\pi R^2}\times S^1,\bF_\ell)=u^{d\lfloor \ell/\pi r^2\rfloor}i(\Lambda_r)$ in $H^0\sC^{\bZ/\ell}_{\ell}(B_{\pi R^2}\times S^1,\bF_\ell)$.
\item However, the degree comparison makes $i(\Lambda_r)=0$ for large enough $\ell$. This is a contradiction because we know that $\eta^{\bZ/\ell,c}_{\ell}(B_{\pi R^2}\times S^1,\bF_\ell)\neq 0$.
\end{itemize}
We use the fundamental class, the module structure, and the invariance to define the capacities as we presented in the last two subsections. In this sense, our capacities form a numerical package of Chiu's arguments. Meanwhile, our third main result generalizes Chiu's computation of the Chiu-Tamarkin complex for balls to convex toric domains $X_{\Omega}\subset \bC^d=T^*\bR^{d}$.

As we already know that if ${\mathscr{U}}=U\times S^1$, the computation for both symplectic case and contact case is essentially the same. So, it is enough to compute the symplectic version. We will construct a good kernel for $X_\Omega$ based on Chiu's construction and then compute the symplectic Chiu-Tamarkin complex $C^{\bZ/\ell}_T(X_\Omega,\bF_{p_\ell})$. We will show the following structural theorem.
\begin{iThm}[\autoref{structure toric domain module}]For a convex toric domain $X_\Omega \subset T^*\bR^{d}$, and $\ell\in \bN_{\geq 2}$. If $0\leq T <p_\ell /\|\Omega^\circ_1\|_{\infty}$, we have

\begin{itemize}[fullwidth]
    
\item For each $Z\in \Omega^\circ_T$ (see \eqref{convex toric domain notation}), the inclusion of the segment $\overline{OZ} \subset \Omega^\circ_T$ induces a decomposition of the fundamental class $\eta^{\bZ/\ell}_T(X_{\Omega},\bF_{p_\ell})=u^{I(Z)}\Lambda_{Z,\ell}$ for a non-torsion element $\Lambda_{Z,\ell}\in H^{-2I(Z)}C^{\bZ/\ell}_T(X_\Omega,\bF_{p_\ell})$. In particular, $\eta^{\bZ/\ell}_T(X_{\Omega},\bF_{p_\ell})$ is non-zero.

\item The minimal cohomology degree of $H^*C^{\bZ/\ell}_T(X_\Omega,\bF_{p_\ell})$ is exactly $-2I(\Omega^\circ_T)$, i.e.,
    \[ H^*C^{\bZ/\ell}_T(X_\Omega,\bF_{p_\ell})\cong  H^{\geq -2I(\Omega^\circ_T)}C^{\bZ/\ell}_T(X_\Omega,\bF_{p_\ell}), \] 
    and
\[ H^{-2I(\Omega^\circ_T)}C^{\bZ/\ell}_T(X_\Omega,\bF_{p_\ell})\neq 0 . \] 
\item $H^*C^{\bZ/\ell}_T(X_\Omega,\bF_{p_\ell})$ is a finitely generated $\bF_{p_\ell}[u]$-module. The free part is isomorphic to $A=\bF_{p_\ell}[u,\theta]$, so $H^*C^{\bZ/\ell}_T(X_\Omega,\bF_{p_\ell})$ is of rank $2$ over $\bF_{p_\ell}[u]$. 

The torsion part is located in cohomology degree $[-2I(\Omega^\circ_T),-1]$.  $H^*C^{\bZ/\ell}_T(X_\Omega,\bF_{p_\ell})$ is torsion free when $X_\Omega$ is an open ellipsoid.\end{itemize}
\end{iThm}

\subsection{Related works}  
In \cite{large_non_equeezing_GF}, we construct a $\bZ/\ell$-equivariant generating function homology theory for $\ell\geq 2$ using a similar idea, and give a new proof of the contact non-squeezing theorem; we also define a geometric notion {\em translated chains}, which explains the geometry intuition behind of $\bZ/\ell$-equivariant generating function homology. The notion of translated chains explains the same geometric intuition in the Chiu-Tamarkin complex.

Algebraically, in \cite{CyclicZHANG}, we construct an $S^1$-equivariant Chiu-Tamarkin complex which is similar to the Tsygan-Loday-Quillen definition of the cyclic cohomology. In particular, we construct an algebraic $S^1$-action that encodes $C^{\bZ/\ell}_T(U,\bK)$ for all $\ell$ and $T\in \bR_{\geq 0}$. However, if we want to define a contact invariant in this way, we need again $T/\ell\in \bN_0$ for all $\ell$, which is possible only for $n=0$. It explains algebraically why we cannot define an $S^1$-theory for the contact case using the same idea. However, the $\bZ/\ell$-theory here works perfectly for the contact case. The contact capacities we defined above encode sufficient numerical information that is enough for the contact non-squeezing theorem. In this way, we can think of the contact capacities as a numerical approximation of the $S^1$-action.

\subsection{Organization and Conventions of the Paper}We will review preliminary notions of sheaf theory in \autoref{section: reminder}. In \autoref{section: CT complex}, we will present the main constructions, including microlocal kernels, the Chiu-Tamarkin complex, fundamental class and capacities. We will focus on the toric domains in \autoref{section: toric domain}. We would like to exhibit all constructions and computations for the toric domains therein. Subsequently, we will state how our construction works for prequantized contact manifold $T^*X\times S^1_\sigma$ in \autoref{section: contact inv}.

At the end of the introduction, let us introduce some notation.

The natural number set $\bN$ starts from $1$, and $\bN_0$ denotes $\bN\cup\{0\}$. For $n\in \bN$, we denote $[n]=\{1,\dots,n\}$. For any $\ell\in \bN_{\geq 2}$, we denote the minimal prime factor of $\ell$ by $p_\ell$.

We use subscripts to represent elements in sets. For example, to emphasize $a\in A$, we use the notation $A_a$. For the Cartesian product $A^n$, we define $\delta_{A^n}:A\rightarrow A^n$ to be the diagonal map and its image is denoted by $\Delta_{A^n}$ as well.

Projection maps are always denoted by $\pi$, with a subscript that encodes the fiber of the projection. For example, if there are two sets $X_x$ and $Y_y$, two projections are
$$\pi_Y=\pi_y: X_x\times Y_y \rightarrow X_x,\quad \pi_X=\pi_x: X_x\times Y_y \rightarrow Y_y.$$

If we have a trivial vector bundle $X\times V_v$, its summation map is \[\text{id}_X\times s_V^n =\text{id}_X\times s_v^n : X \times V^n \rightarrow X \times V,\, (x,v_1,\dots,v_n) \mapsto (x,v_1+\cdots +v_n ).\] 
In all cases, we will ignore $\text{id}_X$ and only use $s_V^n=s_v^n$ for simplicity.

For a manifold $X$, we always use $\bq\in X$ to represent both the points and the local coordinates of $X$. Correspondingly, the canonical Darboux coordinate of $T^*X$ will be denoted by $(\bq,\bp)$. Vector spaces that are considered as parameter spaces are an exception. For example, $\bR_t$, its dual coordinate is denoted by $\tau \in (\bR_t)^*=\bR_\tau$.

For a manifold $X$, the $1$-jet space is $J^1X=T^*X\times \bR_t$, which is a contact manifold equipped with the contact form $\alpha=dt+\bp d\bq$. The symplectization of $J^1X$ is identified with $T^*X\times T^*_{\tau>0}\bR_t=T^*X\times \bR_t\times \bR_{\tau>0}$, equipped with the symplectic form $\omega=d\bp\wedge d\bq+d\tau\wedge dt$. The symplectic reduction of $T^*X\times T^*_{\tau>0}\bR_t$ with respect to the hypersurface $\{\tau=1\}$ is denoted by $\rho$, which is identified with 
\begin{equation}
  \rho:T^*X\times T^*_{\tau>0}\bR_t \rightarrow T^*X,\, (\bq,\bp,t,\tau) \mapsto (\bq,\bp/\tau).\label{Tamarkin conemap}  
\end{equation}
We call it the Tamarkin's cone map. The map $\rho$ factors through the symplectization map $q$ tautologically:
\begin{center}
    \begin{tikzcd}
T^*X\times T^*_{\tau>0}\bR_t \arrow[r,"q"] \arrow[rr, "\rho", bend right] & J^1X \arrow[r] & T^*X.
\end{tikzcd}
\end{center}

In this paper, we equip the prequantized manifold $T^*X\times S^1_\sigma$ with the contact form $\alpha=d\sigma+\bp d\bq$, which is different with the contact form $d\sigma+\frac{1}{2}(\bq d\bp-\bp d\bq)$ we mentioned before. Then the canonical covering map $J^1X \rightarrow T^*X\times S^1_\sigma$ preserves its contact form.

\subsection*{Acknowledgements}This work is part of my PhD thesis at Université Grenoble Alpes. I would like to thank St\'ephane Guillermou for many very helpful discussions, lots of important observations, and advice on the text; Claude Viterbo for pointing out the existence of capacities, and helpful discussions; Vivek Shende, Nicolas Vichery and Jun Zhang for many helpful comments and discussions; Shengfu Chiu, Octav Cornea for the encouragement. Thanks for Joseph Helfer for pointing out some mistakes in early versions of the article. Thanks for the anonymous referees and their helpful comments. I also want to thank the encouragement of Shaoyun Bai, Honghao Gao, Xuelun Hou, Yichen Qin, Jian Wang, Xiaojun Wu, and suggestions for the improvement of the current text. This work is supported by the ANR projects MICROLOCAL (ANR-15CE40-0007-01), COSY (ANR-21-CE40-0002), and the Novo Nordisk Foundation grant NNF20OC0066298.

\section{Reminder on Sheaves and Equivariant Sheaves}\label{section: reminder}
In this section, we review the notions and tools of sheaves we will use. Let $\bK$ be a commutative ring with finite global dimension. In practice, we only interest the case that $\bK$ is a field or $\bK=\bZ$. For a manifold $X$, let us denote $D(X)$ the derived category of complexes of sheaves of $\bK$-modules over $X$. We note that we do not specify the boundedness of complexes we used in general. In most of our applications, the complexes are locally bounded in the sense that their restrictions on relatively compact open sets are bounded. We refer to \cite{KS90} as the main reference of the section. 

\subsection{Microsupport of Sheaves and Functorial Estimate}For a locally closed inclusion $i:Z\subset X$ and $F\in D(X)$ we set
\[F_Z=i_!i^{-1}F, \quad {\textnormal{R}}\Gamma_ZF=i_*i^!F.\]
\begin{Def}[{\cite[Definition 5.1.2]{KS90}}]For $F\in D(X)$ the  microsupport of $F$ is 
\begin{equation*}
SS(F)= \overline{\left\lbrace(\bq,\bp)\in T^*X: \begin{aligned}
    &\text{There is a }C^1\text{-function }f\text{ near }\bq\text{ such that}\\
    &f(\bq)=0\text{, }df(\bq)=\bp\text{ and }\left
({\textnormal{R}}\Gamma_{\{f\geq 0\}}F\right)_\bq\neq 0.
\end{aligned} \right\rbrace}.    
\end{equation*}
 \end{Def}
By definition, $SS(F)$ is a closed subset of $T^*X$, conic with respect to the $\bR_{>0}$-action along fibres. There is a triangulated inequality for the microsupport: for a distinguished triangle $A\rightarrow B\rightarrow C\xrightarrow{+1}$, we have $SS(A)\subset SS(B)\cup SS(C)$. 

\begin{Thm}[{\cite[Theorem 5.4.5(ii)(c)]{KS90}}]\label{microsupported in zero section and local system}For $F\in D(X)$, we have the equivalence:
\[SS(F)\subset 0_X\text{ if and only if } \forall k\in\bZ, \mathcal{H}^k(F)\text{ are local systems.} \]
\end{Thm}
We set $\dot{T}^*X=T^*X\setminus 0_X$, and $\dot{SS}(F)=SS(F)\cap \dot{T}^*X$.

The conicity is an issue since we want to consider general subsets of $T^*X$.
We will use the Tamarkin's cone map $\rho$ of \eqref{Tamarkin conemap}
to resolve the conicity. This is important because most of symplectic geometric problems are non-conic. However, the cone map is only defined when $\tau>0$ and it is helpful to introduce the Legendre microsupport and the sectional microsupport as follows: For sheaves $F\in D(X\times \bR_t)$, we set 
\begin{align}\label{mu supp}
\begin{aligned}
        {\mu}s_L(F)=q\left(SS(F)\cap \{\tau> 0\}\right)\subset J^1X,\\
    {\mu}s(F)=\rho\left(SS(F)\cap \{\tau> 0\}\right)\subset T^*X.
\end{aligned}
\end{align}
A direct consequence is that ${\mu}s_L(F)$ and ${\mu}s(F)$ are not necessarily conic. However, ${\mu}s_L(F)$ and ${\mu}s(F)$ will lose $\tau \leq 0$ information. Usually, we will consider sheaves that satisfy $ SS(F)\subset \{ \tau \geq 0 \}$ and it will often be the case, in practice, that $SS(F)\cap \{ \tau \leq 0 \} \subset 0_{X\times\bR}$. So, \autoref{microsupported in zero section and local system} shows that we will not lose much information.

Let $f: X\rightarrow Y$ be a $C^1$ map of manifolds. Then there is a diagram of cotangent map:
\begin{center}
\begin{tikzcd}
T^*X  & X\times _Y T^*Y \arrow[l, "df^*"'] \arrow[r, "f_\pi"]  & T^*Y \end{tikzcd}
\end{center}
\begin{Def}Let $f: X\rightarrow Y$ be a $C^1$ map of manifolds, and $\Lambda\subset T^*Y$ be a conic subset. One says that $f$ is non-characteristic for $\Lambda$ if for all $(\bq,\bp)\in \Lambda$ and $df^*_\bq(\bp)=0 $, we have $\bp=0.$\end{Def}
Then we list some functorial estimates we need.
\begin{Thm}[{\cite[Theorem 5.4]{KS90}}]\label{functional estimate}Let $f: X\rightarrow Y$ be a $C^1$ map of manifolds, $F\in D(X), G\in D(Y)$. Let $\omega_{X/Y}=f^!\bK_Y$ be the relative dualizing complex. 
\begin{enumerate}[fullwidth]
\item One has
\begin{align*}
&SS(F   \dboxtimes G) \subset SS(F)\times SS(G),\\
&SS(\RHHOM(\pi_X^{-1}F, \pi_Y^{-1} G)) \subset (-SS(F)) \times SS(G).
\end{align*}

\item Assume $f$ is proper on $\supp(F)$, then $SS({\textnormal{R}}f_!F)\subset f_\pi (df^*)^{-1}\left(SS(F) \right)$.

\item Assume $f$ is non-characteristic for $SS(G)$. Then the natural morphism $f^{-1}G \dotimes \omega_{X/Y}  \rightarrow f^!G$ is an isomorphism, and $SS(f^{-1} G) \cup SS(f^{!}G)\subset df^*f^{-1}_\pi \left( SS(G)\right) $.

\item Assume $f$ is a submersion. Then $SS(F) \subset X \times _Y T^*Y$ if and only if $\forall j\in \bZ$, the sheaves $\mathcal{H}^j(F)$ are locally constant on the fibres of $f$.
\end{enumerate} 
\end{Thm}

\begin{Coro} Let $F_1,F_2 \in D(X)$.
\begin{enumerate}[fullwidth]
\item Assume $SS(F_1)\cap (-SS(F_2))\subset 0_X$, then $SS(F_1 \dotimes   F_2) \subset SS(F_1)+SS(F_2).$
\item Assume $SS(F_1)\cap SS(F_2)\subset 0_X$, then $SS(\RHHOM(F_2, F_1) ) \subset (-SS(F_2))+SS(F_1).$
\end{enumerate}
\end{Coro}
The following \autoref{microlocal morse} is called the microlocal Morse lemma. 
\begin{Coro}\label{microlocal morse}For $F\in D(X)$, let $\phi:X\rightarrow \bR$ be a $C^1$-function that is proper on $\supp(F)$. Let $a<b$ in $\bR$ and assume $d\phi(x)\notin SS(F)$ for $a\leq \phi(x) <b$. Then the natural morphisms $\textnormal{R}\Gamma(\{\phi(x)<b\},F)\rightarrow \textnormal{R}\Gamma(\{\phi(x)<a\},F)$ and $\textnormal{R}\Gamma_{\{\phi(x)\geq b\}}(X,F)\rightarrow \textnormal{R}\Gamma_{\{\phi(x)\geq a\}}(X,F)$ are isomorphisms.
\end{Coro}
For the non-proper pushforward, we have
\begin{Thm}[{\cite[[Corollary 3.4]{tamarkin2013}}]{\label{non proper pushforward estimate}}Let $V$ be an $\bR$-vector space, $\pi_V:X\times V\rightarrow X$, and $\pi_V^\#: T^*X\times V\times V^*\rightarrow T^*X\times V^*$ be the corresponding projections, and $i:T^*X\rightarrow T^*X\times V^*$ be the inclusion. Then for $F\in D(X\times V)$, we have
\[SS(\pi_{V!}F),SS(\pi_{V*}F) \subset i^{-1}\overline{\pi_V^\#(SS(F))}.\]

\end{Thm}

\subsection{Convolution and Tamarkin Category}
Let $X_1,X_2,X_3$ be three manifolds. Recall, $\pi_X: X\times Y \rightarrow Y$ is a projection whose fiber is $X$ for arbitrary $Y$.

\begin{Def}For $F\in D(X_1\times X_2\times \bR_{t_1})$, $G\in D(X_2\times X_3\times \bR_{t_2})$. The convolution is defined as 
\[F\underset{X_2}{\star} G\coloneqq {\textnormal{R}}s_{t!}^2{\textnormal{R}}\pi_{X_2!}(\pi_{(X_3,t_2)}^{-1}F\dotimes \pi_{(X_1,t_1)}^{-1}G)\in D(X_1\times X_3\times \bR_t) .\]
In particular, when $X_2$ is a point, we use the notation $F_1\boxstar F_2$ to empathise. If there is no confusion, we can drop the subscript $X_2$ and write $F\star G=F\underset{X_2}{\star}G$ directly.

Similarly, for $F\in D(X_1\times X_2)$, $G\in D(X_2\times X_3)$, the composition is defined as
\[F\underset{X_2}{\circ}  G= F\circ G\coloneqq {\textnormal{R}}\pi_{X_2!}(\pi_{X_3}^{-1}F  \dotimes \pi_{X_1}^{-1}G  )  \in D(X_1\times X_3 ) .\]
\end{Def}
For $0\in \bR$, $F\in D(X\times \bR)$, we have $ \bK_{0}\star F \cong F$. So, the functor $ \bK_{0}\star$ plays the role of the identity functor. Besides, $\star$ and $\circ$ satisfy the following monoidal identities:
\begin{align}\label{monoidal identities}
\begin{aligned}
&(F_1\star F_2)\star F_3 \cong F_1\star (F_2\star F_3),\quad (F_1\circ F_2)\circ F_3 \cong F_1\circ (F_2\circ F_3),\\
& F_1 \star F_2 \cong F_2 \star F_1, \quad F_1 \circ F_2 \cong F_2 \circ F_1,\\
&(F_1\star F_2)\circ F_3 \cong F_1\star (F_2\circ F_3).
\end{aligned}
\end{align}
Here, the commutative identities are induced by the identification $X_1\times X_3 \cong X_3\times X_1$.

\begin{RMK}\label{remark convolution and composition comparision}
In specific cases, convolution could be presented by composition on $X\times \bR_t$. For example, if $F\in D(X^2\times \bR_{t_1})$ and $G\in D(X\times \bR_{t_2})$. Let $m(t,t')=t-t'=t_1$, then we have
\[F\star G\cong (m^{-1}F)\circ G.\]
In fact, taking $t'=t_2$, we can prove the isomorphism by the proper base change and the projection formula since $s_t^2(t_1,t_2)=t$. 

However, convolution involves spaces of lower dimension. Therefore, we prefer to use convolution in this paper. More important, in geometric applications, the factor $\bR_t$ will play the role of action. Then, convolutions are more helpful for us to look at action information.
\end{RMK}

Before going into further discussion, let us review the notion of semi-orthogonal decomposition of a triangulated category. 

Let $\mathcal{T}$ be a triangulated category and $\mathcal{C}$ a thick full triangulated subcategory of $\mathcal{T}$. The left semi-orthogonal of $\mathcal{C}$ is defined by 
\begin{equation}
    ^\perp \mathcal{C}  \coloneqq \{ X\in \mathcal{T} : \HOM_{\mathcal{T}}(X,Y)=0, \forall Y\in \mathcal{C}   \}.\label{definitionofsemiorthcomplement}
\end{equation}

One can show that the following proposition holds, see \cite[Chapter 4 and Exercise 10.15.]{KS2006}.
\begin{Prop}\label{semiorthcomplement} Using the above notation, we have the following three equivalent properties:
\begin{enumerate}[fullwidth]
    \item The inclusion $ \mathcal{C} \rightarrow \mathcal{T}$ admits a left adjoint functor $L:\mathcal{T}\rightarrow \mathcal{C}$.
    \item There is an equivalence $\mathcal{T}/\mathcal{C} \xrightarrow{\cong}      {^\perp}\mathcal{C} $, where $\mathcal{T}/\mathcal{N}$ is the Verdier localization.
    \item There are two functors $P,Q:\mathcal{T} \rightarrow \mathcal{T}$ such that $\forall X\in\mathcal{T}$, we have the distinguished triangle:
    \[P(X)\rightarrow X \rightarrow Q(X) \xrightarrow{+1}\]
    such that $P(X)\in {^\perp} \mathcal{C}$, and $Q(X)\in \mathcal{C}$.
\end{enumerate}
In this situation, we say one of these data gives a left semi-orthogonal decomposition of $\mathcal{T}$. One can verify, if one of the conditions is satisfied, that $P^2\cong P$, and $Q^2\cong Q$. $P,\,Q$ are called a pair of projectors associated to $\mathcal{C}$.
\end{Prop}
Now, let $\mathcal{T}=D(X\times\bR_t)$, and $\mathcal{C}=\left\{F: SS(F)\subset \{\tau \leq 0\}\right\}$. The triangulated inequality of microsupport shows that $\mathcal{C}$ is a thick full triangulated subcategory of $\mathcal{T}$. Tamarkin constructs a pair of projectors associated to $\mathcal{C}$ given by convolution:
\begin{Thm}[\cite{tamarkin2013}]\label{Tamarkin projector}
  The functors $F \mapsto\bK_{[0,\infty)}\star  F $, $F \mapsto \bK_{(0,\infty)}[1]\star F$ on $D(X \times\bR_t)$
and  the excision triangle,
\[ \bK_{[0,\infty)} \rightarrow \bK_{0} \rightarrow  \bK_{(0,\infty)}[1]\xrightarrow{+1},\]
 give a left semi-orthogonal decomposition of $D(X\times\bR_t)$ associated to $\mathcal{C}$. Namely, for $F\in D(X \times\bR_t)$
we have the distinguished triangle
\begin{equation}
\bK_{[0,\infty)}\star  F\rightarrow F \rightarrow \bK_{(0,\infty)}[1]\star F \xrightarrow{+1}, \label{Tamarkinorthgonalcomplement}  
\end{equation}
with $\bK_{[0,\infty)}\star F\in {^\perp} \mathcal{C}$, $\bK_{(0,\infty)}[1]\star  F\in \mathcal{C}$.
\end{Thm}
One can also see \cite[Proposition 4.19]{GS2014} for a proof and some generalizations of the proposition.

\begin{Def}\label{tamarkincategory} We define the Tamarkin category as the following left semi-orthogonal complement:
\[\mathcal{D}(X)={^\perp} \left\{F: SS(F)\subset \{\tau \leq 0\}\right\}\cong D(X\times \bR) /\left\{F: SS(F)\subset \{\tau \leq 0\}\right\}.\]
By \autoref{semiorthcomplement} and \eqref{Tamarkinorthgonalcomplement} , $F\in D(X\times \bR)$ is in $\mathcal{D}(X)$ if and only if 
\begin{equation}\label{canonical rep of objects in  the Tamarkin category}
F \cong \bK_{[0,\infty)}\underset{\pt}\star  F\cong  \bK_{\Delta_{X^2}\times [0,\infty)}\underset{X}\star F.\end{equation}
\end{Def}
Consequently, the convolution functor $\bK_{\Delta_{X^2}\times [0,\infty)}\underset{X}\star$ of the Tamarkin category $\mathcal{D}(X)$ coincides with the identity functor.

For $F\in \mathcal{D}(X)$, one can show $SS(F)\subset \{\tau \geq 0\}$ using microsupport estimates we mentioned last subsection, see \cite[Proposition 4.17]{GS2014}.

To build microlocal kernels, we follow Chiu's construction and use the Fourier-Sato transform, which is a sheaf-theoretic analogue of the Fourier transform. The Fourier-Sato transform defines a functor $D(V) \rightarrow D(V^*)$, where $V$ is a real vector space and $V^*$ is the dual of $V$. One can see \cite[Section 3.7, Section 5.5]{KS90} for more details. We mention that the Fourier-Sato transform gives an equivalence between $\bR_{> 0}$-equivariant sheaves on $V$ and $V^*$. Tamarkin introduced a new version of the Fourier transform on the category $\mathcal{D}(V)$, which also works for non-$\bR_{> 0}$-equivariant sheaves. We call it the Fourier-Sato-Tamarkin transform. For the relation between the different versions of Fourier transforms, we refer to \cite{FTDAgnolo,radonHonghao}.

Let $Leg(V)=\lbrace (z,\zeta,t): t+\langle z,\zeta\rangle\geq 0 \rbrace \subset V\times V^*\times \bR_t$, we consider ${\bK}_{Leg(V)}\in \cD(V\times V^*)$. 
\begin{Def}The Fourier-Sato-Tamarkin transform is defined as the functor
\begin{align*}
    &\text{FT}:  {D}(X\times V_z\times\bR) \rightarrow  \cD(X\times V^*_{\zeta}),\\
    &\text{FT}(F)=\widehat{F}\coloneqq F \underset{V_z}{\star} {\bK}_{Leg(V)}[\dim V].
\end{align*}
One can see that the restriction of $\text{FT}$ on $\cD(X\times V_z)$ is an equivalence of categories in \cite[Theorem 3.5]{tamarkin2013}.

Sometimes, for $F\in D(V_{\zeta}^*)$, we will use the notation
\begin{equation}\label{already set}
 \widetriangle{F}\coloneqq  {\bK}_{Leg(V)}[\dim V]\underset{V_\zeta^*}{\circ} F.\end{equation}
\end{Def}
Geometrically, the set $Leg(V)$ is associated with the Legendre transform between $V$ and $V^*$. The important thing for us is the microsupport estimate under the Fourier-Sato-Tamarkin transform. Combining Theorem 3.5 and Theorem 3.6 (and its proof) of \cite{tamarkin2013}, we have
\begin{Thm}
Let $\varphi\colon J^1 (X\times V) \rightarrow J^1(X\times V^* )$ be the map  $\varphi(\bq,\bp,z,\zeta,t)=(\bq,\bp,\zeta,-z,t-\langle z,\zeta\rangle)$, where we identify $V^{**}$ with $V$ naturally. Then for $F\in \cD(X\times V)$, we have the microsupport relation:
\begin{equation}
  \mu s_L (\widehat{F})=\varphi( \mu s_L (F)   ).  \label{microsupportfouriertransform}
\end{equation}
\end{Thm}
\begin{proof}
The original statement of \cite[Theorem 3.6]{tamarkin2013} claim that $\mu s (\widehat{F})\subset \varphi_0( \mu s (F)   )$, here $\varphi_0(\bq,\bp,z,\zeta)=(\bq,\bp,\zeta,-z)$. However, the proof indicates that the inclusion can be lifted to $J^1 (X\times V)$ and $\varphi$, i.e.:
\[  \mu s_L (\widehat{F})\subset \varphi( \mu s_L (F)   ). \]
Moreover, Theorem 3.5 in loc. cit. shows that the Fourier transform $F\mapsto \widehat{F}$ has an inverse which is given by $G\mapsto \check{G}=G\underset{V_\zeta^*}{\star} \bK_{Leg'(V)}$ where $Leg'(V)=\lbrace (\zeta,z,t): t\geq \langle z,\zeta\rangle \rbrace \subset V^*\times V\times \bR_t$. We also have an estimate
\[  \mu s_L (\check{G})\subset \varphi^{-1}( \mu s_L (G)   ).  \]
Then the equal of \eqref{microsupportfouriertransform} follows by taking $G=\widehat{F}$.
\end{proof}

\subsection{Guillermou-Kashiwara-Schapira Sheaf Quantization}\label{GKSsection}As a sheaf pattern of Hamiltonian action, we introduce the Guillermou-Kashiwara-Schapira (GKS for short) sheaf quantization as a basic tool here, see \cite{GKS2012} for more details. 

Consider $\dot{T}^*Y$ as a symplectic manifold equipped with the Liouville symplectic form and with a $\bR_{>0}$-action by dilation along the cotangent fibers. If $\varphi:I\times\dot{T}^*Y  \rightarrow \dot{T}^*Y$ is a $\bR_{>0}$-equivariant symplectic isotopy, one can show that it must be Hamiltonian with a $\bR_{>0}$-equivariant Hamiltonian function $H$.

Consider its total graph
\begin{equation}
  \Lambda_{{\varphi}}\coloneqq \left\lbrace (z, - {H_z}\circ{\varphi}_z(\bq,\bp)  ,(\bq,-\bp),{\varphi}_z(\bq,\bp) ) : (\bq,\bp) \in \dot{T}^*Y, z\in I \right \rbrace. \label{totalgraph} 
\end{equation}

Then Guillermou, Kashiwara, and Schapira proved the following theorem:
\begin{Thm}[{\cite[Theorem 3.7]{GKS2012}}]\label{GKS} Using the above notation, we have a sheaf $K=K({\varphi}) \in D(I\times Y^2)$ such that

\begin{enumerate*}[fullwidth, itemjoin={\quad}]
\item $\dot{SS}(K)= \Lambda_{{\varphi}}$,
\item $K_0=\bK_{\Delta_{Y^2}}$,
\end{enumerate*}
where $K_{z}=K|_{\{z\}\times Y^2}$. 

If we set $K_z^{-1}=v^{-1}\RHHOM(K_z, \omega_Y\dboxtimes \bK_Y)$, $v(\bq_1,\bq_2)=(\bq_2,\bq_1),\,\bq_1,\bq_2\in Y,\,z\in I $, then
\begin{enumerate}[fullwidth, label={\alph*}$ )$]
\item $\textnormal{supp}(K)\rightrightarrows I \times Y $ are both proper,
\item $K_z\circ K_z^{-1}\cong  K_z^{-1}\circ K_z\cong \bK_{\Delta_{Y^2}}$,
\item $K$ is unique up to a unique isomorphism. 
\end{enumerate}
Consequently, $F\mapsto F\circ K_z$, $D(Y)\rightarrow D(Y)$ is an equivalence of categories for all $z\in I$, whose quasi inverse is $F\circ K_z^{-1}$.\end{Thm}
For $F\in D(Y)$, $z_0\in I$, we have 
\begin{equation}\label{quantization commute with microsupport}
    \dot{SS}(F\circ K_{z_0})=\varphi_{z_0}(\dot{SS}(F)).
\end{equation}
It means that, geometrically, $\circ K_z$ acts as the Hamiltonian isotopy $\varphi$.
\begin{RMK}Our convention for \eqref{totalgraph} is different from \cite{GKS2012}. Specifically, \eqref{totalgraph} is $\Lambda_{{\varphi^{-1}}}$ in loc. cit., and then our $K_z$ should be $K_z^{-1}$ in loc. cit. We choose such a convention for our convenience in adapting Chiu's proof for \autoref{Existence} without causing further consistency problems.
\end{RMK}

Let us describe two situations where we will use the theorem.
\begin{enumerate}[fullwidth]
\item[I)] Let $\varphi: I\times T^*X \rightarrow T^*X$ be a compactly supported Hamiltonian isotopy. For $Y=X\times \bR_t$, one can lift $\varphi$ to $\widehat{\varphi}: I\times\dot{T}^*Y  \rightarrow \dot{T}^*Y $. Specifically, we have the following:
\begin{Prop}[{\cite[Proposition A.6]{GKS2012}}]\label{cone map}Let $\varphi: I\times T^*X  \rightarrow T^*X$ be a compactly supported Hamiltonian isotopy, whose Hamiltonian function is $H\in C^\infty( I\times T^*X )$. 

There is a $\bR_{>0}$-equivariant Hamiltonian isotopy $\widehat{\varphi}:I\times\dot{T}^*Y  \rightarrow  \dot{T}^*Y $ such that:

\begin{enumerate}[fullwidth,label=\emph{\alph*}$ )$]
\item The function $\widehat{H}(z,\bq,t,\bp,\tau)=\tau H(z,\bq,\bp/\tau)$ is a Hamiltonian function of $\widehat{\varphi}$ on $\{\tau\neq 0\}$.
\item The lifting $\widehat{\varphi}$ commutes with both the symplectization and the Tamarkin's cone map.
\item We can take 
\begin{align*}
&\widehat{\varphi}(z,\bq,t,\bp,\tau)=(\tau\cdot\varphi(z,\bq,\bp/\tau), t-S_{H}(z,\bq,\bp/\tau)), &\tau \neq 0,\\
&\widehat{\varphi}(z,\bq,t,\bp,0)=(\bq,\bp, t+v(z),0), &\tau =0,
\end{align*}
where $u\in C^\infty( I\times T^*X ), v\in C^\infty(I) $, and  $S_{H}(z,\bq,\bp)=\int_0^z [\alpha(X_{H_\lambda})-H_\lambda]\circ \varphi^\lambda_{H}(\bq,\bp)d\lambda$ is the symplectic action function.
\end{enumerate}
We call this $\widehat{\varphi}$ or $\widehat{\varphi}_z$ the conification of $\varphi$. 
\end{Prop}
\begin{RMK}We notice that it is easy to lift $\varphi$ to $T^*X\times T^*_{\tau>0}\bR_t$ without the compactly supported assumption, however this is not enough to apply the Guillermou-Kashiwara-Schapira theorem. If we want to lift $\varphi$ to $\dot{T}^*(X\times \bR_t)$, we need the compactly supported condition.
\end{RMK}  
Now, applying \autoref{GKS} to $\widehat{\varphi}$, we obtain a sheaf $K(\widehat{\varphi}) \in D(I\times X^2\times \bR^2_t)$. 

In our later application, we prefer to use only one $t$-variable, and to use convolution. This is possible as follows. Consider $m(t_1,t_2)=t_2-t_1$, then \cite[Corollary 2.3.2]{guillermou2019sheaves} shows that there is a unique $\mathcal{K}(\widehat{\varphi}) \in D(I\times X^2 \times \bR_t)$ such that $K(\widehat{\varphi})\cong m^{-1}\mathcal{K}(\widehat{\varphi})$, and $\mathcal{K}(\widehat{\varphi})\cong{\textnormal{R}}m_!K(\widehat{\varphi})$.
Then we can take $\mathcal{K}(\widehat{\varphi})$ as the sheaf quantization of $\varphi$. One can show that, for $F\in D(X\times \bR)$, we have $F\circ K(\widehat{\varphi}_z)\cong F\star \mathcal{K}(\widehat{\varphi}_z) $, see \autoref{remark convolution and composition comparision}. 

By the commutativity of the lifting with symplectization, we have the following estimates for the Legendrian microsupport and sectional microsupport of $\mathcal{K}(\widehat{\varphi})$:
\begin{align}\label{microsupportofsheafquantization}
\begin{aligned}
    {\mu}s_L(\mathcal{K}(\widehat{\varphi})) &\subset \lbrace (z,- H(\bq,\bp), \bq, -\bp,\varphi_z(\bq,\bp) ,-S_{H}(z,\bq,\bp) ):(z,\bq,\bp)\in  I\times T^*X\rbrace, \\
   \mu s(\mathcal{K}(\widehat{\varphi}))&\subset \lbrace (z,- H(\bq,\bp), \bq, -\bp,\varphi_z(\bq,\bp) ):(z,\bq,\bp)\in  I\times T^*X \rbrace.       
\end{aligned}
\end{align}
From the point of view of \eqref{quantization commute with microsupport}, for $F\in D(X\times \bR)$, we have
\begin{equation}\label{microlocally gks act as maps}
    \mu s(F\star \mathcal{K}(\widehat{\varphi}_z))=\mu s(F\circ{K}( \widehat{\varphi}_z))=\varphi_z(\mu s(F)).
\end{equation}
\item[II)] Let $\varphi: I\times T^*X\times S^1 \rightarrow T^*X\times S^1$ be a contact isotopy of $T^*X\times S^1$ with a contact Hamiltonian $H\in C^\infty(I\times T^*X\times S^1 )$. One can lift $\varphi$ to a $\bZ$-equivariant contact isotopy ${\varphi'}$ of $J^1(X)=T^*X\times \bR_t$, where $\bZ$ acts by shifting $t$. Here, by $\bZ$-equivariant, we mean that $J^1(\tnT_{k}){\varphi'}={\varphi'}J^1(\tnT_{k})$ for $k\in \bZ$, where $J^1(\tnT_{k})(\bq,\bp,t)=(\bq,\bp,t+k)$. 

\begin{RMK}\label{contact isotopy quantization rmk}
In the symplectic case the Hamiltonian $H$ does not depend on $t$, and does commute with $\tnT'_c$ for all real number $c$. In the contact case, $\varphi'$ commute with $\tnT'_c$ only when for $c=k\in \bZ$.
\end{RMK}
Then it is easy to lift $\varphi'$ to the symplectization, $T^*X\times T^*_{\tau>0}\bR_t$, of $J^1(X)$ to a $ \bZ\times \bR_{>0}$ equivariant Hamiltonian isotopy $\widehat{\varphi'}: I\times T^*X\times T^*_{\tau>0}\bR_t\rightarrow T^*X\times T^*_{\tau>0}\bR_t$. Here, by $\bZ$-equivariance, we mean that $d\tnT^*_{k}\widehat{\varphi'}=\widehat{\varphi'}d\tnT^*_{k}$ for $k\in\bZ$, where $d\tnT^*_{k}(\bq,\bp,t,\tau)=(\bq,\bp,t+k,\tau)$ is the cotangent map of the shifting map $\tnT_k(\bq,t)=(\bq,t+k)$.

Similarly to the symplectic case, the compactly supported condition is necessary to extend $\widehat{\varphi'}$ to whole $\dot{T}^*(X\times \bR_t)$. 

In this case, we still take the sheaf quantization ${K}={K}(\widehat{\varphi'})\in D(I\times X^2 \times \bR^2)$ of $\widehat{\varphi'}$ as sheaf quantization of $\varphi$. However now, since the contact Hamiltonian $H(\bq,\bp,t)$ will depend on the variable $t$, ${K}={K}(\widehat{\varphi'})$ is not pulled back from $D(I\times X^2 \times \bR)$ by $m$. So, we will work with compositions rather than convolutions.

The $\bZ$-equivariance is inherited by the sheaf ${K}(\widehat{\varphi'})$. Precisely, it means that \begin{equation}\label{Z-equi for contact isotopy}
{K}(\widehat{\varphi'})\circ \bK_{\Delta_{X^2}\times \{(t,t+k):t\in \bR\}}\cong \bK_{\Delta_{X^2}\times \{(t,t+k):t\in \bR\}}\circ {K}(\widehat{\varphi'}).    
\end{equation} This is due to $\bK_{\Delta_{X^2}\times \{(t,t+k):t\in \bR\}}=\bK_{\Gamma_{\tnT_k}}$ quantizes $d\tnT^*_{k}$, then we apply the uniqueness part of \autoref{GKS} to $\widehat{\varphi'}=d(\tnT^{-1}_{k})^*\widehat{\varphi'}d\tnT^*_{k}=d\tnT^*_{-k}\widehat{\varphi'}d\tnT^*_{k}$ to obtain the isomorphism \eqref{Z-equi for contact isotopy}.
\end{enumerate}

\subsection{Equivariant Sheaves}
Here, we review basic notions of equivariant sheaves. We refer to \cite{BernsteinLunts} for all details about the general theory of equivariant sheaves and equivariant derived categories. Suppose $G$ is a compact Lie group. For a manifold $X$ with a $G$ action $\rho: G\times X \rightarrow X$, a $G$-equivariant sheaf is a pair $(F, \psi)$ where $F\in Sh(X)$ and $\psi: \rho^{-1}F \cong \pi^{-1}_G F$ is an isomorphism of sheaves satisfying the cocycle conditions:
\[d_0^{-1}\psi \circ d_2^{-1}\psi=d_1^{-1}\psi, \quad s_0^{-1}\psi=\text{Id}_F,\]
where
\begin{align*}
  d_0(g,h,x)=(h,g^{-1}x),\quad
  d_1(g,h,x)=(gh,x),\quad
  d_2(g,h,x)=(g,x),\quad
  s_0(x)=(e,x).
\end{align*}
A sheaf morphism between two $G$-equivariant sheaves is equivariant if it commutes with the $\psi$'s.
We set $Sh_G(X)$ be the category of $G$-equivariant sheaves.
For example, when $X=\pt$, $Sh_G(X)\simeq 
{\bK}[G]-\Mod$, the category of all $G$-modules. The category of $G$-equivariant sheaves $Sh_G(X)$ is abelian. Moreover, Grothendieck proved in \cite{Tohoku} that when $G$ is finite, $Sh_G(X)$ admits enough injective objects. Therefore, the derived category $D(Sh_G(X))$ makes sense for finite groups, which is treated as a naive version of equivariant derived category of sheaves. 

For general topological groups, the naive version is not good as our expectation. A basic difference is that $\textnormal{Ext}^*_{D(Sh_G(X))}(\bK_X,\bK_X)$ is not isomorphic to the equivariant cohomology of $X$. A more serious problem is how to define 6-operations with correct adjunction properties.

To resolve these problems, we must use the equivariant derived category $D_G(X)$ defined by Burnstein-Lunts, in where the expected isomorphism holds, and the correct 6-operations live. 

For the compact Lie group $G$, there exists a universal bundle $EG$ and a classifying space $BG=G\backslash EG$, which are unique up to homotopy. Now, we have a diagram of topological spaces:
\[X\xleftarrow{p} X\times EG \xrightarrow{q} X\times _G EG.\]
\begin{Def}An object $F\in D_G(X)$ is a triple $F=(F_X, \overline{F},\beta_F)$, where $F_X\in D(X)$, $\overline{F}\in D(X\times_G EG)$, and $\beta_F: p^{-1}F_X\rightarrow q^{-1} \overline{F}$ is an isomorphism in $D(X\times EG)$. A morphism $\alpha: F\rightarrow H$ is a pair $(\alpha_X,\overline{\alpha})$ where $\alpha_X:F_X\rightarrow H_X$, $\overline{\alpha}: \overline{F}\rightarrow\overline{H}$, and a commutative diagram in $D(X\times EG)$:
\begin{center}
\begin{tikzcd}
p^{-1}F_X \arrow[rr, "\beta_F"] \arrow[d, "p^{-1}\alpha_X"'] &  & q^{-1}\overline{F} \arrow[d, "q^{-1}\overline{\alpha}"] \\
p^{-1}H_X \arrow[rr, "\beta_H"]                              &  & q^{-1}\overline{H}.                                     
\end{tikzcd}
\end{center}
\end{Def}
For example, the equivariant constant sheaf is given by $\bK_X^G=(\bK_X,\bK_{X\times _G EG}, \textnormal{Id}_{\bK_{EG}})$. 
The natural functor $D_G(X)\rightarrow D(X\times_G EG)$, $F=(F_X, \overline{F},\beta_F)\mapsto \overline{F}$ is fully faithful.

We have a forgetful functor $For: D_G(X)\rightarrow D(X)$ which is given by
\[F=(F_X, \overline{F},\beta_F)\mapsto F_X.\]
In general, for Lie subgroups $H\subset G$, we have a restriction functor $For: D_G(X)\rightarrow D_H(X)$, and the forgetful functor correspondence to the case $H$ trivial.

The microsupport of equivariant objects can be defined as follows:
\begin{Def}\label{def: equivariant sheaf microsupport}For an object $F=(F_X,\overline{F},\beta)\in D_G(X)$, where $F_X\in D(X)$, we {\em define} the microsupport of $F$ to be $SS(F)\coloneqq SS(F_X)$.
\end{Def}
This definition makes sense since the contractibility of $EG$ and \autoref{functional estimate}-(4).

For a $G$-map $f:X\rightarrow Y$ between smooth manifolds, we define maps induced from $f$ as follows:
\begin{center}
\begin{tikzcd}
X \arrow[d, "f"] & X\times EG \arrow[l, "p"'] \arrow[r, "q"] \arrow[d, "\hat{f}"] & X\times_G EG \arrow[d, "\overline{f}"] \\
Y                & Y\times EG \arrow[l, "p'"] \arrow[r, "q'"']                      & Y\times_G EG.                          
\end{tikzcd}
\end{center}
Then we can define 6-operations. For example, we have $\textnormal{R}f_{*}F=(\textnormal{R}f_{*}F_X,\textnormal{R}\overline{f}_*\bar{F},\textnormal{R}\hat{f}_{*}\beta_F)$.
\begin{Prop}All properties of the 6-operations hold in the equivariant case. The 6-operations commute with the forgetful functor. 
\end{Prop}
\begin{RMK}Since the 6-operations commute with the forgetful functor, we will frequently use the notation of non-equivariant 6-operations to denote the equivariant 6-operations without explicit emphases.
\end{RMK}
In the equivariant derived category, we can obtain the expected isomorphism:
\begin{equation}\label{equivariant cohomology equivariant hom}
\textnormal{Ext}^*_{D_G(X)}(\bK_X,\bK_X) \cong
\textnormal{Ext}^*_{D(X\times_G EG)}(\bK_{X\times_G EG},\bK_{X\times_G EG}) \cong H^*(X\times_G EG,\bK).  
\end{equation}

In particular, when $X=\text{pt}$ is a point, we have 
\begin{equation}\label{equivariant cohomology of point}
\textnormal{Ext}^*_{D_G(\pt)}(\bK,\bK)  \cong H^*(BG,\bK).    
\end{equation}
For example,
\begin{align}\label{equivariant cohomology of a point}
\begin{aligned}
    \textnormal{Ext}^*_{D_{\bZ/\ell}(\pt)}(\bK,\bK)\cong H^*(L^\infty_\ell,\bK)\cong \bK[u,\theta],
    \end{aligned}
\end{align}
where $L^\infty_\ell=S^\infty/(\bZ/\ell)$ is the infinite dimensional lens space, $\bK$ is a finite field of $\text{char}(\bK)|\ell$, $|u|=2$, $|\theta|=1$, and $\theta^2=ku$ ($k=0$ if $\ell$ is odd and $k=\ell/2$ otherwise). The computation can be found in \cite[Example 3E.2, Exercise 3E.1]{hatcher}.

\begin{RMK}In \cite{BernsteinLunts}, the authors use finite-dimensional approximations of $EG$ to define 6-operations. The reason is, classically, the 6-operations and related propositions (especially the proper base change) are demonstrated for finite (cohomological) dimensional locally compact Hausdorff spaces while $X \times_G EG$ is not in this class. However, in the framework of \cite{Properbasechange}, the authors introduce a relative notion called separated locally proper maps, for which a proper base change formula is true. In particular, our $\hat{f}$ and $\overline{f}$ are separated locally proper, and $\hat{f}_!$ and $\overline{f}_!$ have finite cohomological dimension. Consequently, we can provide simpler formula for the equivariant 6-operations, and they also work in the unbounded derived category.
\end{RMK}

For discrete groups $G$, both the naive and advanced versions are equivalent, i.e., $D(Sh_G(X))\simeq D_G(X)$. In particular, $D({\bK}[G]-Mod)\simeq D_G(\text{pt}) $. In practice, we will use them alternatively without mentioning explicitly. As a rule of convenience, we only write a lower subscript $G$ for all possible places to indicate that we are working on some version of equivariant categories.

\section{Projectors, Chiu-Tamarkin complex, and Capacities}\label{section: CT complex}

\subsection{Projectors Associated to Open Sets in \texorpdfstring{$T^*X$}{} }\label{section:UniquenessandExistence}
In this subsection, we are going to study the categories related to sheaves microsupported in an open set $U\subset T^*X$. Next, we will construct kernels of the projectors onto these categories. 

For a {\em closed} subset $Z \subset T^*X$ we define $\mathcal{D}_Z(X)$ as the full subcategory of $\mathcal{D}(X)$ consisting of the sheaves satisfying $\mu s(F) \subset Z$. For an {\em open} subset $U \subset T^*X$  we define  $\mathcal{D}_{U}(X)$ to be the left semi-orthogonal complement of $\mathcal{D}_{T^*X\setminus U}(X)$ in  $\mathcal{D}(X)$, i.e., $\mathcal{D}_{U}(X)={^\perp}\mathcal{D}_{T^*X\setminus U}(X)$.

Now we have a diagram of inclusions 
\begin{equation}
 \mathcal{D}_{T^*X\setminus U}(X) 
\hookrightarrow \mathcal{D}(X) \hookleftarrow \mathcal{D}_{U}(X) \label{inclusion diagram of categories supported in U}  
\end{equation}

Following Tamarkin and Chiu, we are looking for convolution kernels that represent microlocal projector functors and give the corresponding semi-orthogonal decomposition.
\begin{Def}\label{defadmissibledomains}We say $U$ is {\em $\bK$-admissible} if there is a distinguished triangle 
\[ P_U\rightarrow {\bK}_{\Delta_{X^2} \times[0,\infty)} \rightarrow Q_U \xrightarrow{+1},\]
in $\cD(X^2)$, such that the convolution functor $\star P_U $ is right adjoint to $\mathcal{D}_{U}(X) 
\hookrightarrow \mathcal{D}(X) $ and $\star Q_U$
is left adjoint to  $\mathcal{D}_{Z}(X) 
\hookrightarrow \mathcal{D}(X) $, i.e., 
\[\mathcal{D}_{Z}(X)  \xleftarrow{\star Q_U}\mathcal{D}(X) \xrightarrow{\star P_U}\mathcal{D}_{U}(X),    \]
are two microlocal projectors.

Such a pair of sheaves $(P_U,Q_U)$ together with the distinguished triangle give an orthogonal decomposition of $\mathcal{D}(X)$ by \autoref{semiorthcomplement}. We call the pair $(P_U,Q_U)$ {\em microlocal kernels} associated with $U$, and the distinguished triangle as the defining triangle of $U$.

We say $U$ is {\em admissible} if $U$ is $\bZ$-admissible.
\end{Def}

\begin{RMK}\begin{enumerate}[fullwidth]\label{kernel remarks}
    \item We define the $\bK$-admissibility of $U$ at the beginning. However the coefficient dependence seems redundant because our existence results in the following work for all $\bK$, especially for $\bK=\bZ$. Moreover, one can show that if $U$ is admissible, then $U$ is $\bK$-admissible for all $\bK$ (by taking the tensor product $\bK\dotimes_\bZ $ with kernels and then use the uniqueness below). From this point of view, we do not emphasize the coefficient ring $\bK$ for the kernels $(P_U,Q_U)$. We will see later that the $\bK$ does affect the computation of the Chiu-Tamarkin complex.
    \item The adjoint functors of the inclusion functor $\mathcal{D}_{T^*X\setminus U}(X) 
\hookrightarrow \mathcal{D}(X) $ is also studied in \cite{kuo2021wrapped}. The author constructs the left and right adjoint of the inclusion functor, which are called infinite wrapping functors. Same with our existence results below (e.g. \autoref{opensetsareadmissible}), the author also use the Guillermou-Kashiwara-Schapira sheaf quantization as a fundamental tool. Our point here is the existence of the kernel $P_U$.
\end{enumerate}\end{RMK}

In the following, we will present the functorial property, uniqueness of kernels, and existence of kernels for bounded open sets. Let us start with some basic facts.
\begin{Lemma}\label{functorial lemma}Suppose $U_1 \subset U_2 $ is an inclusion between $\bK$-admissible open subsets in $T^*X$ and their defining triangles are
    \[ P_{U_i}\xrightarrow{a_i} {\bK}_{\Delta_{X^2} \times[0,\infty)} \xrightarrow{b_i} Q_{U_i} \xrightarrow{+1}, \quad i=1,2.\]
    
\begin{enumerate}[fullwidth]\item We have $Q_{U_2}\star P_{U_1}\cong 0$, and the natural morphism
\[a_2\star P_{U_1} =[P_{U_2} \star P_{U_1}\rightarrow P_{U_1}],\]
is an isomorphism. In particular, we have $P_{U}\star P_{U}\cong P_{U}$ and $Q_{U}\star P_{U}\cong 0$ for any admissible open set $U$.
\item  For any admissible open set $U$ and for all $F,G \in D(X^2 \times \bR)$, we have the isomorphism:\[\HOM_{D(X^2\times \bR)}(F\star P_{U} ,G\star P_{U})\rightarrow  \HOM_{D(X^2\times \bR)}(F\star P_{U},G).\]
\item We have $\RHOM(P_{U_1},Q_{U_2})\cong 0$ and
\begin{equation}\label{non-equivariant CT and yoneda form}
  \RHOM(P_{U_1},a_2):\,\RHOM(P_{U_1},P_{U_2}) \cong\RHOM(P_{U_1},{\bK}_{\Delta_{X^2} \times[0,\infty)}) .\end{equation}
    \end{enumerate}
\end{Lemma}
\begin{proof}(1) follows from the definition. For (2), consider the functor $\widetilde{\mathcal{P}}(F)=F\star P_U : D(X^2\times\bR)\rightarrow D(X^2\times\bR)$, which is a projector on $D(X^2\times\bR)^{op}$ in the sense of \cite[Definition 4.1.1]{KS2006}. Notice that, the functor $\widetilde{\mathcal{P}}$ has the same formula as the microlocal projector but they have different domains. Then (2) follows from Proposition 4.1.3 in loc.cit.. For the vanishing of (3), we take $U=U_1$, $F={\bK}_{\Delta_{X^2} \times\{0\}}$, and $G=Q_{U_2}[d]$ for all $d\in \bZ$ in (2). Next, applying $\RHOM(P_{U_1},-)$ to the defining triangle of $U_2$, we have that $\RHOM(P_{U_1},a_2)$ is an isomorphism.\end{proof}
The functorial property of microlocal kernels is proven in \cite[Theorem 4.7(2)]{chiu2017} for the contact case, and the uniqueness appears in \cite[Section 4.6]{zhang2020quantitative} for the symplectic case. Here, we prove a strong form of the functorial property of kernels, which ensures that the defining triangle is also functorial and unique. 
\begin{Prop}\label{functorial}For any inclusion $U_1 \subset U_2 \subset T^*X$ between $\bK$-admissible open subsets and their defining triangles
    \[ P_{U_i}\xrightarrow{a_i} {\bK}_{\Delta_{X^2} \times[0,\infty)} \xrightarrow{b_i} Q_{U_i} \xrightarrow{+1},\quad i=1,2,\]we have a morphism between the defining triangles:
\begin{center}
\begin{tikzcd}
P_{U_1} \arrow[d, "a"] \arrow[r, "a_1"] & {{\bK}_{\Delta_{X^2}\times[0,\infty)}} \arrow[r, "b_1"] \arrow[d, "\id"] & Q_{U_1} \arrow[d, "b"] \arrow[r, "+1"] & {} \\
P_{U_2} \arrow[r, "a_2"]                & {{\bK}_{\Delta_{X^2}\times[0,\infty)}} \arrow[r, "b_2"]                  & Q_{U_2} \arrow[r, "+1"]                &  { .}
\end{tikzcd}
\end{center}
These morphisms $a,b$ are natural with respect to inclusions of admissible open sets. In particular, when $U_1=U_2$ (while $P_{U_1}$ and $P_{U_2}$ are a priori not the same), the morphism of the defining triangles is an isomorphism of distinguished triangles. \end{Prop}
\begin{proof}The construction of $a,b$ can be found in \cite[Theorem 4.7(2)]{chiu2017}. He verifies that $a_1=a_2a$ and $b_2=b b_1$ and $a,b$ are natural with respect to inclusions. However a priori, $a,b$ do not give a morphism of distinguished triangle. 

So, we consider the following morphism of distinguished triangles constructed by TR3:
\begin{center}
    \begin{tikzcd}
P_{U_1} \arrow[d, "a"] \arrow[r, "a_1"] & {{\bK}_{\Delta_{X^2}\times[0,\infty)}} \arrow[r, "b_1"] \arrow[d, "\id"] & Q_{U_1} \arrow[d, "\psi"] \arrow[r, "+1"] & {} \\
P_{U_2} \arrow[r, "a_2"]                & {{\bK}_{\Delta_{X^2}\times[0,\infty)}} \arrow[r, "b_2"]                  & Q_{U_2} \arrow[r, "+1"]                   & {.}
\end{tikzcd}
\end{center}

Then we have $\psi b_1=b_2$. As a corollary of \autoref{functorial lemma}-(3), we have the isomorphism:
\[\HOM(Q_{U_1},Q_{U_2} )\xrightarrow{-\circ b_1} \HOM({\bK}_{\Delta_{X^2} \times[0,\infty)},Q_{U_2}).\]
Finally, one conclude that $b=\psi$ since their image under the isomorphism $-\circ b_1$ is $b_2$.
\end{proof}
By the results of \cite{GKS2012} recalled in \autoref{GKSsection}, there exists $\mathcal{K}\in D(X^2\times
  \bR_t)$ (taking $\mathcal{K}=\mathcal{K}(\widehat{\varphi})_z$) such that the convolution functor \[D(X\times \bR_t) \to D(X\times \bR_t),\quad F \mapsto
   F\star \cK,\]is an equivalence of categories and $\mu s_L(F \star \cK) =\varphi_z(\mu
  s_L(F))$.  Since $\widehat{\varphi}_z$ preserves $\tau$ of $T^*(X\times
  \bR_t)$, this functor descends to the quotient $\cD(X)$ and gives an auto-equivalence. 
\begin{Coro}{{\cite[Proposition 4.5]{chiu2017}}}\label{invariance of kernel}Let $U\subset T^*X$ be an admissible open set, and $\varphi$ be a compactly supported Hamiltonian isotopy $\varphi:I\times T^*X  \rightarrow T^*X $. Then $\varphi_z(U)$ is admissible for all $z\in I$ and we have an isomorphism $P_{\varphi_z(U)} \cong \cK^{-1} \star P_U \star \cK$.

In particular, for $U=T^*X$, the isomorphism is realized by: \[ \bK_{\Delta_{X^2}\times [0,\infty)}\cong \cK^{-1}\star \cK\star \bK_{\Delta_{X^2}\times [0,\infty)} \cong \cK^{-1}\star\bK_{\Delta_{X^2}\times [0,\infty)}\star \cK .\]
\end{Coro}

On the other hand, we can also consider the rescaling of the size of $U$ (or the rescaling of the symplectic form). To be simpler, assume that $X=\bR^d$ is a $\bR$-vector space. Consider the map $R:X\times \bR \rightarrow X\times \bR$, $(\bq,t)\mapsto (\bq/r,t/r^2)$ for $r>0$. Then \autoref{functional estimate}-(2) shows that $SS(R_!F)=\{(r\bq,\bp/r,r^2t,\tau/r^2): (\bq,\bp,t,\tau)\in SS(F)\}$. Therefore, if $\mu s(F)\subset T^*X\setminus U$, we have $\mu s(R_!F)\subset T^*X\setminus rU$ since $\rho(r\bq,\bp/r,r^2t,\tau/r^2)=(r\bq, (\bp/r)/(\tau/r^2))=(r\bq,r(\bp/\tau))$. We can directly verify that $R_!:D(X^2\times
  \bR_t) \rightarrow D(X^2\times
  \bR_t)$ induces an equivalence $R_!:\cD_{U}(X)\xrightarrow{{\cong}}\cD_{rU}(X)$ whose quasi inverse is $R^{-1}_!$. Then we have an isomorphism of functors $\star P_{rU}\cong R^{-1}_!  (-\star P_U)R_!$. Moreover, we have
  \begin{Coro}\label{conformal invariance of kernel}Let $U\subset T^*\bR^d$ be an admissible open set, and $r>0$. Then $rU$ is admissible and we have an isomorphism $P_{rU} \cong  \bar{R}_!P_U$ where $\bar{R}(\bq,\bq',t)=(\bq/r,\bq'/r,t/r^2)$.      \end{Coro}
\begin{proof}Obviously, we have a distinguished triangle
\[ \bar{R}_!P_U\rightarrow {\bK}_{\Delta_{X^2} \times[0,\infty)} \rightarrow \bar{R}_!Q_U \xrightarrow{+1}.\]
Let us show that it is a defining triangle of $rU$. If so, the result follows from the uniqueness.

Take $\mathcal{R}=\bK_{\Gamma_R}$ and $\mathcal{R}^{-1}=\bK_{\Gamma_{R^{-1}}}$ where $\Gamma_g$ denotes the graph of $g$. Then we have an isomorphism of functors $\circ \mathcal{R}  \cong R_!$. Be careful that the convolution kernels $\mathcal{R},\mathcal{R}^{-1}$ are objects in $D((X\times \bR)^2)$, and they cannot descent to $\cD(X^2)$. So, we are necessary to consider compositions rather than convolutions. Based on \autoref{remark convolution and composition comparision}, we take $\mathscr{P}_U =m^{-1}P_U$ and $\mathscr{Q}_U =m^{-1}Q_U$ to guarantee that we have isomorphisms of functors: $\circ \mathscr{P}_U\cong \star P_U$ and $\circ \mathscr{Q}_U\cong \star Q_U$. In the following, we only discuss $\mathscr{P}_U$, but all the rest are true for $\mathscr{Q}_U$.

The previous discussion shows that $\circ \mathscr{P}_{rU}\cong R^{-1}_!  (-\circ \mathscr{P}_U) R_! \cong -\circ (\mathcal{R}^{-1}   \circ \mathscr{P}_U\circ \mathcal{R})   $ as functors. Noticed that $\circ$ here only represent composition of sheaves. On the other hand, we have isomorphisms of composition kernels:
$\mathcal{R}^{-1}   \circ \mathscr{P}_U\circ \mathcal{R} \cong (R\times R)_!\mathscr{P}_U$. It follows from the general fact that $A\circ \bK_{\Gamma_g}\cong (\id\times g)_!A$ and $\bK_{\Gamma_{g^{-1}}}\circ A\cong (g\times \id)_!A$ for any $A$ and $g$. 

Then we conclude by the base change isomorphism $(R\times R)_!m^{-1}\cong m^{-1}\bar{R}_!$.
\end{proof}
Now, let us study the existence of admissible open sets $U\subset T^*X$. In general, we can take a smooth Hamiltonian function $H$ such that $U=\{H<1\}$ (\cite[Theorem 2.29]{JohnLee}). Our tools to construct kernels are sheaf quantizations and the Fourier-Sato-Tamarkin transform. 

For our later application for toric domains, let us state our idea in a more general form. Suppose there is a Hamiltonian $\bR^m_{z}$-action  on $T^*X$, i.e., a symplectic action $\varphi:\bR^m_{z}\times T^*X\rightarrow T^*X$ with a moment map $\mu: T^*X\rightarrow (\bR^m_{z} )^*=\bR^m_{\zeta}$. Let us consider $U$ of the form $\mu^{-1}(\Omega)$, where $\Omega\subset \bR^m_{\zeta}$.
We assume that there exists a {\em sheaf quantization} $\cK\in \cD(\bR_z^m\times X^2)$ associated with the Hamiltonian action in the sense:
\begin{align}\label{sheafquantization}
\begin{aligned}
 &\mathcal{K}|_{z=0} \cong  {\bK}_{\Delta_{X^2 }\times [0,\infty)},\\
&\mu s(\mathcal{K}) \subset \{ (z,- \mu(\bq,\bp), \bq,-\bp,\varphi_z(\bq,\bp) ):(z,\bq,\bp)\in  \bR^m_{z}\times T^*X\}.   
\end{aligned}
\end{align}
\begin{RMK}
One can see that when $m=1$, it is exactly the single Hamiltonian situation $\mu=H$. Now, if we additionally assume that $\mu=H$ is compactly supported up to constant, then we can take $\mathcal{K}=\mathcal{K}(\widehat{\varphi^H})\star \bK_{[0,\infty)}$ as the sheaf quantization. Here, $\mathcal{K}(\widehat{\varphi^H})$ is introduced in \autoref{GKSsection}.

In fact, as \cite[Remark 3.9]{GKS2012} discussed, if we have an Hamiltonian action of $\bR^m$ with $m>1$, the sheaf quantization exists if the action is compactly supported.
\end{RMK}
The sheaf $\mathcal{K}$ lives over the $z$-variable space. Intuitively, if we want to restrict the microsupport of some sheaves into $ \Omega \subset \bR^m_\zeta$, we need a sheaf transform to interchange $z$ and $\zeta$ variables, which are dual to each other. Then we cut-off the support of the resulting sheaf in some way. This operation is classical in mechanics and thermodynamics, i.e. the Legendre transform. We have noticed that the sheaf correspondence of the Legendre transform is the Fourier-Sato(-Tamarkin) transform. Consequently, let us apply the Fourier-Sato-Tamarkin transform to the $z$-variable, i.e. $\widehat{\mathcal{K}}=\mathcal{K}\star \bK_{Leg(\bR^m_{\zeta})}[m]\in \mathcal{D}(\bR^m_{\zeta} \times X^2)$. So by \eqref{sheafquantization} and \eqref{microsupportfouriertransform}, we have 
\begin{equation}\label{fourier transform of sheaf quantization}
\mu s(\widehat{\mathcal{K}})\subset \{ (\mu(\bq,\bp),z, \bq, -\bp,\varphi_z(\bq, \bp)):(z,\bq,\bp)\in  \bR^m_{z}\times T^*X\}.    
\end{equation}
Then, we construct the kernels in the following way. Consider the excision triangle in $D(\bR^m_{\zeta})$:\[{\bK}_{\Omega} \rightarrow {\bK}_{\bR^m_{{\zeta}}} \rightarrow {\bK}_{\bR^m_{{\zeta}}\setminus\Omega}\xrightarrow{+1}.\]
Composing the distinguished
triangle with $\widehat{\mathcal{K}}$, we obtain a distinguished triangle in $\mathcal{D}(X^2)$: 
\[\widehat{\mathcal{K}}\circ {\bK}_{\Omega} \rightarrow \widehat{\mathcal{K}}\circ {\bK}_{\bR^m_{{\zeta}}} \rightarrow \widehat{\mathcal{K}}\circ {\bK}_{\bR^m_{{\zeta}}\setminus\Omega}\xrightarrow{+1} .\]
By the associativity of convolutions and compositions (see \eqref{monoidal identities}), we have $\widehat{\mathcal{K}}\circ F = (\mathcal{K} \star \bK_{Leg(\bR^m_{\zeta})}[m])\circ F\cong\mathcal{K} \star (\bK_{Leg(\bR^m_{\zeta})}[m] \circ F) \cong\mathcal{K} \star \widetriangle{F}$ for $F\in D(\bR^m_{\zeta})$, where $\widetriangle{F}=\bK_{Leg(\bR^m_{\zeta})}[m] \circ F$ in \eqref{already set}.

So as $\bK_{Leg(\bR^m_{\zeta})}[m] \circ \bK_{\bR^m_\zeta}=\bK_{\lbrace z=0,\,t \geq 0\rbrace}$, we have 
\[(\widehat{\mathcal{K}}\circ \bK_{\bR^m_\zeta}) \cong \mathcal{K} \star \bK_{\lbrace z=0,\,t \geq 0\rbrace}\cong \bK_{\Delta_{X^2} \times [0,\infty)},\]
where the last isomorphism comes from \eqref{sheafquantization}, i.e., $\mathcal{K}|_{z=0} \cong {\bK}_{\Delta_{X^2}\times [0,\infty)}$. Therefore, we have the distinguished triangle
\[\widehat{\mathcal{K}}\circ {\bK}_{\Omega} \rightarrow \bK_{\Delta_{X^2} \times [0,\infty)}\rightarrow \widehat{\mathcal{K}}\circ {\bK}_{\bR^m_{{\zeta}}\setminus\Omega}\xrightarrow{+1} .\]
\begin{Prop}\label{Existence}Let $\varphi$ be a  Hamiltonian $\bR^m_{{z}}$-action on $T^*X$ with a moment map $\mu: T^*X\rightarrow \bR^m_{{\zeta}} $. We assume that there exists a sheaf quantization $\mathcal{K}\in D(\bR^m_{{z}} \times X^2\times \bR_t)$ of the Hamiltonian action in the sense of \eqref{sheafquantization}. For an open subset $\Omega \subset  \bR^m_{ {\zeta}}$ such that for all $\zeta\in \Omega$, $\mu^{-1}(\zeta)$ is compact, the open set $U=\mu^{-1}(\Omega)\subset T^*X$ is admissible.

More precisely, the pair of sheaves 
\begin{equation}\label{standard model of kernels}
  P_U\coloneqq \widehat{\mathcal{K}}\circ {\bK}_{\Omega},\,Q_U\coloneqq \widehat{\mathcal{K}}\circ {\bK}_{\bR^m_{ {\zeta}}\setminus\Omega},  
\end{equation}
 and the distinguished triangle
 \begin{equation}\label{dt of admissible open set}
    \widehat{\mathcal{K}}\circ {\bK}_{\Omega} \rightarrow {\bK}_{\Delta_{X^2}\times [0,\infty)} \rightarrow \widehat{\mathcal{K}}\circ {\bK}_{\bR^m_{ {\zeta}}\setminus\Omega}\xrightarrow{+1}, 
 \end{equation}

provide the microlocal kernels of $U$ and the semi-orthogonal decomposition.
\end{Prop}
\begin{proof}
Our construction is a straightforward generalization of Chiu's result \cite[Theorem 3.11]{chiu2017}. One only needs to notice that we consider a Hamiltonian $\bR^m_z$-action more than a single Hamiltonian function, and we replace $(-\infty,R)$ in Chiu's paper by $\Omega$. The properness condition of $\Omega$ is a technical condition that is automatically true in the situation of Chiu. One can check the proof of Chiu to confirm that our condition is enough to ensure the virtue of the semi-orthogonal decomposition without any other modification. Furthermore, Chiu's argument works for $\bZ$. So we obtain not only $\bK$-admissibility but also admissibility.
\end{proof}

\begin{Def}\label{dynamicallyopensets}We say an admissible open set $U$ is \emph{dynamically admissible} if there exists $\varphi$, $\cK$, $\Omega$ that satisfies \autoref{Existence}.
\end{Def}
So, for dynamically admissible sets, \autoref{Existence} provides us with a standard way to construct microlocal kernels. Now, let us state some corollaries. 

\begin{Coro}\label{opensetsareadmissible}Bounded open sets are dynamically admissible.
\end{Coro}
\begin{proof}Let $U\subset T^*X$ be a bounded open set, we have $T^*X \setminus U$ is a closed subset of $T^*X$. Then there exists a smooth function $H: T^*X
\rightarrow [0,1]$ such that $U=\{H<1 \}$ and $T^*X \setminus U =\{H\geq 1\}$. Actually, we take a non negative function $f$ such that $f^{-1}(0)=T^*X\setminus U$, see \cite[Theorem 2.29]{JohnLee}. Then we take $H(x)=1-f(x)$.

Since $U$ is bounded, the subsets $\{H=a\}\subset U$ with $a<1$ are compact. Moreover, $dH$ has compact support. So we can take the GKS quantization $\mathcal{K}(\widehat{\varphi^H})$. Then the result follows from \autoref{Existence} by taking $\Omega=(-\infty,1)$.
\end{proof}

The second corollary here concerns the kernel of products of open sets.

\begin{Coro}\label{product}Suppose we have two bounded open sets $U_i \subset T^*X_i$, with two pairs of kernels
$ (P_{U_i}, Q_{U_i}) ,\, i=1,2.$ Then $U_1\times U_2$ is dynamically admissible and $P_{U_1\times U_2}\cong  P_{U_1}\boxstar P_{U_2}$.
\end{Coro}
\begin{proof}By the assumption, we have two Hamiltonian functions $H_i \in C^\infty(T^*X_i)$ such that $U_i=\{H_i<1\}$ and we associate with them two sheaf quantizations $\mathcal{K}_i$. Then  \[ (P_{U_i}, Q_{U_i})= (\widehat{\mathcal{K}}_i\circ {\bK}_{(-\infty,1)},\widehat{\mathcal{K}}_i\circ {\bK}_{[1,\infty)} ),\quad i=1,2.\]
Now, consider the product Hamiltonian  $\bR_{{z}}^2$-action on $T^*(X_1\times X_2)$ whose moment map is $\mu=(H_1, H_2)$. Then $\mathcal{K}_1\boxstar \mathcal{K}_2$ is a sheaf quantization of the Hamiltonian action in the sense of \eqref{sheafquantization}.

Observe that if we take $\Omega=\{\zeta=(\zeta_1,\zeta_2):\zeta_1<1, \zeta_2<1 \}$, then we have $U_1\times U_2=\mu^{-1}(\Omega)$. Consequently, \autoref{Existence} implies that $U_1\times U_2$ is admissible by the following distinguished triangle
\[\widehat{\mathcal{K}}\circ {\bK}_{\Omega} \rightarrow {\bK}_{\Delta\times \{t\geq 0\}} \rightarrow \widehat{\mathcal{K}}\circ {\bK}_{\bR^2_{{\zeta}}\setminus\Omega}\xrightarrow{+1}.\]
Subsequently, let us compute $\widehat{\mathcal{K}}\circ {\bK}_{\Omega}$.

Recall \[\widehat{\mathcal{K}}\circ {\bK}_{\Omega}\cong  \mathcal{K}\star \widetriangle{ {\bK}_{\Omega}}.
\]
Notice $\Omega$ is an open convex set.
Therefore, $\widetriangle{ {\bK}_{\Omega}}={\bK}_{\{({z}, {\zeta},t): t+{z} \cdot {\zeta}\geq 0\}}[2] \circ {\bK}_{\Omega}$ is the constant sheaf ${\bK}_{\Omega^\circ}$ supported on the polar cone $\Omega^\circ$ of $\Omega$, where 
\[\Omega^\circ =\{({z},t): t+{z} \cdot {\zeta}\geq 0, \forall \zeta \in \Omega \}.\]
In fact, ${\bK}_{\{({z}, {\zeta},t): t+{z} \cdot {\zeta}\geq 0\}}[2] \circ {\bK}_{\Omega}$ is isomorphic to the classical Fourier-Sato transform of the constant sheaf supported on the conification of $\Omega$ (upto a degree shifting depends only on dimension). The conification is a convex cone. The Fourier-Sato transform of a constant sheaf supported on a convex cone is the constant sheaf supported on the polar cone of the given cone. A direct computation shows that the polar cone of the conification of $\Omega$ is exactly $\Omega^\circ$. Then, our computation follows.

In particular, when $\Omega=\{\zeta_1<1, \zeta_2<1 \}$, we have ${\Omega}^\circ=\{({z},t): z=(z_1,z_2),,z_1\leq 0, z_2\leq 0, t\geq -(z_1+z_2)\geq 0\}$. Moreover, ${\bK}_{{\Omega}^\circ} \cong {\textnormal{R}}s_{t!}^2( \bK_{\gamma_1 \times \gamma_2})$, where $\gamma_i = \{(z_i,t) :t \geq -z_i\geq 0\}$.

Now we have 
\begin{align*}
 \widehat{\mathcal{K}}\circ {\bK}_{\Omega}&\cong  \mathcal{K}\star \widetriangle{ {\bK}_{\Omega}} \cong \mathcal{K} \star \bK_{\Omega^\circ}\cong \mathcal{K} \star {\textnormal{R}}s_{t!}^2( \bK_{\gamma_1 \times \gamma_2})\\
 &\cong  (\mathcal{K}_1\boxstar \mathcal{K}_2) \star {\textnormal{R}}s_{t!}^2( \bK_{\gamma_1 \times \gamma_2})\cong  (\mathcal{K}_1 \star {\bK}_{\gamma_1}) \boxstar (\mathcal{K}_2 \star {\bK}_{\gamma_2}).
 \end{align*}
Finally, noticing that $ {\bK}_{\{(z,t) :t \geq -z\geq 0\}}\cong {\bK}_{\{(z, \zeta,t): t+z\zeta\geq 0\}}[1] \circ {\bK}_{(-\infty,1)} $, one can conclude that 
\[ P_{U_1\times U_2}\cong \widehat{\mathcal{K}}\circ {\bK}_{\Omega}  \cong (\mathcal{K}_1 \star {\bK}_{\gamma_1}) \boxstar (\mathcal{K}_2 \star {\bK}_{\gamma_2}) \cong P_{U_1}\boxstar P_{U_2}.\]\end{proof}

\subsection{Chiu-Tamarkin complex}
Let $\bZ/\ell$ be the finite cyclic group of order $\ell\in \bN$ and $X$ be a smooth manifold of dimension $d$. 

Now take an admissible open set $U\subset T^*X$, and let $P_U$ be the kernel associated with $U$. 
The manifold $(X^2\times \bR_t)^{\ell}$ admits a $\bZ/\ell$-action induced by the cyclic permutation of the $\ell$ factors. According to \cite[Section 2.2]{lonergan_2021}, the object $P_U^{\dboxtimes \ell}$ of $D((X^2\times \bR_t)^\ell)$ has a natural lift $St_D(P_U)$ as an object of the equivariant derived category $D_{\bZ/\ell}((X^2\times \bR_t)^\ell)$, which we also denote, due to historically reason, by $P_U^{\dboxtimes \ell}$. Then we have $P_U^{\boxstar \ell}=\tnR s_{t!}^\ell P_U^{\dboxtimes \ell} \in D_{\bZ/\ell}((X^2)^\ell \times \bR_t)$.

Consider the $\bZ/\ell$-equivariant maps
\begin{align*}
   \pi_{\underline{\bq}}&: X^\ell \times \bR\rightarrow \bR,\\
    \Tilde{\Delta}_X&: X^\ell \times \bR\rightarrow X^{2\ell} \times \bR,\\
    \Tilde{\Delta}_X(\bq_1,\dots,\bq_\ell,t)&=(\bq_\ell,\bq_1,\bq_1,\dots,\bq_{\ell-1},\bq_{\ell-1},\bq_\ell,t),\\
    i_T&: \{T\}\hookrightarrow \bR,\quad
\end{align*}
where $\underline{\bq}=(\bq_1,\dots,\bq_\ell)$ and $\Tilde{\Delta}_X$ is a twisted diagonal map of $X$.

There is an adjoint pair $(\alpha_{\ell,T,X},\beta_{\ell,T,X})$:
\begin{center}
    \begin{tikzcd}
F\in D_{\bZ/\ell}((X^2\times \bR_t)^\ell) \arrow[rr, "\alpha_{\ell,T,X}", shift left] &  & D_{\bZ/\ell}(\pt) \ni G, \arrow[ll, "\beta_{\ell,T,X}"]
\end{tikzcd}
\end{center}
defined by:
\begin{align}\label{definition of adjoint loop functor}
\begin{aligned}
    &\alpha_{\ell,T,X}(F)=i_T^{-1}{\textnormal{R}}\pi_{\underline{\bq}!}   \Tilde{\Delta}_X^{-1}{\textnormal{R}}s_{t!}^\ell\left(F\right),\\
    &\beta_{\ell,T,X}(G)=s_t^{\ell!}\Tilde{\Delta}_{X*} \pi_{\underline{\bq}}^!i_{T*} G.
\end{aligned}
\end{align}
Now, we define a functor 
\begin{equation}\label{definition of F}
    F_{\ell,X}={\textnormal{R}}\pi_{\underline{\bq}!}   \Tilde{\Delta}_X^{-1}{\textnormal{R}}s_{t!}^\ell: D_{\bZ/\ell}((X^2\times \bR_t)^\ell) \rightarrow D_{\bZ/\ell}(\bR).
\end{equation}
Then $\alpha_{\ell,T,X}=i_T^{-1} F_{\ell,X}$. 

Similarly, we define $\alpha'_{\ell,T,X},\beta'_{\ell,T,X}, F'_{\ell,T,X}$ by removing $s_{t}^\ell$ in the corresponding definitions. If there is no risk of confusion, forget some of $\ell, T, X$ in subscripts of $\alpha,\beta,F$ for simplicity.
\begin{RMK}
We will use $\alpha_{\ell,T,X}, \beta_{\ell,T,X}$ ($\alpha'_{\ell,T,X}, \beta'_{\ell,T,X}$), and $ F_{\ell,X}$ ($ F'_{\ell,X}$) in the non-equivariant categories. We denote them by the same notation later.
\end{RMK}
\begin{Def}\label{def CT complex}With the notation above, we first define
\[F_{\ell}(U,\bK)\coloneqq F_{\ell,X}(P_U^{\dboxtimes \ell})=F'_{\ell,X}(P_U^{\boxstar \ell})\in D_{\bZ/\ell}(\bR).\]
Then we define an object of $D_{\bZ/\ell}(\text{pt})$ that we call the {\em Chiu-Tamarkin complex} by
\begin{align*}
    C^{\bZ/\ell}_T(U,\bK)&=\RHOM_{\bZ/\ell}\left(\alpha_{\ell,T,X}(P_U^{\dboxtimes \ell}),\bK[-d]  \right)\\
    &=\RHOM_{\bZ/\ell}\left((F_{\ell}(U,\bK))_T, \bK[-d]\right)\\
    &\cong\RHOM_{\bZ/\ell}\left(P_U^{\dboxtimes \ell},\beta_{\ell,T,X}\bK[-d]   \right). 
\end{align*}
We set $A = \text{Ext}_{\bZ/\ell}^*(\bK,\bK)$, which is isomorphic to $H_{\bZ/\ell}^*(B\bZ/\ell.\bK)$ (see \eqref{equivariant cohomology of point}). Then $H^*C^{\bZ/\ell}_T(U,\bK)$ and is a graded module over $A \cong \text{Ext}_{\bZ/\ell}^*(\bK[-d],\bK[-d])$ via the Yoneda product.

When $\ell=1$, i.e. the cyclic group $\bZ/1$ is trivial, we also denote the non-equivariant Chiu-Tamarkin complex $C^{\bZ/1}_T(U,\bK)$ by $C_T(U,\bK)$.
\end{Def}
\begin{RMK}\label{remark def CT}
\begin{enumerate}[fullwidth]
    \item The object $C^{\bZ/\ell}_T(U,\bK)$ is mentioned by Tamarkin in \cite{mc-Tamarkin}, and is defined explicitly by  Chiu in \cite{chiu2017}. Our definition looks slightly different from the definition of Chiu. However, one can check directly that, when $X$ is orientable, $ \beta_{\ell,T,X}\bK[-d] $ is exactly the constant sheaf supported on the twisted diagonal with a degree shift depending only on $\ell$ and $\dim 
X$. So the complex $C^{\bZ/\ell}_T(U,\bK)$ is essentially the same as what Chiu defined. 

\item Recall \autoref{functorial lemma}, we have that, for all $\ell\in \bO$, $P_U^{\star \ell}\cong P_U$ in $\cD(X^2)$. Then we have $F_1(P_U^{\star \ell})\cong F_1(P_U)$. However the definition of convolution shows $F_1(P_U^{\star \ell})\cong F_\ell(P_U^{\dboxtimes \ell})$. Therefore, we obtain an isomorphism, in the non-equivariant derived category,
\begin{equation}\label{cyclicstructure step1}
    F_1(U,\bK)\cong F_\ell(U,\bK).
\end{equation}
So we have $C_{T}(U,\bK)\cong \RHOM\left((F_{\ell}(U,\bK))_T, \bK[-d]\right)$ (in the non-equivariant derived category). In this way, it is clear that $C^{\bZ/\ell}_T(U,\bK)$ is the equivariant generalization of $C_{T}(U,\bK)$.
\end{enumerate}
\end{RMK}

Let us compute an example when $U=T^*X$. 
Recall $P_{T^*X}=\bK_{\Delta_{X^2}\times [0,\infty)}$, so we have
$\Tilde{\Delta}^{-1}\left( P_{T^*X}^{\boxstar \ell}\right)= \bK_{\Delta_{X^{\ell}}\times [0,\infty) }.$ Then we obtain
\[F_{\ell}(T^*X,\bK)= {\textnormal{R}}\pi_{\underline{\bq}!}(\bK_{\Delta_{X^{\ell}}\times [0,\infty) })=E_{[0,\infty)}, \]
where $E=R\Gamma_c(\Delta_{X^{\ell}},\bK)$, $E_{[0,\infty)}$ is the constant sheaf supported on $[0,\infty)$
and ${\bZ/\ell}$ acts on $E={\textnormal{R}}\Gamma_c(\Delta_{X^{\ell}},\bK)\cong {\textnormal{R}}\Gamma_c(X,\bK)$ trivially.
Since ${\bZ/\ell}$ acts on $E$ trivially, we have, by Poincar\'e-Verdier duality,
\begin{equation}\label{F of the cotangent bundle}
  \begin{split}
    &C^{\bZ/\ell}_T(T^*X,\bK)\cong  \RHOM_{\bZ/\ell}(E,\bK[-d])   
    \cong  \RHOM_{\bZ/\ell}(\bK,\bK)\dotimes\RHOM(E,\bK[-d])  \\ \cong & \RHOM_{\bZ/\ell}(\bK,\bK)\dotimes {\textnormal{R}}\Gamma(X,\omega_{X}[d])
  \cong  \RHOM_{\bZ/\ell}(\bK,\bK)\dotimes {\textnormal{R}}\Gamma_X(T^*X,\bK)[d]. \\
\end{split}
\end{equation}
Finally, for $T\geq 0$ and a field $\bK$, we have 
\begin{align}\label{C of the cotangent bundle}
\begin{aligned}
&H^*C^{\bZ/\ell}_T(T^*X,\bK)\cong   A\otimes H^{BM}_{d-*}(X,\bK) \cong   A\otimes H^{*+d}_X(T^*X,\bK),
\end{aligned}
\end{align}
where $H^{BM}_{*}=H^{-*}(X,\omega_X)$ stands for the Borel-Moore homology of $X$.

One of the most important theorems about the Chiu-Tamarkin complex is 
\begin{Thm}[{{Theorem 4.7 of} \cite{chiu2017}}]\label{invariance1} Let $U,U_1,U_2$ be admissible open sets and let $U_1\xhookrightarrow{i} U_2$ be an inclusion. Then one has, for $T\geq 0$, 
\begin{enumerate}[fullwidth]
    \item There is a morphism $C^{\bZ/\ell}_T(U_2,\bK) \xrightarrow{i^*} C^{\bZ/\ell}_T(U_1,\bK)$, which is functorial with respect to inclusions of admissible open sets.
    \item For a compactly supported Hamiltonian isotopy $\varphi:I\times T^*X  \rightarrow T^*X $, then there is an isomorphism, in the equivariant category, $\Phi^{\bZ/\ell}_{z,T }:C^{\bZ/\ell}_T(U,\bK) \xrightarrow{\cong} C^{\bZ/\ell}_T\left(\varphi_z(U),\bK\right)$, for all $z\in I$. The isomorphism $\Phi^{\bZ/\ell}_{z,T}$ is functorial with respect to the restriction morphisms in (1). When $U=T^*X$, we have $\Phi^{\bZ/\ell}_{z,T}=\id$.
\end{enumerate}
\end{Thm}
Taking into account the structure of $A=\text{Ext}^*_{\bZ/\ell}(\bK,\bK)$-modules, we have 

\begin{Coro}\label{moduleisomorphism}With the notation of \autoref{invariance1}, we have:
\begin{enumerate}[fullwidth]
\item $H^*(i^*)$ is a morphism of $A$-modules.
\item $H^*(\Phi^{\bZ/\ell}_{z,T})$ is an isomorphism of $A$-modules.
\end{enumerate}
\end{Coro}
For our later application, let us present a proof here. The notation is the same as in \autoref{invariance1}.
\begin{proof}[{Proof of \autoref{invariance1}:}]

\begin{enumerate}[fullwidth]
    \item \label{remarkChiustheorem}Recall \autoref{functorial} shows that we have a natural morphism $P_{U_1}\rightarrow P_{U_2}$. Then we have an equivariant morphism $P_{U_1}^{\dboxtimes\ell }\rightarrow P_{U_2}^{\dboxtimes\ell }$. Applying $F_{\ell}$, we obtain
    \begin{equation}\label{functoriality of F}
     F_{ \ell}(U_1,\bK)\xrightarrow{F_{\ell}(i,\bK)}F_{\ell}(U_2,\bK).   
    \end{equation} 
    Then the first part follows by taking stalks over $T$.

\item To prove the invariance we use the expression $C^{\bZ/\ell}_T(U,\bK) \cong \RHOM_{{\bZ/\ell}}
  (P_U^{\boxstar \ell},\beta'_T\bK[-d])$ given by the adjoint isomorphism.

\begin{itemize}[fullwidth]
\item  It is shown in \autoref{invariance of kernel} that we have an isomorphism  $P_{\varphi_z(U)} \cong \cK^{-1} \star P_U \star \cK$ where $\cK$ is given using the GKS quantization of $\varphi$.

Let us write $\cK_\ell = \cK^{\boxstar \ell}$, $\cK_\ell^{-1} = (\cK^{-1})^{\boxstar \ell}$. We remark that $\cK_\ell$ has a natural lift in the equivariant category and that $\cK_\ell$, $\cK_\ell^{-1}$
are mutually inverse for the convolution. Hence $\cK_\ell \star -$ is an equivalence and
$\RHOM_{{\bZ/\ell}}(G,H) \cong \RHOM_{{\bZ/\ell}}( \cK_\ell \star G, \cK_\ell \star H)$ for any $G,H\in
D_{\bZ/\ell}(X^{2l}\times \bR_t)$.  
We denote by $\kappa$ the auto-equivalence on $D_{\bZ/\ell}(X^{2\ell}\times \bR_t)$ induced by conjugation with $\cK_\ell$:
\begin{equation}\label{definition of kappa}
\kappa(F)\coloneqq \mathcal{K}^{-1}_\ell\star F \star  \mathcal{K}_{\ell}.
\end{equation}
Then we have an isomorphism 
$P_{\varphi_z(U)}^{\boxstar \ell} \cong \cK_\ell^{-1} \star P_U^{\boxstar \ell} \star \cK_\ell=\kappa(P_U^{\boxstar \ell} )$, and for $U=T^*X$, the isomorphism is realized by $\bK_{\Delta_{X^2}\times    [0,\infty)}^{\boxstar \ell}\cong\cK_\ell^{-1}\star \cK_\ell \star \bK_{\Delta_{X^2}\times [0,\infty)}^{\boxstar \ell} \cong \cK_\ell^{-1} \star \bK_{\Delta_{X^2}\times [0,\infty)}^{\boxstar \ell} \star \cK_\ell$.
Then the composition induces the isomorphism
\[\RHOM_{{\bZ/\ell}}(P_U^{\boxstar \ell}, \beta'_T\bK ) \overset{\kappa}{\cong} 
\RHOM_{{\bZ/\ell}}(\kappa(P_U^{\boxstar \ell}),\kappa( \beta'_T \bK )){\cong} \RHOM_{{\bZ/\ell}}(P_{\varphi_z(U)}^{\boxstar \ell},\kappa( \beta'_T \bK )).
\]
\item Therefore, to complete the proof, it is enough to construct an isomorphism $\kappa(\beta'_T\bK)= \cK_\ell^{-1} \star \beta'_T\bK \star \cK_\ell \cong
\beta'_T\bK$. Compared to Chiu's original proof, we will construct the isomorphism explicitly.

Notice that $\beta'_T\bK$ is, up to orientation and shift, the
constant sheaf on the graph of the permutation map $f \colon X^\ell \to X^\ell$, $(\bq_1,\ldots,\bq_\ell) \mapsto (\bq_2,\ldots,\bq_\ell,\bq_1)$. Set $Y=X^\ell $ and identify $Y^2 = (X^2)^\ell$ by $(\bq^1_1,\ldots,\bq^1_\ell, \bq^2_1,\ldots,\bq^2_\ell) \mapsto (\bq^1_1,\bq^2_1,\ldots,\bq^1_\ell,\bq^2_\ell)$. Then, up to degree shifting, we have \[\beta'_T\bK \cong\bK_{\Gamma_f\times \{T\}} \star E \cong E\star \bK_{\Gamma_f\times \{T\}} ,\] where $E = \delta_{Y^2!}(\omega_Y)\dboxtimes \bK_{\{0\}}$, with $\omega_Y$ the dualizing sheaf and $\delta_{Y^2}$ the usual diagonal embedding. In general, we have $E\star- \cong -\star E$.

Now we have the general fact $G \star \bK_{\Gamma_g\times \{T\}} \cong (\id_Y \times g \times \textnormal{T}_{T})_!(G)$ for any $G$ and any map $g$.  This formula has the symmetric form $\bK_{\Gamma'_g\times \{T\}} \star G\cong (g\times \id_Y \times \textnormal{T}_{T})_!(G)$ where $\Gamma'_g$ is the switched graph $\Gamma'_g =\{(g(y),y): y\in Y\}$. When $g$ is invertible, we have $\Gamma_{g^{-1}}= \Gamma'_g$. So, we obtain
\[\cK_\ell \star \beta'_T\bK \cong \cK_\ell \star \bK_{\Gamma_f\times \{T\}} \star E \cong   (\id_Y \times f \times \textnormal{T}_{T})_!(\cK_\ell)\star E,\]
and 
\[\beta'_T\bK \star \cK_\ell \cong E\star \bK_{\Gamma_f\times \{T\}}\star \cK_\ell=E\star \bK_{\Gamma'_{f^{-1}}\times \{T\}} \star \cK_\ell\cong E \star (f^{-1} \times \id_Y \times \textnormal{T}_{T})_!(\cK_\ell).\]
In coordinate $(X^2)^\ell$ we have $(f\times f)((\bq^1_j,\bq^2_j))_{j\in \bZ/\ell}=((\bq^1_{j+1},\bq^2_{j+1}))_{j\in \bZ/\ell}$. In other words $f\times f$ is the cyclic permutation of the $X^2$ factors in $(X^2)^\ell$. It is then clear that $(f \times f \times \id_\bR)_!\cK_\ell \cong  \cK_\ell$ (even in the equivariant category). Then we deduced that
\[\beta'_T\bK\star \cK_\ell  \cong E \star (f^{-1} \times \id_Y \times \textnormal{T}_{T})_!(\cK_\ell)\cong E\star (\id_Y \times f \times \textnormal{T}_{T})_!(\cK_\ell)\cong \cK_\ell \star \beta'_T\bK.\]
Consequently, we have 
\[\kappa(\beta'_T\bK)= \cK_\ell^{-1} \star \beta'_T\bK \star \cK_\ell \cong \cK_\ell^{-1} \star \cK_\ell \star \beta'_T\bK\cong
\beta'_T\bK.\]
In summary, the $\Phi^{\bZ/\ell}_{z,T}$ is defined as following. For any $f\in \textnormal{Ext}^*_{\bZ/\ell}(P_U^{\boxstar \ell},\beta'_T\bK[-d])$, we have
\[\Phi^{\bZ/\ell}_{z,T}(f): P_{\varphi_z(U)}^{\boxstar \ell} {\cong} \kappa(P_U^{\boxstar \ell})\xrightarrow{\kappa(f)}\kappa(\beta'\bK[-d]) {\cong}\beta'\bK[-d].  \]
The functoriality of $\Phi^{\bZ/\ell}_{z,T}$ follows since $\kappa$ is a functor. \end{itemize}\end{enumerate}For $U=T^*X$, the isomorphism $P_{T^*X}^{\boxstar \ell} {\cong} \kappa(P_{T^*X}^{\boxstar \ell})$ is induced by the natural isomorphism $P_{T^*X}\star \cK\cong \cK \star P_{T^*X}$. So does $\kappa(\beta'\bK[-d]) {\cong}\beta'\bK[-d]$. Then the induced isomorphism $\Phi^{\bZ/\ell}_{z,T}(f)$ is the identity on the cohomology level.
\end{proof}
Actually, the isomorphism $\kappa(\beta'_T\bK) {\cong}\beta'_T\bK$ is still true if we replace $\bK$ by $M\in D_{\bZ/\ell}(\pt)$, and moreover the isomorphism is functorial with respect to $M$. In fact, for $M\in D_{\bZ/\ell}(\pt)$, we only need to replace $K=\delta_{Y^2!}(\pi_Y^!\bK)$ in the proof by $K(M)=\delta_{Y^2!}(\pi_Y^!M)$.
Consequently, we can construct an isomorphism of functors
\[\Phi^{\bZ/\ell}_{z,T}(-):\RHOM_{\bZ/\ell}(F_\ell(U,\bK)_T,-)  \xrightarrow{\cong} \RHOM_{\bZ/\ell}(F_\ell(\varphi(U),\bK)_T,-).\]
Now, let us take $M=F_\ell(\varphi(U),\bK)_T$. Then $\text{Id}_{F_\ell(\varphi(U),\bK)_T}$ provide us with an isomorphism
\[{\Phi^{\bZ/\ell }_{z,T}}'\coloneqq \left(\Phi^{\bZ/\ell}_{z,T}({F_\ell(\varphi(U),\bK)_T})\right)^{-1}(\text{Id}_{F_\ell(\varphi(U),\bK)_T}):F_\ell(U,\bK)_T\rightarrow F_\ell(\varphi(U),\bK)_T.\]
In summary, we have
\begin{Prop}\label{strongly invariance}
For a compactly supported Hamiltonian isotopy $\varphi:I\times T^*X  \rightarrow T^*X $, there exists an isomorphism, in the equivariant category, ${\Phi^{\bZ/\ell }_{z,T}}':F_\ell(U,\bK)_T\rightarrow F_\ell(\varphi_z(U),\bK)_T$, for all $z\in I$.
\end{Prop}
\begin{RMK}In \cite[Subsection 3.7]{CyclicZHANG}, we explain how to give a new proof of the \autoref{strongly invariance} without using adjunctions. We also explain that the proposition is true for Hamiltonian homeomorphism in loc. cit. 
\end{RMK}

\subsection{Geometry of \texorpdfstring{$F_{\ell}(U,\bK)$}{}}\label{Discrete Hamiltonian loop section}
In this subsection, we assume that $U$ is dynamically admissible (\autoref{dynamicallyopensets}) and we give a more accessible expression  of the Chiu-Tamarkin complex using sheaf quantization. We then discuss the underlying geometry.

Following ideas of Chiu, we first compute $F_{\ell}(U,\bK)\cong {\textnormal{R}}\pi_{\underline{\bq}!}  \widetilde{\Delta}_X^{-1}\left({\textnormal{R}}s_{t!}^\ell P_U^{\dboxtimes \ell}\right)$ going back to the construction of $P_U$.

We recall that $\mathcal{K}$ is the sheaf quantization of a Hamiltonian $\bR^m_{ {z}}$ action on $T^*X$ with a  moment map $\mu$, $\Omega \subset \bR_{ {\zeta}}^m$ and $U=\mu^{-1}(\Omega)$. Then we have $P_U\cong \widehat{\mathcal{K}}\circ {\bK}_{\Omega} \cong \mathcal{K}\star \widetriangle{{\bK}_{\Omega}}$.

As a corollary of the proper base change and the projection formula, we have the following:
\begin{align*}
P_U^{\dboxtimes \ell} &\cong  (\mathcal{K}\star \widetriangle{{\bK}_{\Omega}} )^{\dboxtimes \ell}
\cong \textnormal{R}\pi_{\underline{z}!}\textnormal{R}s_{{\bR^\ell!}}^{2}\left( \pi_{{t_2}}^{-1}\cK^{\dboxtimes \ell} \dotimes   \pi_{{t_1}}^{-1}\widetriangle{{\bK}_{\Omega}}^{\dboxtimes \ell}\right).
\end{align*}
Next, we have
\begin{align*}
   F_{\ell}(U,\bK) &\cong {\textnormal{R}}\pi_{\underline{\bq}!}   \widetilde{\Delta}_X^{-1}{\textnormal{R}}s_{t!}^\ell\textnormal{R}\pi_{\underline{z}!} \textnormal{R}s_{\bR^\ell!}^{2}\left( \pi_{{t_2}}^{-1}\cK^{\dboxtimes \ell} \dotimes   \pi_{{t_1}}^{-1}\widetriangle{{\bK}_{\Omega}}^{\dboxtimes \ell}\right)  \\
&\cong \textnormal{R}\pi_{\underline{z}!}{\textnormal{R}}s_{t!}^\ell\textnormal{R}s_{{\bR^\ell!}}^{2} \left( \pi_{t_2}^{-1}\left({\textnormal{R}}\pi_{\underline{\bq}!}    \widetilde{\Delta}^{-1}_X\mathcal{K}^{\dboxtimes \ell  }   \right) \dotimes \pi_{t_1}^{-1}  \widetriangle{{\bK}_{\Omega}}^{\dboxtimes \ell  }     \right),
\end{align*}
where $\underline{z}=(z_1,\dots,z_\ell)\in (\bR^m)^\ell$, $t_i=(t_i^1,\dots,t_i^\ell)\in \bR^\ell$ for $i=1,2$, and $t=(t^1,\dots,t^\ell)=s^2_{
\bR^\ell}(t_1,t_2)$. Now, let $z=z_1+\cdots+z_\ell$ and take $t'_i=t_i^1+\cdots+t_i^\ell$. Using this change of coordinate, we have the decomposition $\pi_{\underline{z}}=\pi_z  s^\ell_{{z}} $ and $ s_{t}^\ell s_{{\bR^\ell}}^{2}=s_{t'}^2 (s_{t_1}^\ell\times s_{t_2}^\ell)$. Therefore, we obtain
\begin{align}\label{step1}
\begin{aligned}   F_{\ell}(U,\bK) &
\cong \textnormal{R}\pi_{\underline{z}!}{\textnormal{R}}s_{t!}^\ell\textnormal{R}s_{{\bR^\ell!}}^{2} \left( \pi_{t_2}^{-1}\left({\textnormal{R}}\pi_{\underline{\bq}!}    \widetilde{\Delta}^{-1}\mathcal{K}^{\dboxtimes \ell  }   \right) \dotimes \pi_{t_1}^{-1}  \widetriangle{{\bK}_{\Omega}}^{\dboxtimes \ell  }     \right)\\
&\cong {\textnormal{R}}\pi_{{ {z}}!}{\textnormal{R}}s_{t'!}^2 {\textnormal{R}}s_{{z}!}^\ell {\textnormal{R}}(s_{t_1}^\ell\times s_{t_2}^\ell)_!\left( \pi_{t_2}^{-1}\left({\textnormal{R}}\pi_{\underline{\bq}!}    \widetilde{\Delta}^{-1}\mathcal{K}^{\dboxtimes \ell  }   \right) \dotimes \pi_{t_1}^{-1}  \widetriangle{{\bK}_{\Omega}}^{\dboxtimes \ell  }     \right)\\
&\cong {\textnormal{R}}\pi_{{ {z}}!}{\textnormal{R}}s_{t'!}^2 {\textnormal{R}}s_{{z}!}^\ell  \left( \pi_{t'_2}^{-1}\left({\textnormal{R}}\pi_{\underline{\bq}!}    \widetilde{\Delta}^{-1}\mathcal{K}^{\boxstar \ell  }   \right) \dotimes \pi_{t'_1}^{-1}   \widetriangle{{\bK}_{\Omega}}^{\boxstar \ell  }     \right).
\end{aligned}
\end{align}
This formula shows, as the construction itself, that we can consider separately the Hamiltonian action and the cut-off by $\Omega$. Let us study the Hamiltonian action first.
In view of \eqref{step1}, it is convenient to define 
\begin{equation}
  \label{eq:defLell}
  \begin{split}
    CL_\ell(\mathcal{K}) &\coloneqq {\textnormal{R}}\pi_{\underline{\bq}!} (\widetilde{\Delta}^{-1}(\mathcal{K}^{\boxstar \ell})) \in D_{\bZ/\ell}((\bR^m_z)^{\ell}  \times  \bR_t) ,\\
   {\mathcal{CL}}_\ell(\mathcal{K}) &\coloneqq {\textnormal{R}}s_{{z}*}^{\ell}\, CL_\ell(\mathcal{K}) \in D_{\bZ/\ell}(\bR^m_z  \times  \bR_t) .
  \end{split}
\end{equation}
Noticed that the formula $CL_\ell(\mathcal{K}) = {\textnormal{R}}\pi_{\underline{\bq}!} (\widetilde{\Delta}^{-1}(\mathcal{K}^{\boxstar \ell}))$ is similar to $F'_{\ell}(P_U^{\boxstar \ell})={\textnormal{R}}\pi_{\underline{\bq}!} (\widetilde{\Delta}^{-1}(P_U^{\boxstar \ell}))$ (see \eqref{definition of F}). But in these two formulas, $\pi_{\underline{\bq}}$ has different meaning. The codomain of the first $\pi_{\underline{\bq}}$ is $\bR^m_{z}\times \bR_t$ while the codomain of second $\pi_{\underline{\bq}}$ is just $\bR_t$. So they are different formulas.

The sheaves $CL_\ell(\cK)$ and ${\mathcal{CL}}_\ell(\cK)$ encode the cohomology information of a discrete Hamiltonian loop space. Precisely, we have
\begin{Prop}\label{Discrete Hamiltonian loop} With the notation \eqref{eq:defLell} we have
\begin{enumerate}[fullwidth]
\item The sectional microsupport $\mu s(C{{L}}_\ell(\mathcal{K}))$, which is a subset of $T^*(\bR^m_z)^{\ell} $,  is contained in
\[
\left\lbrace ({z}_j, {\zeta}_j)_{j\in \bZ/\ell}:
    \begin{aligned}
    &\text{There exist } ( \bq_j, \bp_j)_{j\in \bZ/\ell}\in T^*(( X^2)^\ell)\text{ such that }
\\
&(\bq_{j+1}, \bp_{j+1})= {z_j}\cdot(\bq_j,\bp_j),\, \zeta_j=-\mu(\bq_j,\bp_j)\, ,
j\in \bZ/\ell
    \end{aligned}
   \right\rbrace  
\]
\item $C{{L}}_\ell(\mathcal{K})\cong (s_{ {z}}^{\ell})^{-1}{\textnormal{R}}s_{ {z}*}^{\ell}C{{L}}_\ell(\mathcal{K})$, $\mathcal{CL}_\ell(\mathcal{K})\cong {\textnormal{R}}s_{ {z}*}^{\ell}(s_{ {z}}^{\ell})^{-1}\mathcal{CL}_\ell(\mathcal{K})$.
\end{enumerate}
\end{Prop}
\begin{proof}
\begin{enumerate}[fullwidth]
\item It follows directly from the functorial estimate of microsupport.
First, the formula \eqref{microsupportofsheafquantization} shows that 
\[\mu s (\cK^{\boxstar\ell} )\subset \{({z}_j, {\zeta}_j, \bq_j, -\bp_j,\bq_j', \bp_j')_{j\in \bZ/\ell}:(\bq'_{j}, \bp'_{j})= {z_j}\cdot(\bq_j,\bp_j),\, j\in \bZ/\ell\}.\]
The transpose derivative of $\widetilde{\Delta}$ is given by
\[d\widetilde{\Delta}^*(\bq_{\ell},\bq_1,\dots, \bq_{\ell-1},\bq_{\ell};\bp_1,\bp_2,\dots,\bp_{2\ell-1},\bp_{2\ell})=(\bq_1,\dots, \bq_{\ell};\bp_2+\bp_3,\dots,\bp_{2\ell}+\bp_{1}).\]
By the bound \autoref{functional estimate}-(3), we deduce that $\mu s (\widetilde{\Delta}^{-1}(\cK^{\boxstar\ell})) $ is a subset of
\[ \left\lbrace ({z}_j, {\zeta}_j,\bq_j'',\bp_j'')_{j\in \bZ/\ell}:
    \begin{aligned}
    &\text{There exist } ( \bq_j, -\bp_j,\bq_j', \bp_j')_{j\in \bZ/\ell}\in T^*((X^2)^{2\ell})\text{ such that }
(\bq'_{j}, \bp'_{j})= {z_j}\cdot(\bq_j,\bp_j), \\
&\bq_j''=\bq_j'=\bq_{j+1},\,
\bp_j''=\bp_j'-\bp_{j+1},\zeta_j=-\mu(\bq_j,\bp_j)\, ,
j\in \bZ/\ell
    \end{aligned}
   \right\rbrace  .\]
Finally, let us apply the non proper estimate \autoref{non proper pushforward estimate}. The set $\pi^{\#}_{\underline{\bq''}}(SS(\widetilde{\Delta}^{-1}(\mathcal{K}^{\boxstar \ell})))$ comes from forgetting $\bq_j''$ for all $j$ from $SS(\widetilde{\Delta}^{-1}(\mathcal{K}^{\boxstar \ell}))$. Then $(z_j,\zeta_j,t,1)_{j\in \bZ/\ell}\in \mu s(C{{L}}_\ell(\mathcal{K}))$ if there exists a sequence $({z}_j^n, {\zeta}_j^n,\bp_j''^n)_{j\in \bZ/\ell} \in \pi^{\#}_{\underline{\bq''}}(SS(\widetilde{\Delta}^{-1}(\mathcal{K}^{\boxstar \ell})))$ such that ${z}_j^n\rightarrow z_j$, $\zeta_j^n\rightarrow \zeta_j$, and $\bp_j''^n\rightarrow 0$ for all ${j\in \bZ/\ell}$.

On the other hand, the relations above imply that there exists $( \bq_j^n, -\bp_j^n,\bq_j{'}^n, \bp_j{'}^n)_{j\in \bZ/\ell}\in T^*((X^2)^{\ell})$ such that $(\bq_{j}{'}^n, \bp_{j}{'}^n)= {z_j^n}\cdot(\bq_j^n,\bp_j^n)$ and $\bq_j''^n=\bq_j'^n=\bq_{j+1}^n,\,
\bp_j''^n=\bp_j'^n-\bp_{j+1}^n,\zeta_j^n=-\mu(\bq_j^n,\bp_j^n)$ for all ${j\in \bZ/\ell}$. So the continuity of the group action and the moment map show that, after taking limit $n\rightarrow \infty$, we have $(\bq'_{j}, \bp'_{j})= {z_j}\cdot(\bq_j,\bp_j)$ and $\bq_j''=\bq_j'=\bq_{j+1},\,
0=\bp_j'-\bp_{j+1},\zeta_j=-\mu(\bq_j,\bp_j)$ for all $ {j\in \bZ/\ell}$. Then we have that $\mu s(C{{L}}_\ell(\mathcal{K}))$ is contained in
\[ \left\lbrace ({z}_j, {\zeta}_j)_{j\in \bZ/\ell}:
    \begin{aligned}
    &\text{There exist } ( \bq_j, -\bp_j,\bq_j', \bp_j')_{j\in \bZ/\ell}\in T^*((X^2)^{\ell})\text{ such that }
(\bq'_{j}, \bp'_{j})= {z_j}\cdot(\bq_j,\bp_j), \\
&\bq_j''=\bq_j'=\bq_{j+1},\,
{0}=\bp_j''=\bp_j'-\bp_{j+1},\zeta_j=-\mu(\bq_j,\bp_j)\,,
j\in \bZ/\ell
    \end{aligned}
   \right\rbrace.\]
Finally, we simplify the notation by reducing the variables with primes.
\item If $(z_j,\zeta_j)_{j\in \bZ/\ell}\in \mu s(C{{L}}_\ell(\mathcal{K}))$, there exists $( \bq_j, \bp_j)_{j\in \bZ/\ell}\in T^*(( X^2)^\ell)$ such that $(\bq_{j+1}, \bp_{j+1})= {z_j}\cdot(\bq_j,\bp_j)$ for all ${j\in \bZ/\ell}$.
Therefore, the invariance of the moment map shows that 
\[\zeta_{j+1}=\mu(\bq_{j+1}, \bp_{j+1})=(\bq_{j}, \bp_{j})=\zeta_j,\qquad j\in \bZ/\ell.\]
Then, the isomorphism follows from \cite[Proposition 5.4.5(ii)]{KS90}.
\end{enumerate}\end{proof}
\begin{RMK}Even if $\mathcal{K}$ is a sheaf quantization that comes from a non-autonomous Hamiltonian function, the microsupport estimate for $CL_\ell(\mathcal{K})$ is still true, but the second statement is not true in this case.
\end{RMK}
Now, using the projection formula, we can write the formula \eqref{step1} as
\begin{equation}\label{loop sheaf rep F}
F_{\ell}(U,\bK)  \cong {\textnormal{R}}\pi_{{ {z}}!}{\textnormal{R}}s_{t!}^2   \left( \pi_{t_2}^{-1}  {{\mathcal{CL}}}_\ell(\mathcal{K}) \dotimes \pi_{t_1}^{-1}  {\textnormal{R}}s_{z!}^\ell\widetriangle{{\bK}_{\Omega}}^{\boxstar \ell  }\right).    
\end{equation}
Next, let us study ${\textnormal{R}}s_{z!}^\ell\widetriangle{{\bK}_{\Omega}}^{\boxstar \ell  }\cong {\textnormal{R}}s_{(z,t_2)!}^\ell\widetriangle{{\bK}_{\Omega}}^{\dboxtimes \ell  }$. First, with the help of \cite[Section 6, Appendix A]{FTDAgnolo},  $\widetriangle{{\bK}_{\Omega}}$ is the (inverse) Fourier-Sato transform $\widehat{\bK_{\Omega'}}$ of ${\bK}_{{\Omega}'}$, where ${\Omega}'=\{(\zeta,\tau): \tau\zeta\in \Omega, \tau>0\}$. Now, using the functorial properties of the Fourier-Sato transformation (see \cite[Section 3.7]{KS90}), and writing in the same way the two Fourier transforms, we have:
\begin{align*}
   {\textnormal{R}}s_{z!}^\ell\widetriangle{{\bK}_{\Omega}}^{\boxstar \ell  }\cong {\textnormal{R}}s_{(z,t_2)!}^\ell\widehat{{\bK}_{\Omega'}}^{\dboxtimes \ell  }
   \cong {\textnormal{R}}s_{(z,t_2)!}^\ell\widehat{{\bK}_{\Omega'}^{\dboxtimes \ell  }}
   \cong {\textnormal{R}}s_{(z,t_2)!}^\ell\widehat{{\bK}_{\Omega'^\ell}}
   \cong  \widehat{{(^ts_{(z,t_2)}^\ell)^{-1}{\bK}_{\Omega'^\ell}}}.
\end{align*}
Since the transpose of the summation map $s_{(z,t_2)}^\ell$ is the diagonal map $\delta_{(z,t_2)^\ell}$, we conclude that
\[{\textnormal{R}}s_{z!}^\ell\widetriangle{{\bK}_{\Omega}}^{\boxstar \ell  }\cong {\textnormal{R}}s_{(z,t_2)!}^\ell\widehat{{\bK}_{\Omega'}}^{\dboxtimes \ell  }
   \cong \widehat{\delta_{(z,t_2)^\ell}^{-1}{\bK}_{\Omega'^\ell}} \cong \widehat{{\bK}_{\Omega'}}\cong \widetriangle{{\bK}_{\Omega}}.\]
Our external tensor power is in fact an object of the ${\bZ/\ell}$-equivariant derived category. We need to mention that the Fourier transform (of any version) is a convolution functor defined by a kernel, which is a constant sheaf supported on a closed subset. So, on the product space, the Fourier transform is defined by a kernel that is a constant sheaf supported on a {\em product} of the same closed subsets. Then the kernel is a ${\bZ/\ell}$-equivariant sheaf. Moreover, the external tensor power is compatible with the Grothendieck 6-operations. Therefore, the Fourier transform can be defined on the equivariant derived category. Finally, all maps here are $\bZ/\ell$-equivariant with respect to cyclic permutation action and the formulas we used here are valid in the equivariant category. In conclusion,  all identities here are true in the equivariant derived category.

Consequently, \eqref{loop sheaf rep F} could be read as
\begin{align}\label{reformulation of F}
F_{\ell}(U,\bK)\cong {\textnormal{R}}\pi_{ {z}!}{\textnormal{R}}s_{t!}^2   \left(\pi_{t_2}^{-1}   {\mathcal{CL}}_\ell(\mathcal{K}) \dotimes\pi_{(\underline{\bq},t_1)}^{-1}  \widetriangle{{\bK}_{\Omega}}  \right) \cong {\textnormal{R}}\pi_{ {z}!}\left( {\mathcal{CL}}_\ell(\cK) \star  \widetriangle{{\bK}_{\Omega}}  \right).
\end{align}
From this formula, the study of $F_{\ell}(U,\bK)$ is reduced to understanding $ {\mathcal{CL}}_\ell(\mathcal{K})$.

The case $m=1$ is particularly useful for our applications. Now $\Omega=(-\infty,1)$ and $\widetriangle{{\bK}_{\Omega}}\cong\bK_{\{(z,t):-t\leq z\leq 0\}}$. For $T\geq 0$, \eqref{reformulation of F} shows
\begin{align}\label{reformulation of F_T}
\begin{aligned}
&\alpha_{\ell,X,T}(P_U^{\dboxtimes \ell})\cong F_{\ell}(U,\bK)_T\cong {\textnormal{R}}\Gamma_c\left(\bR_z\times\bR_{(t_1,t_2)}^2; \left({\mathcal{CL}}_\ell(\cK)\dboxtimes \bK_{\bR_{t_2}} \right)_{Z}\right),
\end{aligned}   
\end{align}
where $Z=\{(z,t_1,t_2):t_1+t_2=T,-t_2\leq z\leq 0\}$.

Again, using the formula \eqref{reformulation of F}, we obtain the following action spectrum estimate of the microsupport of $F_{\ell}(U,\bK)$ for dynamically admissible sets. 
\begin{Lemma}{{\cite[formula 74]{zhang2020quantitative}}}\label{lemma-actionspectrumestimate}Let $U=\{H<1\}$ be a dynamically admissible set defined by a Hamiltonian function $H$. If the boundary $\partial U$ is a non-degenerated hypersurface of restricted contact type (RCT) given by $\partial U=\{H=1\}$, then we have
\begin{equation}\label{actionspectrumestimate}
{\mu}s_L(F_\ell(U,\bK))\subset\left\lbrace t\in \bR: t= \bigg|\int_c \bp d\bq \bigg|\text{ for a closed orbit }  c\text{ of }\varphi_z^H \text{ in } \partial U \right\rbrace.
\end{equation}
\end{Lemma}
Actually, since $F_{1}(U,\bK)\cong F_{\ell}(U,\bK)$ in $D(\bR)$ (by \eqref{cyclicstructure step1}), we only need to verify the estimate for $F_1(U,\bK)$ (see \autoref{def: equivariant sheaf microsupport}). Notice that when computing the upper bound, we need the contact boundary condition to make sure we can attach only one non-constant closed characteristic.

Geometrically, we call the right hand side {\em the action spectrum} of the Reeb action in $\partial U$.

So far, we have found two different ways to understand $F_{\ell}(U,\bK)$. Initially, from the definition of $F_{\ell}(U,\bK)$, we first cut off the energy of a Hamiltonian isotopy up to Legendre transform to obtain the kernels and then use the functor $\alpha_T$ to obtain cohomology of some discrete loop space with action bound $T$. On the other hand, we can study discrete loops of a Hamiltonian isotopy first, and then cut off energy up to Legendre transform. The result of the section clarifies that these two ways are the same. The second way is more direct than the first in many cases; we will see more about this point of view when doing computation for toric domains.
\subsection{Fundamental class and capacities}\label{capacities}
Now, let us assume that $X$ is an oriented manifold of dimension $d$ with a fixed orientation and $\bK$ is a field. For an admissible open subset $U\xhookrightarrow{i_U} T^*X$ and $T \geq 0$,  \autoref{invariance1}-(1) shows that we have a morphism in the ${\bZ/\ell}$-equivariant derived category:
\[C^{\bZ/\ell}_T(T^*X,\bK) \xrightarrow{i_U^*} C^{\bZ/\ell}_T(U,\bK),\]
and it induces a morphism of $A=\text{Ext}^*_{\bZ/\ell}(\bK,\bK)$-module on cohomology
\[H^{BM}_{d-*}(X,\bK)\otimes A\cong H^*C^{\bZ/\ell}_T(T^*X,\bK) \xrightarrow{i_U^*} H^*C^{\bZ/\ell}_T(U,\bK),\]
where the first isomorphism is given in \eqref{F of the cotangent bundle}. Since $X$ is orientable, we have the fundamental class $[X]$ of $X$ in $H^{BM}_d(X,\bK)$, which is defined via $1\in H^0(X,\bK)\cong H^{BM}_d(X,\bK)$. We set $[X]^{\bZ/\ell}=[X]\otimes 1$, where $1\in A$ is the identity. 
\begin{Def}\label{definition of fundamental class}For an admissible open set $U\xhookrightarrow{i_U} T^*X$, and $T\geq 0$, we define its {\em fundamental class} $\eta^{\bZ/\ell}_T(U,\bK)$ as the image of $[X]^{\bZ/\ell}$ under $i_U^*$. That is, $\eta_T^{\bZ/\ell}(U,\bK)\coloneqq i_U^*([X]^{\bZ/\ell})\in H^0C^{\bZ/\ell}_T(U,\bK).$ When $\ell=1$, we use $\eta_T(U,\bK)$ for short.
\end{Def}
By definition, the fundamental class can be computed as the following composition: \begin{align}
  (F_{\ell}(U,\bK))_T \rightarrow (F_{\ell}(T^*X, \bK))_T\cong {\textnormal{R}}\Gamma_c(X,\bK)\xrightarrow{or} H^d{\textnormal{R}}\Gamma_c(X,\bK)[-d]\cong \bK[-d].
  \end{align}	
As a corollary of \autoref{invariance1}, we have
\begin{Prop}\label{functorial fundamental class} \begin{enumerate}[fullwidth]
    \item Let $U\subset U' \subset T^*X$ be an inclusion of admissible open sets. Through the natural morphism
\[H^0{C^{\bZ/\ell}_T(U',\bK)}\rightarrow H^0{C^{\bZ/\ell}_T(U,\bK)}\]
we have
\[\eta^{\bZ/\ell}_T(U',\bK) \mapsto \eta^{\bZ/\ell}_T(U,\bK).\]
\item Let $\varphi:I\times  T^*X \rightarrow T^*X$ be a compactly supported Hamiltonian isotopy and $U$ be an admissible open set. Recall the $A$-module isomorphism, defined in \autoref{invariance1},  \[H^*(\Phi^{\bZ/\ell}_{z,T}): H^*C^{\bZ/\ell}_T(U,\bK) \xrightarrow{\cong} H^*C^{\bZ/\ell}_T(\varphi_z(U),\bK).\]
Then we have $H^0(\Phi^{\bZ/\ell}_{z,T})(\eta^{\bZ/\ell}_T(U,\bK))=\eta^{\bZ/\ell}_T(\varphi_z(U),\bK)$ for all $z\in I$.
\end{enumerate}
\end{Prop}
We have $\eta^{\bZ/\ell}_T(T^*X,\bK)=[X]^{\bZ/\ell}$ for all $T\geq 0$. So, if there exists an open set $X'\subset X$ such that $U\subset T^*X'\subset T^*X$, we have $\eta^{\bZ/\ell}_T(U,\bK)=i_U^*([X]^{\bZ/\ell})=i_U^*([X']^{\bZ/\ell})$ by \autoref{functorial fundamental class}-(1).

Now, for $\ell \in \bN_{\geq 2}$, $p_\ell$ is the minimal prime factor of $\ell$, and $\bF_{p_\ell}$ is the finite field of order $p_\ell$. The Yoneda algebra $A=\text{Ext}^*_{\bZ/\ell}(\bF_{p_\ell},\bF_{p_\ell})$ is isomorphic to ${\bF_{p_\ell}}[u,\theta]$ (see \eqref{equivariant cohomology of a point}), where $|u|=2,\,|\theta|=1$, and $\theta^2=ku$ ($k=0$ if $\ell$ is odd and $k=\ell/2$ if $\ell$ is even). 

\begin{Def}\label{definition of capacities}For an admissible open set $U$ and $k\in \bN$ we define
\begin{equation*} \label{definition of c_k}\textnormal{Spec}(U,k) \coloneqq
\left\lbrace
  T \geq 0:\begin{aligned} &\exists p\text{ prime such that }\forall\ell \in \bO, \, p_\ell \geq p,\, \\&\eta^{\bZ/\ell}_T(U,\bF_{p_\ell}) \in u^kH^{*}C^{\bZ/\ell}_T(U,\bF_{p_\ell})
  \end{aligned}
\right\rbrace,
\end{equation*}
and
\begin{equation}
 c_k(U)\coloneqq \inf \textnormal{Spec}(U,k) \in [0,+\infty ].
 \end{equation}For a general open set $U$, we define
\[c_k(U)=\sup\{c_k(U'): U'\subset U,\,U'\text{ is admissible}\}.\]
\end{Def}
In the following, we will prove that $(c_k)_{k\in\bN}$ defines a sequence of non-trivial symplectic capacities.
\begin{Thm}\label{capacity property symplectic} The functions $c_k:\text{Open}(T^*X)\rightarrow [0,\infty]$ satisfy the following:
\begin{enumerate}[fullwidth]
    \item $c_k \leq c_{k+1}$ for all $k\in \bN$.
    \item For two open sets $U_1 \subset U_2$, we have $c_k(U_1) \leq c_k(U_2)$.
    \item For a compactly supported Hamiltonian isotopy $\varphi:I\times  T^*X \rightarrow T^*X$, we have 
    $c_k(U)=c_k(\varphi_z(U)).$ 

    \item If $X=\bR^d$, then $c_k(rU)=r^2c_k(U)$ for all $k\in \bN$ and $r>0$.
    \item Suppose $U=\{H<1\}$ is admissible such that $\partial U=\{H=1\}$ is a non-degenerated hypersurface of restricted contact type defined by a Hamiltonian function $H$. If $c_k(U) < \infty $, then $c_k(U)$ is represented by the action of a closed characteristic in the boundary  $\partial U$.
    
    \item $c_k(U)>0$ for all open sets $U$.
\end{enumerate}
\end{Thm}
\begin{proof} We can assume $U$ is admissible; the general case follows directly.
Then (1) is a consequence of \autoref{definition of capacities}. Results (2), (3) are corollaries of \autoref{functorial fundamental class}.

For (4), recall that \autoref{conformal invariance of kernel} shows that
$P_{rU} \cong \bar{R}_!P_U$ where $\bar{R}(\bq,\bq',t)=(\bq/r,\bq'/r,t/r^2)$. Then direct computation shows that we have an isomorphism in $D_{\bZ/\ell}(\bR)$:
\[R_!F_{\ell}(rU,\bK)\cong F_{\ell}(U,\bK),\]
where $R(t)=t/r^2$. In particular, we have $F_{\ell}(rU,\bK)_{r^2T}\cong F_{\ell}(U,\bK)_{T}$ for $T\geq 0$. This isomorphism commutes with the inclusion morphism induced by $U\subset T^*\bR^d$, the (4) follows.

For (5), let $T=c_k(U)$. Suppose that it is not given by the action of a closed characteristic. 

By assumption, the boundary $\partial U$ has non-degenerated Reeb dynamics, so there are only finitely many closed characteristics with action less than $2T$. So there is a small $\varepsilon>0$ such that there is no action happening in $[T-\varepsilon,T+\varepsilon]$.

However, we have the following microsupport estimate \eqref{actionspectrumestimate} for all fields $\bK$: 
\[{\mu}s_L(F_{\ell}(U,\bK))\subset\left\lbrace t\in \bR: t= \bigg|\int_c \bp d\bq \bigg|\text{ for some closed orbit }  c\text{ of }\varphi_z^H \right\rbrace.\]
Therefore $F_{\ell}(U,\bK)$ is constant on $[T-\varepsilon,T+\varepsilon]$. Consequently, $(F_{\ell}(U,\bK))_{T-\varepsilon}\cong (F_{\ell}(U,\bK))_T$, and then $\eta^{\bZ/\ell}_{T-\varepsilon}(U,\bK)=\eta^{\bZ/\ell}_T(U,\bK)$ for all $\ell$ and all $\bK$, in particular for $\bK=\bF_{p_\ell}$ for all $
\ell\in \bN$. Then we have $c_k(U)\leq T-\varepsilon$, which gives a contradiction. So we have
\begin{equation}
c_k(U)\in \left\lbrace \bigg|\int_c \bp d\bq \bigg|: c \text{ is a closed orbit of }\varphi_z^H \right\rbrace.\end{equation}
Finally, let us prove that the $c_k$'s are positive.
We will see, in \autoref{nonzeroellisoid}, that for a ball $B_a$, one has $c_k(B_a)=\lceil k/d \rceil a$.

For a general admissible open set $U$, we can assume that there exists $(\bq,0)\in U$ by applying a compactly support cut-off of a translation along $p$-direction, which is a Hamiltonian map. It does not change $c_k(U)$ by (3). Then we take a neighborhood $X'\cong \bR^d$ of $\bq$. By (2), we have $c_k(U)\geq c_k(U\cap T^*X')$. To prove $c_k(U\cap T^*X')>0$, let us take an admissible open subset $W$ of $U\cap T^*X'$ such that $(\bq,0)\in W$.

On the other hand, the functorial property \autoref{functorial fundamental class}-(1) shows that $\eta^{\bZ/\ell}_T(W,\bK)=i_W^*([X]^{\bZ/\ell})=i_W^*([X']^{\bZ/\ell})$. So, we only need to think $W$ as an open subset of $T^*X'\cong T^*\bR^d$, and then we can assume $X=X'=\bR^d$ and $\bq=0$ now. We take a standard symplectic ball $B_a\subset W$, then $c_k(W)\geq c_k(B_a)>0$. Consequently, we have $c_k(U)\geq c_k(U\cap T^*X')\geq c_k(W)>0 $.
\end{proof}
\begin{RMK}We also see from $c_k(B_a)=\lceil k/d \rceil a $ that if $U$ is a bounded open set (which is admissible by \autoref{opensetsareadmissible}), then $c_k(U)< \infty$.
\end{RMK}

\begin{RMK}
Finally, let us remark about the computability of $c_k$. As $H^*C^{\bZ/\ell}_T(U,\bK)$ is defined using $P_U$, which is an object in the derived category. Although it is unique in the derived category, we can take different chain representatives of $P_U$. Therefore, to compute $c_k(U)$, we can choose a particular chain representative of $P_U$. Usually, these chain representatives of $P_U$ admit properties that are not so obvious from general existence results like \autoref{Existence}, and \autoref{opensetsareadmissible}.

In Section 3, we will see how to construct a chain representative of $P_{X_\Omega}$, for a toric domain $X_\Omega$, using generating functions. The particular chain representative helps us to compute capacities for convex toric domains.
\end{RMK}

\section{Toric domains}\label{section: toric domain}
The $2$-dimensional rotation $\varphi_z(u)=\exp{(-2i\pi z)}u$ on $\bC_u$ is the Hamiltonian flow of the Hamiltonian function $H(u)=\pi|u|^2 $. Here, we identify $\bC_u$ with $T^*\bR_q$ by $u=q+ip$.

Consider the product action of single $2$-dimensional rotations given by
\[{ {z}}\cdot(u_1,\dots, u_n) =(\exp{(-2i\pi z_1)}u_1,\dots, \exp{(-2i\pi z_d)} u_d).\]
This is a Hamiltonian action of $\bR^d_{ {z}}$, which is indeed a torus action, on $\bC^d_{u}=T^*V$, where $V=\bR^d_{\bq}$ is a real vector space of dimension $d$, and $u=\bq+i\bp$. We call it the standard Hamiltonian torus action on $\bC^d_{u}=T^*V$.

The moment map of the standard Hamiltonian torus action is \begin{equation}\label{definition of the moment map of rotation}
 \mu: \bC^d_{u}=T^*V\rightarrow (\bR^d_z)^{*}=\bR^d_\zeta,\quad\mu(u_1,\dots, u_n)=(\pi|u_1|^2,\dots,\pi|u_d|^2).   
\end{equation}
\begin{Def}\label{def toric domain}For an open set $\Omega \subset \bR^d_{\zeta}$, we call $X_{\Omega}\coloneqq \mu^{-1}(\Omega)\subset T^*V$ an (open) {\em toric domain}. We say $X_\Omega$ is a {\em convex toric domain} if $|{\Omega}|\coloneqq \{{\zeta}\in \bR^d:(|\zeta_1|,\dots ,|\zeta_d|) \in \Omega\}$ is convex. We say $X_\Omega$ is a {\em concave toric domain} if $\bR^d_{{\zeta}\geq 0}\setminus \Omega$ is convex.
\end{Def}
\begin{RMK}\label{toricdomainremark}Since the moment map $\mu$ has the image $\bR^d_{ {\zeta}\geq 0}$, the toric domain $X_\Omega$ is determined by $\Omega\cap\bR^d_{ {\zeta}\geq 0}$. So we have freedom to choose suitable $\Omega$. For example, we always assume
$ -\bR^d_{ {\zeta}\geq 0} \subset\Omega.$
If $X_{\Omega}$ is a convex or a concave toric domain, one can indeed take $\Omega$ to be convex or concave (in the usual sense) and satisfying the condition $ -\bR^d_{ {\zeta}\geq 0} \subset\Omega.$ (e.g. replace $\Omega$ by $\Omega - \bR^d_{ {\zeta}\geq 0}$).
\end{RMK}
For example, we can take a non-decreasing sequence $a=(a_1,\dots,a_d)$ of positive real numbers, let $\Omega_{D(a)}= \{{\zeta}\in \bR^d_{\zeta}: \zeta_i < a_i, i\in [d]\}$ and $\Omega_{E(a)}=\{{\zeta}\in \bR^d_{\zeta}: \frac{\zeta_1}{a_1}+\cdots +\frac{\zeta_d}{a_d}<1\}$. Then
$X_{\Omega_{D(a)}}=D(a)$ is an open poly-disc and $X_{\Omega_{E(a)}}=E(a)$ is an open ellipsoid. Both are convex toric domains.
\subsection{Generating function model for microlocal kernel of Toric domains}\label{GFmodelToricdomain}

In \cite[Proposition 3.10]{chiu2017}, Chiu constructs a sheaf quantization of Hamiltonian rotation in all dimensions, particularly for the $2$-dimensional  $\varphi_z$, say 
$\mathcal{S} \in \cD(\bR_z\times \bR_{q_1}\times\bR_{q_2} )$. This quantization possesses one more property than we stated for general sheaf quantizations (see \eqref{sheafquantization}), namely
\begin{equation}\label{formula of GF sheaf quantization}
 \mathcal{S}\cong \textnormal{R}\pi_{(q_2,\dots,q_N)!}{\bK}_{\{(z,q_1,\dots,q_{N+1},t):t+\sum_{j=1}^N S_{H}(z/N,q_j,q_{j+1})\geq 0\}},   
\end{equation}
where we identify $q_{N+1}$ with $q_2$ after pushforward, $N$ is big enough so that $z/N \in (-1/4,0)\cup (0,1/4)$, and $S_H$ is the generating function of the Hamiltonian rotation: 
\begin{equation}\label{formula of S_H}
S_H(z,q,q')= \frac{q^2+{q'}^2}{2\tan(2\pi z)} -\frac{qq'}{\sin(2\pi z)}.    
\end{equation}
The formula \eqref{formula of GF sheaf quantization} is essential when computing the Chiu-Tamarkin complexes for convex toric domains.

Let
\begin{equation}\label{sheafquantization of Hamiltonian torus action}
 \mathcal{T}\coloneqq \mathcal{S}^{\boxstar d}= \textnormal{R}s_{t!}^d(\mathcal{S}^{\dboxtimes d}) \in \cD(\bR_{{z}}^d\times V_{1}\times V_{2} ),   
\end{equation}
where $ V_{i}=\bR^d_{\bq_i}$.
The microsupport estimates show that $\mathcal{T}$ is a sheaf quantization of the standard torus action in the sense of \eqref{sheafquantization}. As a corollary of \autoref{Existence}, we have 
\begin{Prop}A toric domain $X_\Omega$ is dynamically admissible by the distinguished triangle
\begin{equation}
\widehat{\mathcal{T}}\circ {\bK}_{\Omega} \rightarrow {\bK}_{\Delta_{V^2}\times \{t\geq 0\}} \rightarrow \widehat{\mathcal{T}}\circ {\bK}_{\bR^d_{{\zeta}}\setminus\Omega}\xrightarrow{+1} ,    
\end{equation}
and the pair of kernels
\begin{equation}
    P_{X_{\Omega}}\coloneqq \widehat{\mathcal{T}}\circ {\bK}_{\Omega} ,\qquad Q_{X_{\Omega}}\coloneqq \widehat{\mathcal{T}}\circ {\bK}_{\bR^d_{{\zeta}}\setminus\Omega}.
\end{equation}
\end{Prop}
This pair of microlocal kernels $(P_{X_{\Omega}}, Q_{X_{\Omega}}) $ constructed from $\mathcal{T}$ is called the {\em generating function model} of the microlocal kernels associated to toric domains.

 Actually, by the microsupport estimate of $\widehat{\mathcal{T}}$(see \eqref{fourier transform of sheaf quantization}), if $(\zeta,z,\bq,\bp,\bq',\bp',t,\tau)\in \dot{SS}(\widehat{\mathcal{T}})$ then we have $\zeta=\mu(\bq,\bp)\in \bR^d_{\zeta\geq 0}$. So, if $\zeta\notin \bR^d_{\zeta\geq 0}$ and $(\zeta,z,\bq,\bp,\bq',\bp',t,\tau)\in SS(\widehat{\mathcal{T}})$, we have $(\bp,\bp',\tau)=0$. Accordingly, for any $\zeta\notin \bR^d_{\zeta\geq 0}$, we have $SS(\widehat{\mathcal{T}}|_{(\zeta,\bq,\bq')})\subset \{\tau= 0\}$ by the microsupport estimate (\autoref{functional estimate}). So $\widehat{\mathcal{T}}|_{(\zeta,\bq,\bq')}\cong M_\bR$ is a constant sheaf over $\bR$ by \autoref{microsupported in zero section and local system} for some $M\in D(\bK-\Mod)$.
 As $\widehat{\mathcal{T}}|_{(\zeta,\bq,\bq')}\in \cD(\pt)$, and we have $\widehat{\mathcal{T}}|_{(\zeta,\bq,\bq')}\cong M_\bR\cong M_\bR\star \bK_{[0,\infty)} \cong 0 $. We conclude that $\supp(\widehat{\mathcal{T}}) \subset \bR^d_{\zeta\geq 0}$.
 
 Consequently, the kernel $ P_{X_{\Omega}}$ satisfies \begin{equation}\label{kernel toric domain only depends on the toric domain}
  P_{X_{\Omega}}\coloneqq \widehat{\mathcal{T}}\circ {\bK}_{\Omega}\cong \textnormal{R}\pi_{\zeta!}(\widehat{\mathcal{T}}\dotimes \bK_{ \Omega\times X^2\times\bR_t})\cong \textnormal{R}\pi_{\zeta!}(\widehat{\mathcal{T}}\dotimes \bK_{( \Omega\cap \bR^d_{\zeta\geq 0})\times X^2\times\bR_t}),   
 \end{equation}
which only depends on $\Omega\cap \bR^d_{\zeta\geq 0}$. So, it is the same as \autoref{toricdomainremark} that the notation $P_{X_\Omega}$ makes sense.

In general, it is complicated to compute the Fourier transform $\widehat{\mathcal{T}}$. However, with the help of associativity of composition and convolution \eqref{monoidal identities}, we have\begin{equation}\label{GFmodel convex case}
    \widehat{\mathcal{T}}\circ {\bK}_{\Omega} \cong \mathcal{T} \star \widetriangle{{\bK}_{\Omega}}.
\end{equation}When $X_\Omega$ is convex, we can take a suitable $\Omega$, which is convex in the usual sense. Then, the Fourier transform $\widetriangle{{\bK}_{\Omega}}$ is easy to compute. Actually, when $X_\Omega$ is convex, we have $\widetriangle{{\bK}_{\Omega}}\cong \bK_{\Omega^\circ }$ by a similar argument with \autoref{product}, where
\[\Omega^\circ =\{(z,t):t+\langle z,\zeta\rangle \geq 0,\, \forall\zeta \in \Omega\}.\]
The assumption $ -\bR^d_{ {\zeta}\geq 0} \subset\Omega$ shows $\Omega^\circ \subset\bR^d_{z\leq 0} \times [0,\infty)$.
Then we conclude that when $X_\Omega$ is a convex toric domain, we have
\begin{equation}
    P_{X_{\Omega}} \cong \mathcal{T}\star {\bK}_{\Omega^{\circ}}, \qquad F_{\ell}(X_{\Omega},\bK) \cong \tnR\pi_{ {z}!}{\tnR}s_{t!}^2   \left(  \mathcal{CL}_\ell(\mathcal{T}) \star  {\bK}_{\Omega^{\circ}}  \right).
\end{equation}
\begin{eg}\label{example toric domains}Let $a=(a_1,\dots,a_d)$ be a non-decreasing sequence of positive real numbers.
\begin{enumerate}[fullwidth]
\item Suppose $\Omega_{D(a)}= \{ {\zeta}: \zeta_i < a_i, i\in[d]\}$, then $X_{\Omega_{D(a)}}=D(a)$ is an open poly-disc. Let $P_{r}$ be the kernel of the open disc $\{ \pi |u|^2 < r \}$ in $\bC$, then \autoref{product} applies and $P_{D(a)}\cong P_{a_1}\boxstar\cdots \boxstar P_{a_d}$.

\item Suppose $\Omega_{E(a)}=\{ {\zeta}: \frac{\zeta_1}{a_1}+\cdots +\frac{\zeta_d}{a_d}<1\}$, then $X_{\Omega_{E(a)}}=E(a)$ is an open ellipsoid, and $\Omega_{E(a)}^\circ=\{( {z},t): t\geq -a_1z_1=\cdots=-a_dz_d\geq 0\}$. 

Let $i: \bR_z\rightarrow \bR_{  {z}}^d,\, z\mapsto (a_1z,\dots,a_dz)$, then ${\bK}_{\Omega_{E(a)}^\circ}={\textnormal{R}}(i\times \id_{\bR})_!{\bK}_{\{t\geq -z\geq 0\}}$. Therefore, we have
\[P_{E(a)}  \cong \mathcal {T}\star {\textnormal{R}}(i\times \id_{\bR})_!{\bK}_{\{t\geq -z\geq 0\}} \cong ((i\times \id_{\bR})^{-1}\mathcal{T}) \star {\bK}_{\{t\geq -z\geq 0\}} \cong \widehat{(i\times \id_{\bR})^{-1}\mathcal{T}} \circ {\bK}_{(-\infty,1)}.\]
Here we should be careful that, to obtain the second isomorphism, we need to use the explicit formula \eqref{formula of GF sheaf quantization} and \eqref{sheafquantization of Hamiltonian torus action}.

One can check directly that $(i\times \id_{\bR})^{-1}\mathcal{T}$ is the sheaf quantization of the diagonal Hamiltonian rotation $\varphi_z({u})=(\exp{ (\frac{-2i\pi z}{a_1})}u_1,\dots, \exp{ (\frac{-2i\pi z}{a_d})} u_d)$ in the sense of \eqref{sheafquantization}. In particular, when $a_1=\dots=a_d=\pi R^2>0$, the construction is the same as Chiu's for balls.
\end{enumerate}
\end{eg}
\begin{RMK}\label{concave}
For the concave toric domain case, the Fourier transform $\widetriangle{{\bK}_{\Omega}}$ is not as simple as the convex case (which is a complex only concentrated in degree 0). Actually, $\widetriangle{{\bK}_{\Omega}}$ is represented by a complex of sheaves concentrated in cohomological degree $[0,d]$. Accordingly, the results in the next subsection cannot generalize directly to the concave situation. However, some
manual computations of capacities for concave toric domains are still as predicted in
\autoref{conjecture}.

For toric domains neither convex nor concave, the first example we can consider is an open annulus bounded by two concentric spheres. Then we can take $\Omega=\{x\in \bR^d: a<\sum x_i <A\}$. In this case, when $T\geq 0$, we can only extract numerical information about the exterior sphere from $\widetriangle{{\bK}_{\Omega}}$. Then we cannot know numerical information for the interior ball. Maybe it is a feature of the present definition of capacities, we expect more understanding to overcome this defect.
\end{RMK}

\subsection{Chiu-Tamarkin complex and Capacities of Convex Toric Domains}In this subsection, we focus on convex toric domains, that is, $X_{\Omega}=\mu^{-1}(\Omega)$, where $\Omega \subset \bR^d$ is an open set such that $\{(\zeta_1,\zeta_d)\in \bR^d:(|\zeta_1|,\dots ,|\zeta_d|) \in \Omega\}$ is convex. As we discussed in \autoref{toricdomainremark}, we could take a convex $\Omega$ such that $\bR^d_{ {\zeta}\leq 0} \subset\Omega$. The identity \eqref{kernel toric domain only depends on the toric domain} shows that such a choice of $\Omega$ does not affect the computation of Chiu-Tamarkin complex for $X_{\Omega}$ and we will see this feature again in \autoref{rmkgammatopology}.

One can verify that, under such conditions, the polar cone satisfies $\{O\}\times \bR_{\geq 0}\subset \Omega^\circ \subset \bR^d_{\leq 0}\times \bR_{\geq 0}$, where $O \in \bR^d$ is the origin. For  $T\geq 0$, we set
\begin{align*}
  &{\Omega}^\circ_T \coloneqq {\Omega}^\circ\cap \{t=T\}=\{z\in \bR^d: T+\langle z, \zeta \rangle \geq 0, \forall \zeta \in \Omega\}.
\end{align*}

We also define the function $
  I(z)=\sum_{i=1}^d \big\lfloor {-z_i} \big\rfloor,\,z\in \bR^d.
$
For a subset $\Sigma\subset \bR^d$, we define
\begin{equation}\label{definition of I and infty}
   \|\Sigma\|_{\infty}=\max_{z\in \Sigma}\|z\|_{\infty} \quad \text{ and }
 I(\Sigma)=\max_{z\in\Sigma}I(z).
\end{equation}
Then we have $\|\Omega^\circ_T\|_{\infty}=T\|\Omega^\circ_1\|_{\infty}$ for $T\geq 0$.

For $x,y\in \bR^d$, the segment $\overline{xy}$ is defined as $\{tx+(1-t)y: t\in [0,1]\}$.

\begin{Thm}\label{structure toric domain module}Let $X_\Omega \subset  T^*V$ be a convex toric domain and $\ell\in \bN_{\geq 2}$. If $0\leq T <p_\ell /\|\Omega^\circ_1\|_{\infty}$, we have

\begin{itemize}[fullwidth]
    
\item For each $Z\in \Omega^\circ_T$, the inclusion of the segment $\overline{OZ} \subset \Omega^\circ_T$ induces a decomposition of the fundamental class $\eta^{\bZ/\ell}_T(X_{\Omega},\bF_{p_\ell})=u^{I(Z)}\Lambda_{Z,\ell}$ for a non-torsion element $\Lambda_{Z,\ell}\in H^{-2I(Z)}C^{\bZ/\ell}_T(X_\Omega,\bF_{p_\ell})$. In particular, $\eta^{\bZ/\ell}_T(X_{\Omega},\bF_{p_\ell})$ is non-zero.

\item The minimal cohomology degree of $H^*C^{\bZ/\ell}_T(X_\Omega,\bF_{p_\ell})$ is exactly $-2I(\Omega^\circ_T)$, i.e.,
    \[ H^*C^{\bZ/\ell}_T(X_\Omega,\bF_{p_\ell})\cong  H^{\geq -2I(\Omega^\circ_T)}C^{\bZ/\ell}_T(X_\Omega,\bF_{p_\ell}), \] 
    and
\[ H^{-2I(\Omega^\circ_T)}C^{\bZ/\ell}_T(X_\Omega,\bF_{p_\ell})\neq 0 . \] 
\item $H^*C^{\bZ/\ell}_T(X_\Omega,\bF_{p_\ell})$ is a finitely generated $\bF_{p_\ell}[u]$-module. The free part is isomorphic to $A=\bF_{p_\ell}[u,\theta]$, so $H^*C^{\bZ/\ell}_T(X_\Omega,\bF_{p_\ell})$ is of rank $2$ over $\bF_{p_\ell}[u]$. 

The torsion part is located in cohomology degree $[-2I(\Omega^\circ_T),-1]$.  $H^*C^{\bZ/\ell}_T(X_\Omega,\bF_{p_\ell})$ is torsion free when $X_\Omega$ is an open ellipsoid.\end{itemize}
\end{Thm}

Before proving \autoref{structure toric domain module}, let us use it to compute the capacities $c_k(X_\Omega)$.
\begin{Thm}\label{computation of capacities of convex toric domian}For a convex toric domain $X_\Omega \subsetneqq T^*V$, we have
\[c_k(X_\Omega)=\inf\left\{T\geq 0:\exists  {z}\in \Omega_T^\circ, I( {z})\geq k\right\}.\]
\end{Thm}
\begin{proof}
Let $S=\left\{T\geq 0:\exists  {z}\in \Omega_T^\circ, I( {z})\geq k\right\}$, $L=\inf(S)$.

For $T\in S$, there is $ {Z}\in \Omega_T^\circ$ such that $I( {Z})= k$. Consider the closed inclusion of the segment $\overline{OZ} \subset \Omega_T^\circ$. We choose a prime $p$ with $ p >T\|\Omega^\circ_1\|_{\infty}$. Then for all $\ell\in \bO$ with $p_\ell \geq p $, we have $p_\ell > T\|\Omega^\circ_1\|_{\infty}$, and \autoref{structure toric domain module} shows that the closed inclusion induces a decomposition $\eta^{\bZ/\ell}_T(X_\Omega,\bF_{p_\ell})=u^k\Lambda_{Z,\ell}$. So $T\in \textnormal{Spec}(X_{\Omega},k)$, and $L\geq c_k(X_{\Omega})$.

Conversely, if $T\in \textnormal{Spec}(X_{\Omega},k)$, there is a prime $p$ such that for all $\ell\in \bO$ with $p_\ell \geq p$ there is a $\Lambda_\ell \in H^*C^{\bZ/\ell}_T(X_{\Omega},\bF_{p_\ell})$ such that $\eta^{\bZ/\ell}_T(X_\Omega,\bF_{p_\ell})=u^k\Lambda_\ell$. Now, we can take a prime $\ell=p_\ell>p$ big enough such that $T<\ell /\|\Omega^\circ_1\|_{\infty}$, then $\eta^{\bZ/\ell}_T(X_\Omega,\bF_{p_\ell})$ and $\Lambda_\ell$ are non-zero. Hence, we have an equation of degree: $0=|\eta^{\bZ/\ell}_T(X_\Omega,\bF_{p_\ell})|=2k+|\Lambda_\ell|$, which shows that $2k=-|\Lambda_\ell|$. Therefore, \autoref{structure toric domain module} shows $2k=-|\Lambda_\ell| \leq 2I(\Omega^\circ_T)$. Hence $T\in S$, and $c_k(X_{\Omega}) \geq L$.
\end{proof}

Here, we test the result by the example of ellipsoids. They are all direct corollaries of \autoref{structure toric domain module} and \autoref{computation of capacities of convex toric domian}.
\begin{Coro}\label{nonzeroellisoid}Let $X_{\Omega}=E=E(a_1,\dots,a_d)$ be an ellipsoid and $\ell\in\bO$. For $0\leq T <p_\ell a_1$, set $Z(a)=(-T/a_1,\dots, -T/a_d)$. We have $H^*C^{\bZ/\ell}_T(E,\bF_{p_\ell})\cong u^{-I(Z(a))}\bF_{p_\ell}[u,\theta]$, the fundamental class is non-zero in all cases, and $c_k(E)=\min\{T\geq 0: \sum_{i=1}^d\lfloor {T/a_i} \rfloor\geq k\}$. In particular, $c_k(B_a)=\lceil k/d \rceil a$.
\end{Coro}

\subsection{Cohomology sheaf \texorpdfstring{$\mathcal{CL}_\ell(\mathcal{T})$}{} for the standard torus action}\label{Cohomology sheaf section}Recall the results of \autoref{GFmodelToricdomain}, and discussions in \autoref{Discrete Hamiltonian loop section}. It is necessary to study the cohomology sheaf $\mathcal{CL}_\ell(\mathcal{T})$ carefully. Recall that $\mathcal{T}=\mathcal{S}^{\boxstar d}= {\textnormal{R}}s_{t!}^d(\mathcal{S}^{\dboxtimes d})$, where $\mathcal{S}$ is the sheaf quantization of Hamiltonian rotation in dimension $2$. Using the K\"unneth formula and \autoref{Discrete Hamiltonian loop}, we have
\begin{align*}
    \mathcal{CL}_\ell(\mathcal{T}) &\cong {\textnormal{R}}s_{\underline{z}*}^\ell\left( \left(  (s_{ z_j}^{\ell})^{-1}\mathcal{CL}_\ell(\mathcal{S})\right)^{\boxstar d}\right)
    \cong {\textnormal{R}}s_{\underline{z}*}^\ell\left(  (s_{ \underline{z}}^{\ell})^{-1}\left(  \mathcal{CL}_\ell(\mathcal{S})\right)^{\boxstar d}\right)\\
    &\cong {\textnormal{R}}s_{\underline{z}*}^\ell (s_{ \underline{z}}^{\ell})^{-1}\left( \left(  \mathcal{CL}_\ell(\mathcal{S})\right)^{\boxstar d}\right)
    \cong {\textnormal{R}}s_{t!}^d\left(  \mathcal{CL}_\ell(\mathcal{S})\right)^{\boxstar d},
\end{align*}
where $\underline{z}=(z_1,\dots,z_d)$. 
Moreover, an explicit formula for  $\mathcal{CL}_\ell(\mathcal{S})$ is obtained by Chiu:
\begin{Prop}\textnormal{(\cite[Formula (38)]{chiu2017})}\label{chiu result on loop sheaf} For all fields $\bK$, there exists a (unique) sheaf ${\mathcal{E}_{\ell}}\in D_{\bZ/\ell}(\bR_z)$ such that we have  an isomorphism in $D_{\bZ/\ell}(\bR_z\times \bR_t)$
\begin{equation}\label{loop sheaf of GF model of rotation}
 \mathcal{CL}_\ell(\mathcal{S}) \cong {{\mathcal{E}_{\ell}}} \dboxtimes {\bK}_{[0,\infty)}.   
\end{equation}
Moreover, for any $N \in \bN$, 
\begin{equation}\label{E sheaf of GF model of rotation}
\left.{\mathcal{E}_{\ell}}\right|_{(-N\ell /4,0)}\cong {\textnormal{R}}\pi_{\underline{q}!}{\bK}_{\mathcal{W}^{N}_{\ell}},    
\end{equation}
with $\underline{q}=(q_1,\dots,q_{N\ell})$,
\[\mathcal{W}^{N}_{\ell}=\{(z,q_1,\dots,q_{N\ell})\in (-N\ell /4,0) \times \bR^{N\ell }:\sum_{k\in \mathbb{Z}/N\ell} S_H(z/N\ell,q_k,q_{k+1}) \geq 0    \},\]
and 
\[S_H(z,q_k,q_{k+1})= \frac{q_k^2+q_{k+1}^2}{2\tan(2\pi z)} -\frac{q_k q_{k+1}}{\sin(2\pi z)}. \]
The ${\bZ/\ell}$-action on ${\mathcal{E}_{\ell}}$ is induced by the linear action $(q_k)\mapsto (q_{k-N})$ of ${\bZ/\ell}$ on $\bR^{N\ell}$, and ${\bZ/\ell}$ acts trivially on $\bR_z\times \bR_t$. 
\end{Prop}
A disadvantage for the formula \eqref{E sheaf of GF model of rotation} is that we don't know if the isomorphism can be extended to $z=0$ since the right hand side is not defined for $z=0$. Such an extension is necessary for our later computation. So, let us start from an extension of the isomorphism \eqref{E sheaf of GF model of rotation} to $z=0$. Notice that $\sin(2\pi z/N\ell)<0$ for $z/N\ell\in (-1/4,0)$. One can rewrite $\mathcal{W}^{N}_{\ell}$ as follows:
\begin{equation*}\label{original definition of W}
\mathcal{W}^{N}_{\ell}=\left\{(z,q_1,\dots,q_{N\ell})\in (-N\ell/4,0) \times \bR^{N\ell}:\cos(2\pi z/N\ell)\sum_{k\in \mathbb{Z}/N\ell} q_k^2 \leq \sum_{k\in \mathbb{Z}/N\ell} q_k q_{k+1}\right\}.\end{equation*}
Let us define \begin{equation}\label{definition of Q}
Q(z,q_1,\dots,q_{N\ell})\coloneqq  \sum_{k\in \mathbb{Z}/N\ell\mathbb{Z}} \left(q_k q_{k+1}-\cos(2\pi z/N\ell)q_k^2   \right) .
\end{equation}
Since $Q(0,q_1,\dots,q_{N\ell})$ is well defined, we can extend the definition of $\mathcal{W}^{N}_{\ell}$ (using the same notation) to
\begin{equation}\label{definition of W}
\mathcal{W}^{N}_{\ell}=\{(z,q_1,\dots,q_{N\ell})\in (-N\ell/4,0] \times \bR^{N\ell}:Q(z,q_1,\dots,q_{N\ell})\geq 0\}.
\end{equation}
For our convenience, we also set, for $z\in (-N\ell/4,0]$,
\begin{equation}\label{definition of W(z)}
\mathcal{W}^{N}_{\ell}(z)=\{(q_1,\dots,q_{N\ell})\in  \bR^{N\ell}:Q(z,q_1,\dots,q_{N\ell})\geq 0\}.
\end{equation}
The ${\bZ/\ell}$-action on the extension is the same as the original one.

Now take $\mathcal{E}'_\ell\coloneqq {\textnormal{R}}\pi_{\underline{q}!}i_!{\bK}_{\mathcal{W}^{N}_{\ell}} \in D_{\bZ/\ell}((-N\ell/4,+\infty))$, where $i:(-N\ell/4,0]\times \bR^{N\ell}  \hookrightarrow (-N\ell/4,+\infty)\times \bR^{N\ell}$ is the closed inclusion.

By the fundamental inequality, we have $\sum_{k} q_k^2 \geq \sum_{k} q_k q_{k+1}$, and it takes equality when $q_1=\cdots=q_{N\ell}$. So  \[\mathcal{W}^{N}_\ell(0)=\{(q_1,\dots,q_{N\ell})\in  \bR^{N\ell}:q_1=\cdots=q_{N\ell}\}=\Delta_{\bR^{N\ell}}.\]
Then we have that $(\mathcal{E}'_\ell)_0=\textnormal{R}\Gamma_c(\mathcal{W}^{N}_\ell(0),{\bK}_{\Delta_{\bR^{\ell}}})\cong \textnormal{R}\Gamma_c(\Delta_{\bR^{\ell}},{\bK}_{\Delta_{\bR^{\ell}}})$.

On the other hand, one can check that $\mathcal{CL}_\ell(\mathcal{S})|_{\{z=0\}}=\textnormal{R}\Gamma_c(\Delta_{\bR^{\ell}},{\bK}_{\Delta_{\bR^{\ell}}})\dboxtimes {\bK}_{\{t\geq 0\}}$ by definition of $\mathcal{CL}_\ell(\mathcal{S})$ since $\mathcal{S}|_{\{z=0\}}={\bK}_{\Delta_{\bR^{2}}}\dboxtimes {\bK}_{\{t\geq 0\}}$. Therefore, we have $\mathcal{CL}_\ell(\mathcal{S})|_{\{z=0\}}=(\mathcal{E}'_\ell)_0\dboxtimes {\bK}_{\{t\geq 0\}}$. 

However, stalk-wise isomorphism is not necessary extend to a global one in general. So, we need the following prove to obtain a global extension of the isomorphism \eqref{E sheaf of GF model of rotation}.
\begin{Lemma}\label{loop sheaf of GF model of rotation 2}We have an equivariant isomorphism
\[ \mathcal{E}_\ell|_{(-N\ell/4,0]}\cong\mathcal{E}'_\ell|_{(-N\ell/4,0]}={\textnormal{R}}\pi_{\underline{q}!}{\bK}_{\mathcal{W}^{N}_{\ell}}.\]
\end{Lemma}
\begin{proof}Using \autoref{functional estimate} and \autoref{non proper pushforward estimate}, one can show that $SS(\mathcal{E}_\ell), SS(\left(\mathcal{E}_\ell\right)_{(-N\ell/4,0]}) \subset \{\zeta \leq 0\}$. Now, consider the distinguished triangle
\[\textnormal{R}\Gamma_{[0,\infty)}(\left(\mathcal{E}_\ell\right)_{(-N\ell/4,0]}) \rightarrow \left(\mathcal{E}_\ell\right)_{(-N\ell/4,0]} \rightarrow \textnormal{R}\Gamma_{(-N\ell/4,0)}(\left(\mathcal{E}_\ell\right)_{(-N\ell/4,0]})\xrightarrow{+1}.\]

By definition, we have $\supp(\textnormal{R}\Gamma_{[0,\infty)}(\left(\mathcal{E}_\ell\right)_{(-N\ell/4,0]}))\subset \{0\}$. On $(-N\ell/4,+\infty)$, the closed set $[0,\infty)$ is defined by the function $f(z)=z$ and $\{f(z) \geq 0\}$. Therefore, by definition of microsupport, $(\textnormal{R}\Gamma_{\{z\geq 0\}}(\left(\mathcal{E}_\ell\right)_{(-N\ell/4,0]}))_0\cong 0$ since $df_0=(0,1)\notin SS(\left(\mathcal{E}_\ell\right)_{(-N\ell/4,0]})$. So we have $(\textnormal{R}\Gamma_{\{z\geq 0\}}(\left(\mathcal{E}_\ell\right)_{(-N\ell/4,0]}))_0\cong 0$ and we have an isomorphism $\left(\mathcal{E}_\ell\right)_{(-N\ell/4,0]}\cong \textnormal{R}\Gamma_{(-N\ell/4,0)}(\left(\mathcal{E}_\ell\right)_{(-N\ell/4,0]})$. This isomorphism holds in the equivariant category since the corresponding morphism is an equivariant morphism.

The argument is purely microlocal, so we also have $\left(\mathcal{E}'_\ell\right)_{(-N\ell/4,0]} \cong \textnormal{R}\Gamma_{(-N\ell/4,0)}(\left(\mathcal{E}'_\ell\right)_{(-N\ell/4,0]})$.

On the other hand, the isomorphism \eqref{E sheaf of GF model of rotation} and our discussion on $\mathcal{W}^{N}_\ell$ show that $j^{-1}(\left(\mathcal{E}_\ell\right)_{(-N\ell/4,0]}) \cong j^{-1}(\left(\mathcal{E}'_\ell\right)_{(-N\ell/4,0]})$ where $j$ is the open inclusion $(-N\ell/4,0) \hookrightarrow (-N\ell/4,\infty)$. Therefore, the natural isomorphism $\textnormal{R}j_{*}j^{-1}\cong \textnormal{R}\Gamma_{(-N\ell/4,0)}$ shows that
\[\left(\mathcal{E}_\ell\right)_{(-N\ell/4,0]} \cong \textnormal{R}j_{*}j^{-1}(\left(\mathcal{E}_\ell\right)_{(-N\ell/4,0]})\cong \textnormal{R}j_{*}j^{-1}(\left(\mathcal{E}'_\ell\right)_{(-N\ell/4,0]}) \cong \left(\mathcal{E}'_\ell\right)_{(-N\ell/4,0]}.\]
Finally, we conclude by restricting the isomorphism to $(-N\ell/4,0]$ and the definition of $\mathcal{E}'_\ell$.
\end{proof}

{\bf Topology of $\mathcal{W}^N_\ell(z)$: }We know that $(\mathcal{E}_\ell)_{z}\cong\textnormal{R}\Gamma_c(\mathcal{W}^N_\ell(z),\bK)$ if $ -N\ell/4<z \leq 0$ (\autoref{loop sheaf of GF model of rotation 2}). For a fixed $z \in (-N\ell/4,0]$, the function $Q_z(q_1,\dots,q_{N\ell})=Q(z,q_1,\dots,q_{N\ell})$ is a quadratic form by \eqref{definition of Q}. Therefore, it is easy to study the topology of $\mathcal{W}^{N}_{\ell}(z)=\{(q_1,\dots,q_{N\ell})\in\bR^{N\ell}:Q_z\geq 0\}$. The matrix of $Q_z$ in the standard basis is a circulant matrix
\begin{center}
$A_z=\left(\begin{matrix}
     -\cos(\frac{2\pi z}{N\ell})       & \frac{1}{2}       & 0 &\cdots     & \frac{1}{2} \\
     \frac{1}{2} & -\cos(\frac{2\pi z}{N\ell})       & \frac{1}{2} &\cdots     & 0 \\
     0 & \frac{1}{2} & -\cos(\frac{2\pi z}{N\ell}) &\cdots     & 0 \\
      \vdots & \vdots &\vdots &                 & \vdots \\
      \frac{1}{2}       & 0       & 0 &\cdots     & -\cos(\frac{2\pi z}{N\ell}) 
\end{matrix}\right).$
\end{center}
So one can diagonalize $A_z$ unitarily using the discrete Fourier transform \[(\omega^{(i-1)(j-1)})_{i,j=\bZ/N\ell},\]
where $\omega$ is a primitive $N\ell^{th}$ root of unity. Therefore, the eigenvalues of $A_z$ are
\begin{equation}\label{eigenvalue of $A_z$}
\lambda_k(z)=\text{Re}\left(\exp\left(\frac{2\pi k \sqrt{-1}}{N\ell}\right)\right)-\cos\left(\frac{2\pi z}{N\ell}\right)=\cos\left(\frac{2\pi k}{N\ell}\right)-\cos\left(\frac{2\pi z}{N\ell}\right),    
\end{equation}
where $k\in \mathbb{Z}/N\ell$.

We always have $\lambda_0(z)=1-\cos\left(\frac{2\pi z}{N\ell}\right)\geq 0$. It is direct to see that $\lambda_{k}(z)=\lambda_{N\ell-k}(z)$ for $k=1,\dots,N\ell-1$. So, for $k\geq 1$, we need to consider two situations:
\begin{enumerate}[fullwidth, label={(\alph*)}]
\item If $N\ell$ is odd. For $k=1,\dots,(N\ell-1)/2$, $\lambda_{k}(z)\geq 0$ if $k\leq \big\lfloor {-z} \big\rfloor$. Therefore, in this case, $A_z$ admits $\#\{k\in \mathbb{Z}/N\ell: \lambda_k  \geq 0\}=1+2\big\lfloor {-z} \big\rfloor$ non-negative eigenvalues.

\item If $N\ell$ is even. The eigenvalue $\lambda_{N\ell/2}(z)=-1-\cos\left(\frac{2\pi z}{N\ell}\right)<0$ since $z>-N\ell/4$. For $k=1,\dots,(N\ell/2)-1$, $\lambda_{k}(z)\geq 0$ if $k\leq \big\lfloor {-z} \big\rfloor$. Therefore, in this case, $A_z$ also admits $\#\{k\in \mathbb{Z}/N\ell: \lambda_k  \geq 0\}=1+2\big\lfloor {-z} \big\rfloor$ non-negative eigenvalues.
\end{enumerate}

In any case, we have that $A_z$ admits $\#\{k\in \mathbb{Z}/N\ell: \lambda_k  \geq 0\}=1+2\big\lfloor {-z} \big\rfloor$ non-negative eigenvalues.

Therefore, $\mathcal{W}^N_\ell(z)=\{ Q_z\geq 0\}$ is a quadratic cone of index $1+2\big\lfloor {-z} \big\rfloor$. In particular, $\mathcal{W}^N_\ell(z)=\{ Q_z\geq 0\}$ is properly homotopic to a vector space $\bR^{1+2\big\lfloor {-z}\big\rfloor}$.

Now we can describe the non-equivariant structure of $\mathcal{E}_\ell|_{(-\infty,0]}$. Here, we forget its equivariant structure and use the same notation $\mathcal{E}_\ell|_{(-\infty,0]}$. In particular,  $\mathcal{E}_\ell|_{(-\infty,0]}\cong  \mathcal{E}_1|_{(-\infty,0]}$ non-equivariantly.  Consider $\pi_{\underline{q}}: \mathcal{W}_{\ell}^N \rightarrow (-N\ell/4,0]$ for $N\ell$ big enough, it restricts to a proper homotopical fiber bundle with fiber $\bR^{1+2n}$ over each interval $(-n-1,-n]$ for $n\in \bN_{\geq 0}$, and $n+1<N \ell/4$. Therefore, we conclude that ${\mathcal{E}_\ell}|_{(-n-1,-n]} \cong \bK_{(-n-1,-n]}[-1-2n]$. On the other hand, in the non-equivariant derived category, $\bK_{(x,y]}$ and $\bK_{(z,w]}[2]$ has no non-trivial extension if $\bK$ is a field. Therefore, $(\mathcal{E}_\ell)_{(-n-1,-n]}$ has no non-trivial extension for different $n$. In conclusion, we have
\begin{Prop}\label{non-equivariant decomposition}For all fields $\bK$ and for all $\ell\in \bN$, we have the decomposition in the non-equivariant derived category $D( (-\infty,0])$:
\[\mathcal{E}_\ell|_{(-\infty,0]}\cong \bigoplus_{n\in \bN_{\geq 0}}\bK_{(-n-1,-n]}[-1-2n].\]
\end{Prop}

To describe the ${\bZ/\ell}$-action on $\mathcal{W}^N_\ell(z)$, it is better to consider the diagonal form of $Q_z$.

Let $x_k= (q_1,\dots,q_{N\ell})(1,\omega^k,\omega^{2k},\dots,\omega^{(N\ell-1)k})^t\in \bC$, $k\in \bZ/N\ell$. They are coordinates after diagonalization using the discrete Fourier transform. As $\omega$ is a root of unity, we have that $x_k=\overline{x_{N\ell-k}}$. In particular, $x_0$ is a real number. Also recall that $\lambda_{k}(z)=\lambda_{N\ell-k}(z)$. Then the diagonal form of $Q_z$ is
\begin{align}\label{diagonal form definition}
\begin{aligned}
&Q_z(x_0,x_1,\dots,x_{N\ell-1})=\lambda_0(z)x_0^2+\sum_{k=1}^{N\ell-1} \lambda_k(z)|x_k|^2,\\ &(x_0,x_1,\dots,x_{N\ell-1})\in \bR\times \bC^{{N\ell-1}}. \end{aligned} \end{align}
Notice that the discrete Fourier transform that we applied is a complex linear transform, it is easier to work in complex coordinates. However, the constrains $x_k=\overline{x_{N\ell-k}}$ shows that actually we only have half independent complex coordinates, so the real dimension here is still $N\ell$. To our convenience in formulating the action, we still use the complex coordinates. We also need to discuss parity of $N\ell$. Since $N$ is chosen arbitrarily, we can always assume $N$ is odd. Then the parity of $N\ell$ is the parity of $\ell$.

If $\ell$ is odd, then the diagonal form is 
\begin{align}\label{diagonal form of Q}
\begin{aligned}
&Q_z(x_0,x_1,\dots,x_{(N\ell-1)/2})=\lambda_0(z)x_0^2+2\sum_{k=1}^{(N\ell-1)/2} \lambda_k(z)|x_k|^2,\\ &(x_0,x_1,\dots,x_{(N\ell-1)/2})\in \bR\times \bC^{({N\ell-1)/2}}\cong \bR^{N\ell}.         
\end{aligned}
\end{align}
If $\ell$ is even, then the diagonal form is
\begin{align}\label{diagonal form of Q odd}
\begin{aligned}
&Q_z(x_0,x_1,\dots,x_{N\ell/2-1},x_{N\ell/2})=\lambda_0(z)x_0^2+2\sum_{k=1}^{N\ell/2-1} \lambda_k(z)|x_k|^2+\lambda_{N\ell/2}(z)|x_{N\ell/2}|^2,\\ &(x_0,x_1,\dots,x_{N\ell/2-1},x_{N\ell/2})\in \bR\times \bC^{{N\ell/2}-1} \times \bR \cong \bR^{N\ell}.         
\end{aligned}
\end{align}
Now, the action is easier to describe under the diagonal form. By definition of $x_k$, we have 
$x_k=\sum_{i\in \bZ/N\ell} q_{i+1}\omega^{ik}.$
The $\bZ/\ell$-action is given by $(q_i)\mapsto (q_{i-N})$. Then we have
\[x_k=\sum_{i\in \bZ/N\ell} q_{i+1}\omega^{ik}\mapsto \sum_{i\in \bZ/N\ell} q_{i+1-N}\omega^{ik}=\omega^{kN}\sum_{i\in \bZ/N\ell} q_{i+1-N}\omega^{(i-N)k}=\omega^{kN}x_k.\]
Therefore, the ${\bZ/\ell}$-action on the diagonal form is as follows: if we take $\mu=\omega^N$ a primitive $\ell^{th}$ root of unity, then \begin{equation}\label{action formula under diagonal form}
 \mu\cdot(x_k)_k=(\mu^kx_k)_k,   
\end{equation}
where $k=0,1,\dots, N\ell/2-1$ if $\ell$ is odd and $k=0,1,\dots, N\ell/2$ if $\ell$ is even. 

Consequently, the fixed point sets $\left( \mathcal{W}_{\ell}^N (z) \right)^{\bZ/\ell}$ is again a quadratic cone, whose index is $1+2\big\lfloor {-z/\ell} \big\rfloor$. The diagonal $\Delta_{\bR^{N\ell}}$ is given by $\{(x_0,0,\dots,0):x_0\in \bR\}$ in diagonal form, it is a subset of $\left( \mathcal{W}_{\ell}^N (z) \right)^{\bZ/\ell}$.

Finally, we return to the isomorphism of \autoref{non-equivariant decomposition}. Take $z'\leq  z\leq 0$. Since $SS(\mathcal{E}_\ell)\subset \{\zeta\leq 0\}$, the microlocal Morse lemma (\autoref{microlocal morse}) shows that $\textnormal{R}\Gamma(\bR,(\mathcal{E}_\ell)_{[z,0]})\cong (\mathcal{E}_\ell)_{z}$. Then there is a natural morphism
\[(\mathcal{E}_\ell)_{z'}\cong \textnormal{R}\Gamma(\bR,(\mathcal{E}_\ell)_{[z',0]})\rightarrow \textnormal{R}\Gamma(\bR,(\mathcal{E}_\ell)_{[z,0]})\cong (\mathcal{E}_\ell)_{z}.\]
On the other hand, the isomorphism in \autoref{loop sheaf of GF model of rotation 2} shows that $(\mathcal{E}_\ell)_{z}\cong \textnormal{R}\Gamma_c(\mathcal{W}_{\ell}^N (z),\bK)\cong  \textnormal{R}\Gamma_c(\bR^{1+2\lfloor {-z} \rfloor},\bK)$. Then the natural morphism above is given by
\[\textnormal{R}\Gamma_c(\bR^{1+2\lfloor {-z'} \rfloor},\bK) \rightarrow \textnormal{R}\Gamma_c(\bR^{1+2\lfloor {-z} \rfloor},\bK).\]
The decomposition, \autoref{non-equivariant decomposition}, tells us that the natural morphism is $0$ in the non-equivariant category. 

In the equivariant category, the morphism is induced from a vector bundle 
\[\bR^{1+2\lfloor {-z'} \rfloor} \times_{\bZ/\ell} S^\infty \rightarrow \bR^{1+2\lfloor {-z} \rfloor}\times_{\bZ/\ell} S^\infty,\]
which is a lifting of the following vector bundle
\[\bR^{1+2\lfloor {-z'}\rfloor}\times_{S^1} S^\infty \rightarrow \bR^{1+2\lfloor {-z}\rfloor}\times_{S^1} S^\infty\]
via the natural restriction $\bZ/\ell\subset S^1$.

So, in the $S^1$-equivariant derived category, the morphism is given by the mod $\bK$ reduction of the ($\bZ$-coefficient) top Chern class for the second vector bundle, which is $(\lfloor {-z'}\rfloor!/\lfloor {-z}\rfloor!)u^{\lfloor {-z'}\rfloor-\lfloor {-z}\rfloor}\in \textnormal{Ext}^{*}_{S^1}(\bK,\bK)$, which is non-zero. After restricting to the $\bZ/\ell$-equivariant derived category, the morphism is non-zero for a suitable reduction in a finite field $\bK$. For example, we could require $0<\lfloor {-z'} \rfloor<\text{char}(\bK)$ to make sure that the morphism is non-zero. 

{\bf The higher dimension ($d\geq 2$) case: }Now, we start to discuss the higher dimension situation. We already know that $\mathcal{CL}_\ell(\mathcal{T}) \cong   \mathcal{CL}_\ell(\mathcal{S}) ^{\boxstar d}$. Then \autoref{chiu result on loop sheaf} shows that\begin{equation}\label{loop sheaf of torus action}
\mathcal{CL}_\ell(\mathcal{T}) \cong {\mathcal{E}_{\ell}}^{\dboxtimes d}\dboxtimes {\bK}_{\{t\geq 0\}}.
\end{equation}
As the decomposition indicated in \autoref{non-equivariant decomposition}, ${\mathcal{E}_{\ell}}^{\dboxtimes d}|_{\{z\leq 0\}}$ has a decomposition on $\{z\leq 0\}$ indexed by lattice points. Besides, we also have a topological description of ${\mathcal{E}_{\ell}}^{\dboxtimes d}|_{\{z\leq 0\}}$. Let us first discuss the topological description and then state the decomposition. Since we have $d$ copies of $\mathcal{E}_{\ell}$, it is convenient to denote $\bq=(q^1,\dots,q^d)\in \bR^d \eqqcolon V_{\bq}$.
Then \autoref{loop sheaf of GF model of rotation 2} shows us  
\begin{equation}\label{E^d sheaf of torus action}
\left.{\mathcal{E}_{\ell}}^{\dboxtimes d}\right|_{(-N\ell/4,0]^d} \cong {\textnormal{R}}\pi_{\underline{\bq}!}{\bK}_{\prod_{i=1}\mathcal{W}^N_{\ell,i}},   
\end{equation}
where $\mathcal{W}^N_{\ell,i} $ means the $i$th copy of one $\mathcal{W}^N_\ell$, $i\in [d]=\{1,\dots,d\}$,  $\underline{\bq}=(\bq_1,\dots,\bq_{N\ell})$ and $\bq_k=(q_k^1,\dots,q_k^d)$. Let $z=(z_1,\dots,z_d)$, we also define
\begin{align*}
    ^d\mathcal{W}^N_{\ell}\coloneqq &   \prod_{i=1}^d {} \mathcal{W}^N_{\ell,i}
    =\{(z,\bq_1,\dots,\bq_{N\ell})\in (-N\ell/4,0]^d\times V^{N\ell}: Q_{z_i}((q_k^i)_{k\in [N\ell]})\geq 0, \,i\in [d]\},\\
     ^d\mathcal{W}^N_{\ell}(z)\coloneqq&\prod_{i=1}^{d}{}        \mathcal{W}^N_{\ell,i}(z_i)
    =\{(\bq_1,\dots,\bq_{N\ell})\in   V^{N\ell}: Q_{z_i}((q_k^i)_{k\in [N\ell]})\geq 0, \,i\in [d]\}.
\end{align*}
The group ${\bZ/\ell}$ acts on each $\mathcal{W}^N_\ell$ via $(q_k^i)_{k\in [N\ell]}\mapsto (q_{k-N}^i)_{k\in [N\ell]}$. Therefore, ${\bZ/\ell}$ acts diagonally on $^d\mathcal{W}^N_{\ell}$ via $(\bq_k)_{k\in [N\ell]}\mapsto (\bq_{k-N})_{k\in [N\ell]}$. 

The diagonalization applies for each $i\in [d]$, and then on $^d\mathcal{W}^N_{\ell}(z)$. We set $\bx_k=(x^1_k,\dots,x^d_k)$ and $\bx^i=(x^i_1,\dots,x^i_{N\ell})$, then the coordinates of $^d\mathcal{W}^N_{\ell}(z)$ after diagonalization are $(x_k^i)_{i,k}=(\bx_k)_k=(\bx^i)_i$, where $ k=0,1,\dots, (N\ell-1)/2$ if $\ell$ is odd and $k=0,1,\dots, N\ell/2$ if $\ell$ is even.

So for each $z=(z_1,\dots,z_d)\in (-N\ell/4,0]^d$, the space $^d\mathcal{W}^N_{\ell}(z)$ is a product of quadratic cones of indices $1+2\big\lfloor {-z_i}\big\rfloor$ respectively, and then $^d\mathcal{W}^N_{\ell}(z)$ is properly homotopic to a quadratic cone of index $d+2I(z)$, where $I(z)=\sum_{i=1}^d \big\lfloor {-z_i} \big\rfloor$. Therefore, $^d\mathcal{W}^N_{\ell}(z)$ is properly homotopic to $\bR^{d+2I(z)}$ and a refinement of this fact will be proven in \autoref{inductionstep of minimaldegree}. 

The fixed point sets $\left(^d \mathcal{W}_{\ell}^N (z) \right)^{\bZ/\ell}$ is also properly homotopic to a quadratic cone of index $d+2I(z/\ell)$. The diagonal $\Delta_{V^{N\ell}}$ is given by $\{(x_k^i)_{i,k}:\forall i,\,\forall k\neq 0,\,x^i_k=0,\,x_0^i\in \bR\}$ in diagonal form, it is a subset of $\left(^d \mathcal{W}_{\ell}^N (z) \right)^{\bZ/\ell}$.

To be clear, let us set some higher dimensional interval notation. For $x,y\in \bR^d$, we let $(x,y]=\prod_{i=1}^d(x_i,y_i]$ be the half-open cube from $x$ to $y$. We can define half-open cubes $[x,y)$, open cubes $(x,y)$, and closed cubes $[x,y]$ in the same way. Recall that we use $\overline{xy}$ to denote the segment between $x,y$; only when $d=1$, we have $\overline{xy}=[x,y]$. 
Also, recall $O\in \bR^d$ is the origin, and we set $\mathbbm{1}=(1,\dots,1)$ and $e_i=(\delta_{ij})_{j=1}^d$ where $\delta_{ij}$ stands for the Kronecker symbol. 

Then either our topology description of $^d\mathcal{W}^N_{\ell}$ or the decomposition result \autoref{non-equivariant decomposition} shows that 
\begin{Lemma}\label{lattice description}For each $z\leq 0$, we have the equivariant isomorphism\[({\mathcal{E}_{\ell}}^{\dboxtimes d})_{z}\cong \textnormal{R}\Gamma_c(\bR^{d+2I(z)},\bK)\cong \bK[-d-2I(z)].\]
In the non-equivariant derived category, we have a decomposition as follows: 
\begin{align*}
    {\mathcal{E}_{\ell}}^{\dboxtimes d}|_{\{z\leq 0\}} \cong  \bigoplus_{v\in \bN_0^d}\bK_{(-v-\mathbbm{1},-v]}[-d-2I(-v)].
\end{align*}
In the equivariant derived category, for $z',z\in (-\infty,0]^d$, if $z'_i\leq z_i$ for all $i\in [d]$, the natural morphism, 
\[{\mathcal{E}_{\ell}}^{\dboxtimes d}|_{z'}\cong \bK[-d-2I(z')] \rightarrow {\mathcal{E}_{\ell}}^{\dboxtimes d}|_{z}\cong \bK[-d-2I(z)], \]
is induced by the mod-$\bK$ reduction of the top Chern class of the vector bundle \[\bR^{d+2I(z')} \times_{S^1} E S^\infty \rightarrow \bR^{d+2I(z)}\times_{S^1} S^\infty,\]
where $S^1$ acts on $\bR^d$ trivially, and acts on $\bR^{2I(z)}$ by the weight $((1,\dots,\lfloor -z_i\rfloor))_{i\in [d]}.$\end{Lemma}

{\bf Propagation and $\gamma$-topology} Finally, let us describe a propagation phenomena of $\mathcal{E}_{\ell}$. It is simple but crucial for our later application. Notice that, for a given $z\in (-N\ell/4,0]$, the map $z \mapsto \mathcal{W}^N_\ell(z)$ is a decreasing map with respect to the inclusion order. Microlocally, it means that $SS(\mathcal{E}_{\ell})\subset \{\zeta \leq 0\}$, which is already known to us as a general fact from the microsupport estimate (by \autoref{microlocal morse} for example). We have, for $z\leq 0$, that
\[(\mathcal{E}_{\ell})_{z}\cong \textnormal{R}\Gamma_c(\bR, (\mathcal{E}_{\ell})_{[z,0]})\cong \bK[-1-2\big\lfloor {-z} \big\rfloor].\]
In higher dimension, the same thing still happens. 
For $z\in (-\infty,0]^d$, we can compute directly, using \autoref{loop sheaf of GF model of rotation 2}, to see that
\begin{equation}\label{simpler propagating iso}
(\mathcal{E}_{\ell}^{\dboxtimes d})_{z}\cong \textnormal{R}\Gamma_c(\bR^d_z, (\mathcal{E}^{\dboxtimes d}_{\ell})_{[z,O]})\cong \bK[-d-2I(z)].    
\end{equation}

However, as $SS(\mathcal{E}_{\ell}^{\dboxtimes d}) \subset \bR^d_z\times (-\infty,0]^d_\zeta$ and $[z,O]=(\{z\}+[0,\infty)^d)\cap (-\infty,0]^d$, the isomorphisms \eqref{simpler propagating iso} can also be obtained pure microlocally.

For a closed proper convex cone $\gamma\subset \bR^d$, we can consider the $\gamma$-topology on $\bR^d$. We refer to \cite[Section 3.5, Section 5.2]{KS90} and \cite{KS2018} for more about the definition and sheaf theory related to $\gamma$-topology. A closed subset $Z\subset \bR^d$ is $\gamma$-closed if $Z=Z-\gamma$. Now, consider the induced topology of the $\gamma$-topology on $\gamma$. Then the notation $\Sigma_\gamma=(\Sigma-\gamma)\cap \gamma$ is exactly the closure of the $\gamma$-topology for a closed set $\Sigma\subset \gamma$. So, for a closed subset $\Sigma\subset \gamma$, we say $\Sigma_\gamma$ the {\em $\gamma$-closure} of $\Sigma$ and we say $\Sigma$ is {\em $\gamma$-closed} if $\Sigma_\gamma=\Sigma$.

Now, let us take a sheaf $F\in D(\bR^d)$ satisfying $SS(F)\subset \bR^d\times (-\gamma)$. Then we claim that if $\Sigma$ is compact and convex, we have that\begin{equation}\label{gamma closed cohomology}
 \textnormal{R}\Gamma_c(\bR^d_z,F_{\Sigma})\cong \textnormal{R}\Gamma_c(\bR^d_z,F_{\Sigma_\gamma}).   
\end{equation}

We can give a {\em proof} of $\eqref{gamma closed cohomology}$ as follows: The microsupport  $SS(F)\subset \bR^d\times (-\gamma)$ together with the microlocal cut-off lemma \cite[Proposition 5.2.3]{KS90} shows that $F_{\gamma}$ is a $-\gamma$-sheaf\  on $\bR^d_{z}$, i.e. $F_{\gamma}$ is pullbacked from a sheaf on $\bR^d_{z}$ equipped with the $-\gamma$-topology. Then its global section over $\Sigma$ is isomorphic to the global section over the $\gamma$-closure $\Sigma_\gamma$ by \cite[Proposition 3.5]{KS90}.

Now, as $SS(\mathcal{E}_{\ell}^{\dboxtimes d}) \subset \bR^d_z\times (-\infty,0]^d_\zeta$, we can take $\gamma=(-\infty,0]^d$, which is a proper convex cone. So, we can talk about $\Sigma_\gamma$ for a closed subset $\Sigma\subset \gamma$. For example, $\{z\}_\gamma=[z,O]$ for $z\in \gamma$. Then we apply \eqref{gamma closed cohomology} to $\mathcal{E}_{\ell}^{\dboxtimes d}$ to obtain the first isomorphism of \eqref{simpler propagating iso} and the stronger result: for a compact and convex set $\Sigma$, we have
\begin{equation}\label{stronger propagating iso}
\textnormal{R}\Gamma_c(\bR^d_z, (\mathcal{E}_{\ell})_{\Sigma}).
\cong \textnormal{R}\Gamma_c(\bR^d_z, (\mathcal{E}_{\ell})_{\Sigma_\gamma}).
\end{equation}

\subsection{Proof of \autoref{structure toric domain module}}In this subsection, we will prove the structure theorem. 

{\bf Idea and sketch of the proof: }
We present $(F_{\ell}(X_{\Omega},\bK))_T$ as ${\textnormal{R}}\Gamma_c\left(\bR^d, ({\mathcal{E}_{\ell}}^{\dboxtimes d} )_{\Omega_T^\circ}   \right)$, to which we can apply the results in \autoref{Cohomology sheaf section}. Now, consider the inclusion sequence $\{O\}\subset \overline{ZO}\subset \Omega_T^\circ$, then we have a commutative diagram.
\begin{center}
 \begin{tikzcd}
{{\textnormal{R}}\Gamma_c\left(\bR^d, ({\mathcal{E}_{\ell}}^{\dboxtimes d} )_{\Omega_T^\circ}   \right)} \arrow[rd] \arrow[r] & {{\textnormal{R}}\Gamma_c\left(\bR^d, ({\mathcal{E}_{\ell}}^{\dboxtimes d} )_{\overline{ZO}}   \right)} \arrow[d] \arrow[r, "\cong"] & {\bK[-d-2I(Z)]} \arrow[d, "k_Z u^{I(Z)}"] \\
                                                                                                                             & {{\textnormal{R}}\Gamma_c\left(\bR^d, ({\mathcal{E}_{\ell}}^{\dboxtimes d} )_{{O}}  \right)} \arrow[r, "\cong"]                       & {\bK[-d]}                                   
\end{tikzcd}
\end{center}
By definition, the inclined arrow composed with the bottom isomorphism gives the fundamental class, and we call the upper horizontal arrow (up to constant) $\Lambda_{Z,\ell}$. \autoref{lattice description} shows that (up to a constant $k_Z$) the vertical morphism is $u^{I(Z)}$. Eventually, we absorb the constant into $\Lambda_{Z,\ell}$ since the constant is uniquely determined by $Z$ and $\ell$. The commutative diagram induces a decomposition $\eta^{\bZ/\ell}_T(X_{\Omega},\bK)=u^{I(Z)}\Lambda_{Z,\ell}$. In particular, the presence of $\Lambda_{Z,\ell}$ shows us the minimal cohomology degree is smaller than $-2I(Z)$ for all $Z\in \Omega_T^\circ$.

To achieve the non-torsioness, we need to prove that the fundamental class $\eta^{\bZ/\ell}_T(U,\bK)$, a degree $0$ morphism, is non-zero. We have two approaches. The easiest one is to take a small ball $B\subset U$, and then we apply the computation for balls (which can be derived directly from \autoref{lattice description}). The harder one is that we study its cocone, which is computed by homology of a union of finite dimensional manifolds. 

I will discuss the harder approach since it provides us with more structural results, for example, rank and degree distribution of torsion elements. We will argue by a localization trick. In particular, we show that $H^*C^{\bZ/\ell}_T(X_\Omega,\bF_{p_\ell})$ is a finitely generated module over $\bF_{p_\ell}[u]$ whose free part is of rank $2$. Then the argument also shows that torsion cannot happen in non-negative degrees. 

Finally, we study further the cocone of the fundamental class to show that the minimal cohomology degree is greater than $-2I(\Omega_T^\circ)$.

Therefore, our technical discussion will focus on the formula for  ${\textnormal{R}}\Gamma_c\left(\bR^d, ({\mathcal{E}_{\ell}}^{\dboxtimes d} )_{W}   \right)$ for a locally closed set $W\subset \Omega_T^\circ$, and its minimal degree estimate. We will organize our arguments in the following way:
\begin{itemize}[fullwidth]
\item We first compute $(F_{\ell}(X_{\Omega},\bK))_T$ using its isomorphism with ${\textnormal{R}}\Gamma_c\left(\bR^d, ({\mathcal{E}_{\ell}}^{\dboxtimes d} )_{\Omega_T^\circ}   \right)$, where ${\mathcal{E}_{\ell}}^{\dboxtimes d}$ is discussed in the last subsection. Consequently, we derive a similar formula for the cocone of the fundamental class, i.e. $ {\textnormal{R}}\Gamma_c\left(\bR^d, ({\mathcal{E}_{\ell}}^{\dboxtimes d} )_{\Omega_T^\circ \setminus O}   \right)$. 
Then the result of the last subsection will reduce them to a cohomology of a topological space ${\mathcal{W}_{\ell}^N}(\Omega_T^\circ)$ (see \eqref{definition of U(C)} later for its definition). We will achieve the targets in \autoref{cohomology of E^d}.

\item Recall the lattice decomposition (\autoref{non-equivariant decomposition}) of the sheaf $\mathcal{E}^{\dboxtimes d}_\ell$. We hope to utilize the lattice description to obtain a minimal degree estimate for the cocone of the fundamental class. A problem here is that we are computing cohomology of sheaves over $\Omega_T^\circ$, while $\Omega_T^\circ$ is usually curved. So, our idea is to decompose $\Omega_T^\circ$ into ``almost cubes'', which are introduced in \autoref{cornerlemme}. Next, we will study the proper homotopy type of ${\mathcal{W}_{\ell}^N}(\Omega_T^\circ)$ in the case that $\Omega_T^\circ$ is an almost cube. This is \autoref{inductionstep of minimaldegree}. 

\item Finally, we use the computation for almost cubes as an induction step to obtain the minimal degree estimate in general. This is done using the Mayer–Vietoris sequence in \autoref{minimaldegree}. After that, we will finish the proof of \autoref{structure toric domain module}.

\begin{RMK} A technical fact is that in the induction process of the minimal degree estimate, we have to deal with some sets that are not necessarily convex. However they are $\gamma$-closed. So, we will present the result for $\gamma$-closed set $\Sigma$, not only $\Omega_T^\circ$, from the beginning in the following.
\end{RMK}
\end{itemize}

{\bf Preliminary lemmas: }For a convex toric domain $X_\Omega$, by the discussion following \eqref{GFmodel convex case}, we have ${\Omega^\circ}\subset \gamma^d\times [0,\infty)$. 
Then \eqref{reformulation of F} and \eqref{loop sheaf of torus action} show that
\begin{align}\label{F of convex toric domains 2}
\begin{aligned}
F_{\ell}(X_{\Omega},\bK)&\cong {\textnormal{R}}\pi_{{z}!}{\textnormal{R}}s_{t!}^2\left( {\mathcal{E}_{\ell}}^{\dboxtimes d}\dboxtimes {\bK}_{\{t_1\geq 0\}} \dotimes  \pi_{t_1}^{-1}  {\bK}_{\Omega^{\circ}}  \right)
  \cong {\textnormal{R}}\pi_{{z}!} \left[ ({\mathcal{E}_{\ell}}^{\dboxtimes d} \dboxtimes {\bK}_{\{t\geq 0\}})_{ \Omega^\circ}\right].
\end{aligned}
 \end{align}
Therefore, we conclude that
\begin{equation}\label{F of convex toric domains}
(F_{\ell}(X_{\Omega},\bK))_T\cong {\textnormal{R}}\Gamma_c\left(\bR^d, ({\mathcal{E}_{\ell}}^{\dboxtimes d} )_{\Omega_T^\circ}   \right).\end{equation}
In particular, for $X_{\bR^d}=T^*V$, we have 
\[(F_{\ell}(T^*V,\bK))_T\cong {\textnormal{R}}\Gamma_c\left(\bR^d, ({\mathcal{E}_{\ell}}^{\dboxtimes d} )_{O}   \right)\cong \bK[-d].\]
Then, by definition, the fundamental class is
\[ {\textnormal{R}}\Gamma_c\left(\bR^d, ({\mathcal{E}_{\ell}}^{\dboxtimes d} )_{\Omega_T^\circ}   \right)\rightarrow  {\textnormal{R}}\Gamma_c\left(\bR^d, ({\mathcal{E}_{\ell}}^{\dboxtimes d} )_{O}   \right)\cong \bK[-d].\]
For $Z\in {\Omega_T^\circ}$, we apply \eqref{stronger propagating iso} for the segment $\Sigma=\overline{ZO}$, then we have (recall that $[Z,O]$ denotes a cube here)
\[ \textnormal{R}\Gamma_c(\bR^d_z,(\mathcal{E}_{\ell}^{\dboxtimes\ell})_{\overline{ZO}})\cong \textnormal{R}\Gamma_c(\bR^d_z,(\mathcal{E}_{\ell}^{\dboxtimes\ell})_{[Z,O]})\cong \bK[-d-2I(Z)],\]
since $[Z,O]=\overline{ZO}_\gamma$.
Now, we can embed the fundamental class into an excision triangle: 
\[{\textnormal{R}}\Gamma_c\left(\bR^d, ({\mathcal{E}_{\ell}}^{\dboxtimes d} )_{{\Omega}^\circ_T\setminus O }   \right) \rightarrow {\textnormal{R}}\Gamma_c\left(\bR^d , ({\mathcal{E}_{\ell}}^{\dboxtimes d} )_{{\Omega}^\circ_T} \right)\xrightarrow{\eta^{\bZ/\ell}_T(X_\Omega,\bK)} {\textnormal{R}}\Gamma_c\left(\bR^d, ({\mathcal{E}_{\ell}}^{\dboxtimes d})_{O}   \right)  \xrightarrow{+1}.\]

Both ${\Omega_T^\circ}$ and ${\overline{ZO}}$ are compact convex. We would like to apply the isomorphism \eqref{E^d sheaf of torus action} to compute the cohomology of $\mathcal{E}_{\ell}^{\dboxtimes d}$ in term of $^d\mathcal{W}^N_\ell$. 

{\bf Assumption:} For any compact subset $\Sigma \subset \gamma$, we will fix an odd integer $N=N(\Sigma)>0$ and a positive number $\varepsilon>0$ such that $\Sigma \subset [-N\ell/4-\varepsilon,0]^d$. The existence of $N$ and $\varepsilon$ is ensured by the compactness of $\Sigma$.

\begin{Lemma}\label{cohomology of E^d}For a compact set $\Sigma \subset \gamma$ such that $\Sigma\cap [x,y]$ is empty or contractible for all $x\leq y,\,x,y\in \gamma$ (recall here, $[x,y]$ means the closed cube from $x$ to $y$). We have
\begin{equation}\label{section of E^d}
{\textnormal{R}}\Gamma_c\left(\bR^d, ({\mathcal{E}_{\ell}}^{\dboxtimes d})_{\Sigma}   \right) \cong  {\textnormal{R}}\Gamma_c\left(\mathcal{W}_\ell^N(\Sigma), {\bK}    \right),    
\end{equation}
where 
\begin{equation}\label{definition of U(C)}
{\mathcal{W}_\ell^N}(\Sigma)={\bigcup_{{z} \in \Sigma}}  {^d}\mathcal{W}^N_\ell(z)=\pi_z \left({^d}\mathcal{W}^N_\ell \cap (\Sigma \times V^{N\ell})\right).
\end{equation} 
As $\Sigma=\Omega_T^\circ$ is convex for $T \geq 0$, we have, in particular \begin{align}\label{stalks of F of toric domain}
\begin{aligned}
&(F_{\ell}(X_{\Omega},\bK))_T\cong {\textnormal{R}}\Gamma_c\left(\bR^d, ({\mathcal{E}_{\ell}}^{\dboxtimes d} )_{\Omega_T^\circ}   \right)\cong  {\textnormal{R}}\Gamma_c\left({\mathcal{W}_{\ell}^N}(\Omega_T^\circ), {\bK}    \right),\\
& {\textnormal{R}}\Gamma_c\left(\bR^d, ({\mathcal{E}_{\ell}}^{\dboxtimes d} )_{\Omega_T^\circ\setminus O )}   \right)\cong  {\textnormal{R}}\Gamma_c\left({\mathcal{W}_{\ell}^N}(\Omega_T^\circ)\setminus \Delta_{V^{N\ell}}, {\bK}    \right).\\
\end{aligned}
 \end{align}
\end{Lemma}
\begin{proof}For $N=N(\Sigma)>0$ and $\varepsilon>0$ such that $\Sigma \subset [-N\ell/4-\varepsilon,0]^d$, we have the isomorphism \eqref{E^d sheaf of torus action} ${\mathcal{E}_{\ell}}^{\dboxtimes d}|_{[-N\ell/4-\varepsilon,0]^d} \cong {\textnormal{R}}\pi_{\underline{\bq}!}{\bK}_{^d\mathcal{W}^N_\ell \cap ([-N\ell/4-\varepsilon,0]^d\times V^{N\ell})} $, and then we obtain
\begin{align*}
   {\textnormal{R}}\Gamma_c\left(\bR^d, ({\mathcal{E}_{\ell}}^{\dboxtimes d})_{\Sigma}    \right) \cong  {\textnormal{R}}\pi_{ {z}!} \left( ({\textnormal{R}}\pi_{\underline{\bq}!}{\bK}_{^d\mathcal{W}^N_\ell})_{\Sigma}\right)
   \cong {\textnormal{R}}\pi_{ {z}!}  {\textnormal{R}}\pi_{\underline{\bq}!}{\bK}_{^d\mathcal{W}^N_\ell\cap(\Sigma\times V^{N\ell}  )}
   \cong {\textnormal{R}}\pi_{\underline{\bq}!}{\textnormal{R}}\pi_{ {z}!}  {\bK}_{^d\mathcal{W}^N_\ell\cap(\Sigma\times V^{N\ell} )}.
\end{align*}
Claim: When restricted to $^d\mathcal{W}^N_\ell\cap(\Sigma\times V^{N\ell} )$, the fiber of $\pi_{ {z}}$ is compact and contractible if it is non-empty. Indeed, Chiu  proved, in the Lemma 4.10 of \cite{chiu2017}, that the fibers of the restriction of $\pi_{z_i}$ on $\mathcal{W}^N_\ell \cap ([-N\ell/4-\varepsilon,0]\times \bR^{N\ell})$ are closed intervals. So the fibers of the restriction of $\pi_{z}$ on $^d\mathcal{W}^N_\ell \cap ([-N\ell/4-\varepsilon,0]\times \bR^{N\ell})$ are closed cubes. Hence, the fibers of the restriction of $\pi_{z}$ on $^d\mathcal{W}^N_\ell\cap(\Sigma\times V^{N\ell} )$ are intersections of closed cubes and $\Sigma$, which are either empty or compact and contractible by assumption. 

Consequently, the Vietoris-Begel theorem implies
\[{\textnormal{R}}\pi_{{z}!}  {\bK}_{^d\mathcal{W}^N_\ell\cap(\Sigma\times V^{N\ell} )}\cong {\bK}_{\pi_{{z}}(^d\mathcal{W}^N_\ell\cap(\Sigma\times V^{N\ell}  ))}={\bK}_{{\mathcal{W}_{\ell}^N}(\Sigma)}.\]
Therefore, ${\textnormal{R}}\Gamma_c\left(\bR^d, ({\mathcal{E}_{\ell}}^{\dboxtimes d} )_\Sigma   \right)={\textnormal{R}}\pi_{\underline{\bq}!}({\bK}_{{\mathcal{W}_{\ell}^N}(\Sigma)})\cong  {\textnormal{R}}\Gamma_c\left({\mathcal{W}_{\ell}^N}(\Sigma), {\bK}    \right).$

The statements involve $\Sigma=\Omega_T^\circ$ follow from the discussion above the lemma.
\end{proof}

\begin{RMK}\label{rmkgammatopology}The condition in the lemma is true for compact convex sets $\Sigma$. For our last applications, we need to, in adition, consider $\gamma$-closed sets for $\gamma=(-\infty,0]^d$. For a closed set $\Sigma\subset \gamma$, the $\gamma$-closure is defined as $\Sigma_\gamma=(\Sigma-\gamma)\cap \gamma$. We say $\Sigma$ is {\em $\gamma$-closed} if $\Sigma_\gamma=\Sigma$. For example, $\gamma\setminus(\mathring{\gamma}+z)$ is $\gamma$-closed for $z\in \gamma $, and the intersection of two $\gamma$-closed sets is $\gamma$-closed. The $\gamma$-closed sets satisfy the condition of \autoref{cohomology of E^d}. Indeed, for a closed cube $[x,y]$ with $x,y\in \gamma$, a $\gamma$-closed $\Sigma$, and any $z\in \Sigma \cap [x,y]$, we have $\overline{xz}\subset \Sigma \cap [x,y]$. Therefore, $\Sigma \cap [x,y]$ is star-shaped and then contractible. 

As $z\mapsto \mathcal{W}(z)$ is a decreasing map, one can see that
${\mathcal{W}_{\ell}^N}(\Sigma)={\mathcal{W}_{\ell}^N}(\Sigma_\gamma)$
for all compact subset $\Sigma\subset \gamma$. In particular, if $\Sigma$ satisfies the condition of \autoref{cohomology of E^d}, then the lemma implies that
\[{\textnormal{R}}\Gamma_c\left(\bR^d , ({\mathcal{E}_{\ell}}^{\dboxtimes d} )_{\Sigma_\gamma} \right)\xrightarrow{\cong}{\textnormal{R}}\Gamma_c\left(\bR^d , ({\mathcal{E}_{\ell}}^{\dboxtimes d} )_{\Sigma} \right),\]
which can be seen as a generalization of \eqref{stronger propagating iso} (which is only true for compact convex subsets) for compact sets satisfying the condition of \autoref{cohomology of E^d}. Later, we will mainly focus on $\gamma$-closed sets $\Sigma$.
\end{RMK}

Now, to understand the cohomology of ${\mathcal{W}_{\ell}^N}(\Sigma)$ (see \eqref{definition of U(C)}), we start from a special case that $\Sigma$ is an ``almost closed cube'', which will be defined in
\autoref{cornerlemme}. Let us recall some notation and introduce some new ones.

First, recall that, for $x,y\in \bR^d$, we let $(x,y]=\prod_{i=1}^d(x_i,y_i]$ be the half-open cube from $x$ to $y$. Similarly, we define open cubes and closed cubes in this way. Also, recall $O\in \bR^d$ is the origin, and we set $\mathbbm{1}=(1,\dots,1)$ and $e_i=(\delta_{ij})_{j=1}^d$ where $\delta_{ij}$ stands for the Kronecker symbol. For simplicity, we also denote $C_x=[x,x+\mathbbm{1})$ for $x\in \bR^d$.  

Next, for a compact $\gamma$-closed set $\Sigma \subset \gamma$, we set
\begin{align*}
J_\Sigma &=(-\Sigma) \cap \bZ^d_{\geq 0}=\lbrace v\in \bZ^d_{\geq 0}: (-\Sigma)\cap C_v \neq \varnothing \rbrace , \\
\partial J_\Sigma &=\{v\in J_\Sigma: \forall i, v+e_i \not\in J_\Sigma\}= \{v\in J_\Sigma: (-\Sigma)\cap (\overline{C_v}\setminus C_v)=\varnothing\} .     
\end{align*}
The compactness of $\Sigma$ shows that both $J_\Sigma$ and $\partial J_\Sigma$ are finite sets.

\begin{Lemma}\label{cornerlemme} Let  $\Sigma \subset \gamma$ be a compact $\gamma$-closed set. Then $\partial J_\Sigma=\{v\}$ for some $v\in \bZ^d_{\geq 0}$ if and only if $[O,v] \subset -\Sigma \subset  [O,v+\mathbbm{1})$ for the same $v\in \bZ^d_{\geq 0}$. 

We say that $\Sigma$ is an {\em almost cube} if it satisfies these equivalent conditions.
\end{Lemma}
\begin{proof}When $[O,v] \subset -\Sigma \subset  [O,v+\mathbbm{1})$, taking the intersection with $\overline{C_w}\setminus C_w$ for all $w\in J_\Sigma$, we obtain\[ [O,v] \cap (\overline{C_w}\setminus C_w)\subset (-\Sigma) \cap(\overline{C_w}\setminus C_w) \subset  [O,v+\mathbbm{1}) \cap (\overline{C_w}\setminus C_w).  \]
Then we can obtain $\partial J_\Sigma=\{v\}$ from that $[O,v] \cap (\overline{C_w}\setminus C_w)=\varnothing$ only when $v=w$.

Conversely, when $\partial J_\Sigma=\{v\}$, we have $-v\in \Sigma$. So $\Sigma=\Sigma_\gamma$ implies $[-v,O]=\{-v\}_\gamma \subset \Sigma$. Now, suppose $-\Sigma \nsubseteq [O,v+\mathbbm{1})$, then there is a $z\in \Sigma$ such that $-z_i=v_i+1$ for some $i\in [d]$. Therefore, $v+e_i\in J_\Sigma$. If $v+e_i\notin \partial J_\Sigma$, the argument repeats and there exists another $j\in[d]$ such that $v+e_i+e_{j}\in J_{\Sigma}$. We can continue until we obtain a index set $I$ (with possible multiplicities) such that $v+ \sum_{I} e_i\in \partial J_\Sigma$. Since $J_\Sigma$ is a finite set, then the index set must be finite. However, $\partial J_\Sigma=\{v\}$, then $v+ \sum_{I} e_i \notin  \partial J_\Sigma$. Hence we get a contradiction. Then $-\Sigma \subset [O,v+\mathbbm{1})$.\end{proof}
Here, we are going to prove a refinement of the fact that $^d \mathcal{W}^N_{\ell}(-v)$ is properly homotopic to $\bR^{d+2I(-v)}$ as noticed before \autoref{lattice description}.
\begin{Lemma}\label{inductionstep of minimaldegree}For a compact $\gamma $-closed set $\Sigma \subset \gamma$ with $\partial J_\Sigma=\{v\}$, i.e. $\Sigma$ is an almost cube, the subspace $\bR^{d}\times \bC^{I(-v)} $ is a strong deformation retract of ${\mathcal{W}_{\ell}^N}(\Sigma)$ under a proper deformation retraction. Moreover, $\Delta_{V^{N\ell}}\cong \bR^d\times\{0\} \subset \bR^{d}\times \bC^{I(-v)}$ is invariant under the retraction.
\end{Lemma}
\begin{proof}Here, we use the diagonal form of $Q_z$ we introduced in \eqref{diagonal form definition}. Then the coordinate system on $\left(\bR\times \bC^{\frac{N\ell-1}{2}}\right)^d$ is $(x_k^i)_{i,k}=(\bx_k)_{k}=(\bx^i)_{i}$ with $\bx_k=(x^1_k,\dots,x^d_k)$ and $\bx^i=(x^i_0,\dots,x^i_k)$, where $i\in [d]=\{1,\dots,d\}$, $ k=0,1,\dots (N\ell-1)/2$ if $\ell$ is odd and $k=0,1,\dots N\ell/2$ if $\ell$ is even. For shortness, we only deal with the $\ell$ odd case. The $\ell$ even case has the same proof with minor corrections on the notation. Recall that
\begin{align*}\label{definition of V}
& {\mathcal{W}_{\ell}^N}(\Sigma)=\left\lbrace (\bx^i)_{i}=(x_{0}^i,x_{k}^i)_{i,k}: \exists z\in \Sigma, \forall i,\, Q_{z_i}(\bx^i)\geq 0  \right\rbrace, \\
& \Delta_{V^{N\ell}}=\left\lbrace (\bx^i)_{i}=(x_{0}^i,x_{k}^i)_{i,k}: \forall k  \geq 1, i\in [d],\text{ such that } x_{k}^i=0   \right\rbrace.
 \end{align*} 
For $0\leq m \leq (N\ell-1)/2$, $i\in [d]$, consider $h_m^i:\bR\times \bC^{\frac{N\ell-1}{2}}\times [0,1] \rightarrow \bR\times \bC^{\frac{N\ell-1}{2}}$, 
\[h_{m}^i(x_0^i,x_+^i,x_-^i,t)=h^i_{{m},t}(x^i_0,x^i_+,x^i_-)=(x^i_0,x^i_+,t x^i_-),\]
where $x^i_+=(x^i_1,\dots,x^i_{m})$, $x^i_-=(x^i_{{m}+1},x^i_{{m}+2},\dots,x^i_{(N\ell-1)/2})$.

By assumption of the lemma, $v\in -\Sigma\subset [0, N\ell/4)^d$. Then we have $0\leq v_i<N\ell/4 \leq (N\ell-1)/2$.
Now, define $H_v:\left(\bR\times \bC^{\frac{N\ell-1}{2}}\right)^d \times [0,1] \rightarrow \left(\bR\times \bC^{\frac{N\ell-1}{2}}\right)^d$ by
\[H_{v,t}=h_{v_1,t}^1\times \cdots \times h^d_{v_d,t}. \]
Then we have $H_{v,1}$ is the identity map. Next, we have the following:
\begin{itemize}[fullwidth]
    \item $H_{v,t}({\mathcal{W}_{\ell}^N}(\Sigma))\subset {\mathcal{W}_{\ell}^N}(\Sigma)$. 
Indeed, $(\bx^i)_i\in {\mathcal{W}_{\ell}^N}(\Sigma)$ implies there exists $z\in \Sigma$ such that for all $i\in [d]$, we have $Q_{z_i}(\bx^i) \geq 0$. So, in the diagonal form \eqref{diagonal form of Q}, we have for all $i\in [d]$, 
\[\lambda_0(z_i)|x_{0}^i|^2+2\sum_{k=1}^{v_i}\lambda_k(z_i)|x_{k}^i|^2 \geq 2\sum_{k\geq v_i+1}(-\lambda_k(z_i))|x_{k}^i|^2.\]
Now $-\Sigma \subset  [O,v+\mathbbm{1})$ implies that $z_i <v_{i}+1$ for all $i\in [d]$, hence $\lambda_k(z_i)<0$ for $k\geq v_i+1$ and for all $i\in [d]$. So \[\lambda_0(z_i)|x_{0}^i|^2+2\sum_{k=1}^{v_i}\lambda_k(z_i)|x_{k}^i|^2 \geq 2\sum_{k\geq v_i+1}(-\lambda_k(z_i))|x_{k}^i|^2\geq 2t^2\sum_{k\geq v_i+1}(-\lambda_k(z_i))|x_{k}^i|^2,\]
i.e., $Q_{z_i}(h_{v_i,t}(\bx^i))\geq 0$ for all $i\in [d]$. Hence $H_{v,t}(\bx^1,\dots,\bx^d)\in {\mathcal{W}_{\ell}^N}(\Sigma)$.

\item $\left.H_v\right|_{{\mathcal{W}_{\ell}^N}(\Sigma)}$ is proper. Indeed, taking $(\bx^i)_i\in {\mathcal{W}_{\ell}^N}(\Sigma)$ such that $H_v((\bx^i)_i)\in [-M,M]^d$, we have that  $\sum_{k=0}^{v_i}|x_{k}^i|^2+\sum_{k\geq v_i+1}|tx_{k}^i|^2 \leq M$, for all $i\in [d]$.
    
Obviously, $\sum_{k=0}^{v_i}|x_{k}^i|^2 \leq M$, for all $i\in [d]$, and
\begin{align*}
2\max_{\substack{k=0,\dots,v_i\\z\in \Sigma}}|\lambda_k(z_i)|M&\geq \lambda_0(z_i)|x_{0}^i|^2+2\sum_{k=0}^{v_i}\lambda_k(z_i)|x_{k}^i|^2 \\ &\geq 2\sum_{k\geq v_i+1}(-\lambda_k(z_i))|x_{k}^i|^2 \\ &\geq 2\min_{\substack{k\geq v_i+1\\z\in \Sigma}}|\lambda_k(z_i)|\sum_{k\geq v_i+1}|x_{k}^i|^2.   
\end{align*}

Since $\lambda_k(z_i)<0$ for $k\geq v_i+1$, and $z\in \Sigma$, we have $\min_{\substack{k\geq v_i+1\\z\in \Sigma}}|\lambda_k(z_i)|>0$. Consequently,
\[\sum_{k\geq v_i+1}|x_{k}^i|^2 \leq \frac{\max\limits_{\substack{k=1,\dots,v_i\\z\in \Sigma}}|\lambda_k(z_i)|}{\min\limits_{\substack{k\geq v_i+1\\z\in \Sigma}}|\lambda_k(z_i)|} M \eqqcolon KM.\]
It means that $\sum_{k=0}^{v_i}|x_{k}^i|^2+\sum_{k\geq v_i+1}|x_{k}^i|^2 \leq (1+K)M$, for all $i\in [d]$, where $K=K(\Sigma)$ is a constant only depending on ${\mathcal{W}_{\ell}^N}(\Sigma)$. 
    
So, we have shown that the pre-image of a bounded set under $\left.H_v\right|_{{\mathcal{W}_{\ell}^N}(\Sigma)}$ is bounded. It means that $\left.H_v\right|_{{\mathcal{W}_{\ell}^N}(\Sigma)}$ is proper. 
    
Hence $\left.H_v\right|_{{\mathcal{W}_{\ell}^N}(\Sigma)}$ is a proper homotopy with $\left.H_{v,1}\right|_{{\mathcal{W}_{\ell}^N}(\Sigma)}=\text{Id}_{{\mathcal{W}_{\ell}^N}(\Sigma)}$
\item $\bR^d\times \bC^{I(-v)} \times \{0\} \subset {\mathcal{W}_{\ell}^N}(\Sigma)$. Let $(\bx^i)_i\in \bR^d\times \bC^{I(-v)} \times \{0\}$. This means that for all $i\in [d]$, $\bx^i=(x_0^i,x_+^i,x_-^i)$ satisfies $x^i_-=0$. 
  Since $z=-v\in \Sigma$, by assumption, it is enough to check that $Q_{-v_i}(\bx^i) \geq 0$ for all $i\in [d]$. Now $\lambda_k(-v_i)\geq 0$ for $k=0,1,\dots, v_i$. Then for all $i\in [d]$, we have
\[Q_{-v_i}(x_0^i,x^i_+,0)=\lambda_0(-v_i)|x_{0}^i|^2+2\sum_{k=1}^{v_i}\lambda_k(-v_i)|x_{k}^i|^2 \geq 0.\] 
So $(\bx^i)_i\in{\mathcal{W}_{\ell}^N}(\Sigma)$ and then $\bR^d\times \bC^{I(-v)} \times \{0\} \subset {\mathcal{W}_{\ell}^N}(\Sigma)$. 
\item $H_{v,0}({\mathcal{W}_{\ell}^N}(\Sigma))\subset \bR^d\times \bC^{I(-v)} \times \{0\}$. Indeed for $(\bx^i)_i\in {\mathcal{W}_{\ell}^N}(\Sigma)$, we have  $h_{v_i,0}^i(\bx^i)=(x^i_0,x^i_+,0)$ for all $i\in [d]$. Then $H_{v,0}({\mathcal{W}_{\ell}^N}(\Sigma))\subset \bR^d\times \bC^{I(-v)} \times \{0\}$. 

On the other hand, by definition of $h_{v_i}^i$, we have \[h_{v_i}^i(x_0^i,x_+^i,0,t)=(x_0^i,x_+^i,t0)=(x_0^i,x_+^i,0).\]
So $H_{v,t}|_{\bR^d\times \bC^{I(-v)} \times \{0\}}=\text{Id}_{\bR^d\times \bC^{I(-v)} \times \{0\}}$ for all $t\in [0,1]$. Therefore, $\bR^d\times \bC^{I(-v)} \times \{0\}$ is a proper strong deformation retract of ${\mathcal{W}_{\ell}^N}(\Sigma)$ under $H_{v,t}|_{{\mathcal{W}_{\ell}^N}(\Sigma)}$.
\end{itemize}\end{proof}

Below, we will frequently use the equivariant global section functor ${\textnormal{R}}\Gamma_c\left(\bR^d, ({\mathcal{E}_{\ell}}^{\dboxtimes d} )_W   \right)$ for locally closed set $W\subset \gamma$. Then we denote it by
\begin{equation}\label{simpler notation global section}
\Gamma\mathcal{E}(W)\coloneqq {\textnormal{R}}\Gamma_c\left(\bR^d, ({\mathcal{E}_{\ell}}^{\dboxtimes d} )_{W}   \right),    
\end{equation}
for shortening the length of notation until the end of the subsection.

\begin{Lemma}\label{minimaldegree} Let $\Sigma \subset \gamma$ be a compact $\gamma $-closed set such that $\Sigma\subset (-{p_\ell\mathbbm{1}},O]$. Recall the notation $I(\Sigma)$ at \eqref{definition of I and infty}. Then
\[
\textnormal{Ext}^{q-d}_{\bZ/\ell}\left(\Gamma\mathcal{E}({\Sigma\setminus O } ),\bF_{p_\ell} \right) \cong 0
\qquad\text{if }q\notin [-2I(\Sigma),-1] .
\]
The $\bF_{p_\ell}$-vector space $\textnormal{Ext}^{q-d}_{\bZ/\ell}\left(\Gamma\mathcal{E}({\Sigma\setminus O } ),\bF_{p_\ell} \right) $ is finite dimensional.
\end{Lemma}
\begin{proof} 
We proceed by induction on $|J_\Sigma|$. We notice that the maximum $I(\Sigma)$ can be achieved by some $v$ since $\Sigma\cap \bZ^d$ is finite. Moreover, if $v\in J_\Sigma$ satisfies $I(-v)=I(\Sigma)$, then $v\in \partial J_\Sigma$. 
We will use the excision distinguished triangle
\begin{equation}\label{excision sequence toric}
\Gamma\mathcal{E}({\Sigma\setminus O } ) \rightarrow \Gamma\mathcal{E}({\Sigma } )\xrightarrow{\eta_{\Sigma}}\Gamma\mathcal{E}({ O } ) \xrightarrow{+1}.    
\end{equation}

If $|J_\Sigma|=1$, that is $J_\Sigma=\{0\}$, then 
\autoref{inductionstep of minimaldegree} shows that $\eta_{\Sigma}$ is an isomorphism in the derived category. Then  $\Gamma\mathcal{E}({\Sigma\setminus O } ) \cong 0$ by \eqref{excision sequence toric} and the result follows.

Now, we suppose the result is true for all $\Sigma'$ such that $|J_{\Sigma'}|< |J_\Sigma|$ and we
distinguish the cases $|\partial J_\Sigma| = 1$ and $|\partial J_\Sigma| > 1$.

\begin{enumerate}[fullwidth]
    \item\label{1stpossibility}If $\partial J_\Sigma = \{v\}$ is a singleton, i.e. $\Sigma$ is an almost cube. The case $v=0$ is already done and we assume $I(-v)>0$. Then the excision sequence \eqref{excision sequence toric}, \autoref{cohomology of E^d} and \autoref{inductionstep of minimaldegree} together show the
    isomorphisms in the equivariant derived category
    \[\Gamma\mathcal{E}({\Sigma\setminus O } ) \cong {\textnormal{R}}\Gamma_c({\mathcal{W}_{\ell}^N}(\Sigma)\setminus \Delta_{V^{N\ell}},\bF_{p_\ell})\cong {\textnormal{R}}\Gamma(S^{2I(-v)-1},\bF_{p_\ell})[-d-1],\] 
where the action of $\bZ/\ell$ on $S^{2I(-v)-1}$ is given in \eqref{action formula under diagonal form}. We have
\begin{align*}
 \text{Ext}^{*-d}_{\bZ/\ell}\left(\Gamma\mathcal{E}({\Sigma\setminus O } ),\bF_{p_\ell} \right) &\cong\text{Ext}^{*+1}_{\bZ/\ell}\left({\textnormal{R}}\Gamma(S^{2I(-v)-1},\bF_{p_\ell}),\bF_{p_\ell}\right)\\
 &\cong \text{Ext}^{*+1}\left({(\bF_{p_\ell})}_{S^{2I(-v)-1}},\omega_{S^{2I(-v)-1}}^!\right)\\
 &\cong  H_{-*-1}^{\bZ/\ell}(S^{2I(-v)-1},{\bF_{p_\ell}}),
\end{align*}
where we used the equivariant Poincar\'e duality, which holds since $S^{2I(-v)-1}$ is compact and orientable.

Under the assumption $\Sigma\subset (-{p_\ell\mathbbm{1}},O]$, the $\bZ/\ell$-action is free by \eqref{action formula under diagonal form}. Hence $H_{-*-1}^{\bZ/\ell}(S^{2I(-v)-1},{\bF_{p_\ell}})$ computes the usual cohomology of the quotient $Q_{\bZ/\ell}=S^{2I(-v)-1}/(\bZ/\ell)$, which is the lens space of dimension $2I(-v)-1$.

Then, we have
\begin{equation*}
H_q(Q_{\bZ/\ell})=\begin{cases} {\bF_{p_\ell}},&q\in [0,2I(-v)-1],\\
0, &q\notin [0,2I(-v)-1].
\end{cases}
\end{equation*}
Converting to cohomology degree, we obtain: $\text{Ext}^{*-d}_{\bZ/\ell}\left(\Gamma\mathcal{E}({\Sigma\setminus O } ),{\bF_{p_\ell}} \right)$ is concentrated in $[-2I(-v),-1]$ and finite dimensional.

The proof of this part is independent of our induction, so it can be applied to the second case.
\item If $|\partial J_\Sigma|\geq 2$, take $v\in \partial J_\Sigma$ such that $I(-v)=I(\Sigma)$. Then we can take $1>\delta>0$ such that $\Sigma\cap (\gamma+(\delta \mathbbm{1}-v)) \subset (-v,0]$. 
This is possible due to the compactness of $\Sigma$. Let us define:
\begin{align}\label{construcionofAB}
\begin{aligned}
    &A=[\Sigma\cap (\gamma+(\delta \mathbbm{1}-v))]_\gamma,\\
    &B=\Sigma\cap[\gamma\setminus (\mathring{\gamma} +(\delta \mathbbm{1}-v))].
\end{aligned}
\end{align}
Then we have a closed covering $\Sigma=A\cup B$. Moreover, both $A$ and $B$ are compact $\gamma$-closed sets, then so is $A\cap B$ (see \autoref{ConstructionofC}). 
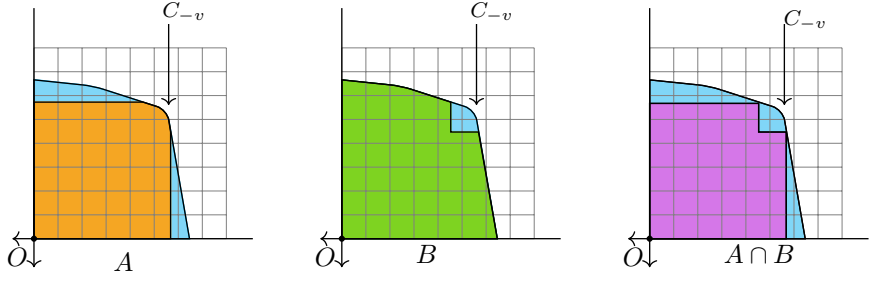
\begin{figure}[htbp]
\begin{center}

    \tikzset{every picture/.style={line width=0.5pt}} 
\begin{tikzpicture}[x=0.75pt,y=0.75pt,yscale=-0.8,xscale=0.8]


\draw  [fill=cyan  ,fill opacity=0.5, cyan, xshift=-2.625,yshift=-3.75] (48.5,70) {[rounded corners]--(85.5,74) -- (131.5,89) }-- (145.5,170) -- (48.5,170) --  cycle ;
\draw  [fill={rgb, 255:red, 245; green, 166; blue, 35 }  ,fill opacity=1, color={rgb, 255:red, 245; green, 166; blue, 35 }, xshift=-2.625,yshift=-3.75 ] (116.5,84.13) {[rounded corners]--(131.5,88.93)} -- (133.7,102.93) -- (133.7,170) -- (48.5,170) -- (48.5,84.13) --  cycle ;
\draw[step=15, help lines](45,165) grid (165,45);
\draw[xshift=-2.625,yshift=-3.75]  (48.5,70.13) {[rounded corners]--(85.5,74) -- (131.5,88.93) }-- (145.5,170) -- (48.5,170) --  cycle ;
\draw[xshift=-2.625,yshift=-3.75]  (116.5,84.13) {[rounded corners]--(131.5,88.93)} -- (133.7,102.93) -- (133.7,170) -- (48.5,170) -- (48.5,84.13) --  cycle ;
\draw[->,xshift=-2.625,yshift=-3.75]    (132.5,35) -- (132.5,86) ;
\draw[xshift=-2.625,yshift=-3.75] (127,20) node [anchor=north west][inner sep=0.75pt]  [font=\footnotesize]  {$C_{-v}$};

\draw[->,xshift=-2.625,yshift=-3.75]    (48.5,25.0) -- (48.5,188) ;
\draw[->,xshift=-2.625,yshift=-3.75]    (185,170) -- (35,170) ;
\draw [fill=black,xshift=-2.625,yshift=-3.75](48.5,170) circle (1.5) ;
\draw[xshift=-2.625,yshift=-3.75] (30,174) node [anchor=north west][inner sep=0.75pt]    {$O$};

\draw (93,172.5) node [anchor=north west][inner sep=0.75pt]    {$A$};


\draw  [fill=cyan  ,fill opacity=0.5 ,xshift=141.375,yshift=-3.75, cyan] (48.5,70.13) {[rounded corners]--(85.5,74) -- (131.5,88.93) }-- (145.5,170) -- (48.5,170) --  cycle ;
\draw  [fill={rgb, 255:red, 126; green, 211; blue, 33 }  ,fill opacity=1, color={rgb, 255:red, 126; green, 211; blue, 33 },xshift=-2.675,yshift=-3.75 ] (308.5,102.93) -- (325.8,102.93) -- (337.5,170) -- (240.6,170) -- (240.6,70.13) {[rounded corners]-- (277.5,74)} -- (308.5,84.13) -- cycle ;
\draw[step=15, help lines,xshift=144](45,165) grid (165,45);
\draw  [xshift=141.375,yshift=-3.75] (48.5,70.13) {[rounded corners]--(85.5,74) -- (131.5,88.93) }-- (145.5,170) -- (48.5,170) --  cycle ;
\draw  [xshift=-2.675,yshift=-3.75 ] (308.5,102.93) -- (325.8,102.93) -- (337.5,170) -- (240.6,170) -- (240.6,70.13) {[rounded corners]-- (277.5,74)} -- (308.5,84.13) -- cycle ;

\draw[->,xshift=141.375,yshift=-3.75]    (132.5,35) -- (132.5,86) ;
\draw[xshift=141.375,yshift=-3.75] (127,20) node [anchor=north west][inner sep=0.75pt]  [font=\footnotesize]  {$C_{-v}$};

\draw[->,xshift=141.375,yshift=-3.75]    (48.5,25.0) -- (48.5,188) ;
\draw[->,xshift=141.375,yshift=-3.75]    (185,170) -- (35,170) ;
\draw [fill=black,xshift=141.375,yshift=-3.75](48.5,170) circle (1.5) ;
\draw[xshift=141.375,yshift=-3.75] (30,174) node [anchor=north west][inner sep=0.75pt]    {$O$};

\draw[xshift=141.375,yshift=-3.75] (93,172.5) node [anchor=north west][inner sep=0.75pt]    {$B$};


\draw  [fill=cyan  ,fill opacity=0.5 ,xshift=285.375,yshift=-3.75] (48.5,70.13) {[rounded corners]--(85.5,74) -- (131.5,88.93) }-- (145.5,170) -- (48.5,170) --  cycle ;
\draw  [fill={rgb, 255:red, 213; green, 119; blue, 232 }  ,fill opacity=1 ,xshift=-2.675,yshift=-3.75] (500.5,102.93) -- (517.7,102.93) -- (517.7,170) -- (432.5,170) -- (432.5,133.5) -- (432.5,84.93) -- (500.5,84.93) -- cycle ;
\draw[step=15, help lines, xshift=288](45,165) grid (165,45);
\draw  [xshift=285.375,yshift=-3.75] (48.5,70.13) {[rounded corners]--(85.5,74) -- (131.5,88.93) }-- (145.5,170) -- (48.5,170) --  cycle ;
\draw[xshift=-2.625,yshift=-3.75]   (500.5,102.93) -- (517.7,102.93) -- (517.7,170) -- (432.5,170) -- (432.5,133.5) -- (432.5,84.93) -- (500.5,84.93) -- cycle ;

\draw[->,xshift=285.375,yshift=-3.75]    (132.5,35) -- (132.5,86) ;
\draw[xshift=288] (127,20) node [anchor=north west][inner sep=0.75pt]  [font=\footnotesize]  {$C_{-v}$};

\draw[->,xshift=285.375,yshift=-3.75]    (48.5,25.0) -- (48.5,188) ;
\draw[->,xshift=285.375,yshift=-3.75]    (185,170) -- (35,170) ;
\draw [fill=black,xshift=285.375,yshift=-3.75](48.5,170) circle (1.5) ;
\draw[xshift=285.375,yshift=-3.75] (30,174) node [anchor=north west][inner sep=0.75pt]    {$O$};

\draw[xshift=285.375,yshift=-3.75] (93,172.5) node [anchor=north west][inner sep=0.75pt]    {$A\cap B$};

\end{tikzpicture}
  \caption{The picture illustrate the construction of $A,B$. $\Sigma$ is the background blue set.}
    \label{ConstructionofC}
    \end{center}
\end{figure}

Then we have the Mayer-Vietoris triangle, 
\[\Gamma\mathcal{E}({\Sigma\setminus O } )\rightarrow \Gamma\mathcal{E}({A\setminus O } ) \oplus \Gamma\mathcal{E}({B\setminus O } )  \rightarrow \Gamma\mathcal{E}({(A\cap B)\setminus O } ) \xrightarrow{+1}.\]
Next, we apply the $\text{Ext}^{*-d}_{\bZ/\ell}\left(-,{\bF_{p_\ell}}\right)\cong\text{Ext}^{*}_{\bZ/\ell}\left(-,{\bF_{p_\ell}}[-d]\right)$ to obtain a long exact sequence
\begin{align}\label{exact sequence}
    \text{Ext}^{*-d}_{\bZ/\ell}\left(\Gamma\mathcal{E}({(A\cap B)\setminus O } ),{\bF_{p_\ell}}\right)\rightarrow\,
    \begin{gathered}
        \text{Ext}^{*-d}_{\bZ/\ell}\left(\Gamma\mathcal{E}({A\setminus O } ),{\bF_{p_\ell}}\right)\\\bigoplus\\ \text{Ext}^{*-d}_{\bZ/\ell}\left(\Gamma\mathcal{E}({B\setminus O } ),{\bF_{p_\ell}}\right)
    \end{gathered}
    \rightarrow  \text{Ext}^{*-d}_{\bZ/\ell}\left(\Gamma\mathcal{E}({\Sigma\setminus O } ),{\bF_{p_\ell}}\right)\xrightarrow{+1}.
\end{align}

By our construction \eqref{construcionofAB}, we have
\begin{itemize}[fullwidth]
    \item $|\partial J_A|=1$. Hence we can apply the result of (1). So that $\text{Ext}^{*-d}_{\bZ/\ell}\left(\Gamma\mathcal{E}(A\setminus O  ),{\bF_{p_\ell}}\right)$ is concentrated in $[-2I(A),-1]\subset [-2I(\Sigma),-1]$.
    \item $|J_B|<|J_\Sigma|$. We can use the induction hypothesis, hence $\text{Ext}^{*-d}_{\bZ/\ell}\left(\Gamma\mathcal{E}({B \setminus O } ),{\bF_{p_\ell}}\right)$ is concentrated in $[-2I(B),-1]\subset [-2I(\Sigma),-1]$.
    \item $|J_{A\cap B}|<| J_\Sigma|$, since $J_{A\cap B} \subset J_A$ but $v\notin J_{A\cap B}$. Then we can use the induction hypothesis, that $\text{Ext}^{*-d}_{\bZ/\ell}\left(\Gamma\mathcal{E}({(A\cap B)\setminus O } ),{\bF_{p_\ell}}\right)$ is concentrated in $[-2I(A\cap B),-1]$. Moreover, in $J_A$, $v$ is the only lattice point such that $I(-v)=I(\Sigma)$, then for all $v'\in J_{A\cap B}\subset J_A\setminus \{v\}$, we have $I(-v')<I (\Sigma)$. Then $|I(A\cap B)| < I (\Sigma)$, and $[-2I(A\cap B),-1]\subset [-2I(\Sigma)+2,-1]$.
\end{itemize}
\end{enumerate}

Therefore, it follows from the long exact sequence \eqref{exact sequence} that $\text{Ext}^{*-d}_{\bZ/\ell}\left(\Gamma\mathcal{E}({\Sigma\setminus O } ),{\bF_{p_\ell}}\right)$ is concentrated in $[-2I(\Sigma),-1]$ and finite dimensional.
\end{proof}

Now, we are in the position to prove \autoref{structure toric domain module}. 
\begin{proof}[{\bf Proof of \autoref{structure toric domain module}}]
The equation \eqref{stalks of F of toric domain} says that $(F_{\ell}(X_{\Omega},\bF_{p_\ell}))_T \cong \Gamma\mathcal{E}(\Omega_T^\circ).$
Now, consider the inclusion sequence $\{O\}\subset \overline{ZO}\subset \Omega_T^\circ$ of closed sets. Then we have a commutative diagram:
\begin{center}
\begin{tikzcd}
\Gamma\mathcal{E}(\Omega_T^\circ) \arrow[rd] \arrow[r] & \Gamma\mathcal{E}(\overline{ZO}) \arrow[d] \arrow[r, "\cong"] & {\bF_{p_\ell}[-d-2I(Z)]} \arrow[d, "k_Z u^{I(Z)}"] \\
                                                       & \Gamma\mathcal{E}(O) \arrow[r, "\cong"]                       & {\bF_{p_\ell}[-d]}                                
\end{tikzcd}
\end{center}
By definition, the inclined arrow compose with the bottom isomorphism gives the fundamental class $\eta^{\bZ/\ell}_T(X_\Omega,\bF_{p_\ell})$. The terms $\Gamma\mathcal{E}(\overline{ZO})$ and $\Gamma\mathcal{E}({O})$ are computed using \autoref{lattice description} and \eqref{stronger propagating iso}. \autoref{lattice description} also shows that the vertical morphism is $k_Zu^{I(Z)}$, where $k_Z$ is a constant only depends on $Z$. We absorb the constant into the horizontal arrow, then we call it $\Lambda_{Z,\ell}$. Therefore, the commutative diagram induces a decomposition $\eta^{\bZ/\ell}_T(X_{\Omega},\bF_{p_\ell})=u^{I(Z)}\Lambda_{Z,\ell}$.

Now, let us embed the fundamental class $\eta^{\bZ/\ell}_T(X_\Omega,\bF_{p_\ell})$ into the excision triangle (the triangle \eqref{excision sequence toric} for $\Sigma={\Omega}^\circ_T$)
\begin{equation}\label{dt1}
\Gamma\mathcal{E}({{\Omega}^\circ_T\setminus O} ) \rightarrow \Gamma\mathcal{E}({{\Omega}^\circ_T} )\xrightarrow{\eta^{\bZ/\ell}_T(X_\Omega,\bF_{p_\ell})}\Gamma\mathcal{E}({ O } ) \xrightarrow{+1}.    
\end{equation}

So, after applying $\RHOM_{\bZ/\ell}(-,\bF_{p_\ell}[-d])$, we get the distinguished triangle:
\begin{equation}\label{tautological dt for X_Omega}
 \textnormal{R}\Gamma(V,\omega_V^{\bZ/\ell})\rightarrow C^{\bZ/\ell}_T(X_\Omega,\bF_{p_\ell}) \xrightarrow{\RHOM_{\bZ/\ell}(\eta^{\bZ/\ell}_T(X_{\Omega},\bF_{p_\ell}),\bF_{p_\ell}[-d])} \RHOM_{\bZ/\ell}(\Gamma\mathcal{E}({{\Omega}^\circ_T\setminus O} ),\bF_{p_\ell}[-d])\xrightarrow{+1}.
\end{equation}
Here $V\cong \bR^d$ and it is equipped with the trivial $\bZ/\ell$-action. Taking cohomology for the distinguished triangle, we get a long exact sequence of the Chiu-Tamarkin cohomology:
\begin{equation}\label{long exact seq for X_Omega}
 H^*_{\bZ/\ell}(V,\bF_{p_\ell})\rightarrow H^*C^{\bZ/\ell}_T(X_\Omega,\bF_{p_\ell}) \xrightarrow{\textnormal{Ext}_{\bZ/\ell}^{*-d}(\eta^{\bZ/\ell}_T(X_{\Omega},\bF_{p_\ell}),\bF_{p_\ell})} \textnormal{Ext}_{\bZ/\ell}^{*-d}(\Gamma\mathcal{E}({{\Omega}^\circ_T\setminus O} ),\bF_{p_\ell})\xrightarrow{+1}.
\end{equation}
When $0\leq T <p_\ell/\|\Omega^\circ_1\|_{\infty}$, we have $\Omega^\circ_T\subset (-p_\ell\mathbbm{1},O]$. 

Then we can apply \autoref{minimaldegree},  $ \textnormal{Ext}_{\bZ/\ell}^{*-d}(\Gamma\mathcal{E}({{\Omega}^\circ_T\setminus O} ),\bF_{p_\ell})$ is a finite dimensional graded $\bF_{p_\ell}$ vector space which is concentrated in degrees $[-2I(\Omega^\circ_T),-1]$. Then, it is torsion as a $\bF_{p_\ell}[u]$-module.

One the other hand, $H^*_{\bZ/\ell}(V,\bF_{p_\ell})\cong \textnormal{Ext}_{\bZ/\ell}^*(\bF_{p_\ell}[-d],\bF_{p_\ell}[-d]) \cong {\bF_{p_\ell}}[u,\theta]$ is concentrated in $[0,\infty)$. 

Therefore, after tensoring with $\bF_{p_\ell}((u))$, $\textnormal{Ext}_{\bZ/\ell}^{*-d}(\eta^{\bZ/\ell}_T(X_{\Omega},\bF_{p_\ell}),\bF_{p_\ell})\otimes_{\bF_{p_\ell}[u]}\bF_{p_\ell}((u))$ is an isomorphism of $\bF_{p_\ell}((u))$-vector spaces. Then, we conclude that $\textnormal{Ext}_{\bZ/\ell}^{*-d}(\eta^{\bZ/\ell}_T(X_{\Omega},\bF_{p_\ell}),\bF_{p_\ell})\neq 0$ and so $\eta^{\bZ/\ell}_T(X_{\Omega},\bF_{p_\ell})\neq 0$. Moreover, we have that $H^*C^{\bZ/\ell}_T(X_\Omega,\bF_{p_\ell})$ is a finitely generated $\bF_{p_\ell}[u]$ module whose rank is 2 and the torsion part of $H^*C^{\bZ/\ell}_T(X_\Omega,\bF_{p_\ell})$ is $\textnormal{Ext}_{\bZ/\ell}^{*-d}(\Gamma\mathcal{E}({{\Omega}^\circ_T\setminus O} ),\bF_{p_\ell})$. 

Now, for the $\bF_{p_\ell}[u]$ module $H^*C^{\bZ/\ell}_T(X_\Omega,\bF_{p_\ell})$, its free part is concentrated in $[0,\infty)$ and its torsion part are concentrated in degrees $[-2I(\Omega^\circ_T),-1]$. Then the minimal degree of $H^*C^{\bZ/\ell}_T(X_\Omega,\bF_{p_\ell})$ is at least $-2I(\Omega^\circ_T)$, and torsion elements of $H^*C^{\bZ/\ell}_T(X_\Omega,\bF_{p_\ell})$ only appear in degree $[-2I(\Omega^\circ_T),-1]$.

On the other hand, this estimate is sharp. Indeed, we take $Z\in \Omega_T^\circ$ such that $I(Z)=I(\Omega^\circ_T)$. Then the decomposition $\eta^{\bZ/\ell}_T(X_{\Omega},\bF_{p_\ell})=u^{I(Z)}\Lambda_{Z,\ell}$ shows that we have a degree equation: $0=2|I(Z)|+|\Lambda_{Z,\ell}|$. Then $|\Lambda_{Z,\ell}|=-2I(\Omega^\circ_T)$ realizes the minimal degree $-2I(\Omega^\circ_T)$. 

For the ellipsoid $E(a)=X_{\Omega_{E(a)}}$ (see \autoref{example toric domains}-(2)), let $Z=(-T/a_1,\dots,-T/a_d)$, then $\left(\Omega_{E(a)}\right)^\circ_T=\overline{ZO}$ is a segment. So, we can compute $H^*C^{\bZ/\ell}_T(X_\Omega,\bF_{p_\ell})$ directly from \autoref{lattice description}, and $\Lambda_{Z,\ell}$ we defined above induces an isomorphism of $A$ module. So $H^*C^{\bZ/\ell}_T(X_\Omega,\bF_{p_\ell})$ is torsion free as a $\bF_{p_\ell}[u]$-module. \end{proof}

\section{Contact Invariants}\label{section: contact inv}
I will explain how the Chiu-Tamarkin complex works for the contact geometry of (contact) admissible open sets in the prequantized space $T^*X\times S^1$. 

For any open set $\mathscr{U}\subset T^*X\times S^1$, we can lift it to a $\bZ$-invariant set $\widetilde{\mathscr{U}}\subset J^1X$ in the sense $T'_{k}(\widetilde{\mathscr{U}})=\widetilde{\mathscr{U}}$, where $T_k'(\bq,\bp,t)=(\bq,\bp,t+k)$ for $k\in \bZ$. In this way, we can discuss sheaves microsupported in $J^1X\setminus \Tilde{\mathscr{U}}$. Then $\mathcal{D}_{J^1X\setminus \widetilde{\mathscr{U}}}(X)$ and its left semi-orthogonal complement are all well-formulated. Specifically, for $\mathscr{Z}=J^1X\setminus \widetilde{\mathscr{U}}$, we define 
\begin{align*}
    \cD^c_\mathscr{Z}(X)=&\{F\in \cD(X): \mu s_L(F) \subset \mathscr{Z}\},\\
    \cD^c_{\mathscr{U}}(X)=&{^\perp}\cD^c_\mathscr{Z}(X), \text{ the left orthogonal complement of }\cD^c_\mathscr{Z}(X).
\end{align*}
Same as the symplectic case, we can define the notion of admissibility and microlocal kernels. To make it compatible with the Hamiltonian action of contact isotopy as we discussed in \autoref{GKSsection}, we will use composition functors rather than convolution functors. On the other hand, in the symplectic case, we require that microlocal kernels are objects in the Tamarkin category. Now, we need a (2-variable) variant version of the Tamarkin category for contact microlocal kernels. Let $\mathscr{D}(X^2)$ be the full triangulated subcategory $\{F\in D(X^2\times \bR^2): 
F \circ \bK_{\{t_2\geq t_1\}}\xrightarrow{\cong} F\}$ of $D(X^2\times \bR^2)$. Then we define 
\begin{Def}\label{contactdefadmissibledomains} We say ${\mathscr{U}}$ is {\em $\bK$-admissible} if there is a distinguished triangle 
\[ \sP_{\mathscr{U}}\rightarrow {\bK}_{\Delta_{X^2} \times\{t_2\geq t_1\}} \rightarrow \sQ_{\mathscr{U}} \xrightarrow{+1},\]
in $\mathscr{D}(X^2)$ such that the composition functor $\circ \sP_{\mathscr{U}} $ is right adjoint to $\mathcal{D}^c_{{\mathscr{U}}}(X) 
\hookrightarrow \mathcal{D}(X) $ and $\circ \sQ_{\mathscr{U}}$
is left adjoint to  $\mathcal{D}^c_{\mathscr{Z}}(X) 
\hookrightarrow \mathcal{D}(X) $, i.e., 
\[\mathcal{D}^c_{\mathscr{Z}}(X)  \xleftarrow{\circ \sQ_{\mathscr{U}}}\mathcal{D}(X) \xrightarrow{\circ \sP_{\mathscr{U}}}\mathcal{D}^c_{{\mathscr{U}}}(X),    \]
are two microlocal projectors.  

Such a pair of sheaves $(\sP_{\mathscr{U}},\sQ_{\mathscr{U}})$ together with the distinguished triangle give an orthogonal decomposition of $\mathcal{D}(X)$ by \autoref{semiorthcomplement}. We call the pair $(\sP_{\mathscr{U}},\sQ_{\mathscr{U}})$ {\em microlocal kernels} associated with ${\mathscr{U}}\subset T^*X\times S^1$, and the distinguished triangle as the defining triangle of ${\mathscr{U}}$. 

We say ${\mathscr{U}}$ is {\em admissible} if ${\mathscr{U}}$ is $\bZ$-admissible.
\end{Def}
The uniqueness and functoriality has the same proof, just need to replace convolution by composition. We have the existence of kernels for the prequantized open set $U\times S^1$ where $U\subset T^*X$ is a symplectic admissible open set. Precisely, we have the following proposition.

\begin{Prop}\label{product kernel in contact case}If $U\subset T^*X$ is (symplectic) admissible by the following distinguished triangle:
\[ P_U\rightarrow {\bK}_{\Delta_{X^2} \times\{t\geq 0\}} \rightarrow Q_U \xrightarrow{+1}.\]
Then $U\times S^1\subset T^*X\times S^1$ is (contact) admissible by the following distinguished triangle:
\[ \sP_{U\times S^1}\rightarrow {\bK}_{\Delta_{X^2} \times\{t_2\geq t_1\}} \rightarrow \sQ_{U\times S^1} \xrightarrow{+1},\]
where $\sP_{U\times S^1}=m^{-1}P_U$, $\sQ_{U\times S^1}=m^{-1}Q_U$ and $m(t_1,t_2)=t_2-t_1$. 
\end{Prop}
Notice that we have ${\bK}_{\Delta_{X^2} \times\{t_2\geq t_1\}}=m^{-1}{\bK}_{\Delta_{X^2} \times[0,\infty)}$.
\begin{proof}The second distinguished triangle comes from applying $m^{-1}$ to the first one and we have $m^{-1}F\in \mathscr{D}(X^2)$ for $F\in \cD(X^2)$. On the other hand, as we mentioned in (1) of \autoref{remark convolution and composition comparision}, we have
\[F\star P_U \cong F\circ \sP_{U\times S^1},\quad F\star Q_U\cong F\circ \sQ_{U\times S^1},\]
for $F\in \cD(X)$. Finally, as $\widetilde{U\times S^1}=U\times \bR$, we have that $\mu s_L(F)\subset J^1X\setminus \widetilde{U\times S^1}$ if and only if $\mu s (F)\subset T^*X\setminus U$. Then the result follows.
\end{proof}
Now, we can define the contact Chiu-Tamarkin complex for admissible open sets ${\mathscr{U}}\subset T^*X\times S^1$. As in the symplectic case, let us introduce the adjoint pair first:
\begin{center}
    \begin{tikzcd}
F\in D_{\bZ/\ell}((X^2\times \bR_t^2)^\ell) \arrow[rr, "\alpha^c_{\ell,T,X}", shift left] &  & D_{\bZ/\ell}(\pt) \ni G, \arrow[ll, "\beta^c_{\ell,T,X}"]
\end{tikzcd}
\end{center}
defined by:
\begin{align}\label{definition of adjoint loop functor contact}
\begin{aligned}
    &\alpha^c_{\ell,n,X}(F)=(i_{n}^\ell)^{-1}{\textnormal{R}}\pi_{\underline{\bq}!}   \Tilde{\Delta}_X^{-1}{\textnormal{R}}\widetilde{m}_!\left(F\right),\\
    &\beta^c_{\ell,n,X}(G)=\widetilde{m}^{!}\Tilde{\Delta}_{X*} \pi_{\underline{\bq}}^!i_{n*}^\ell G[-1], 
\end{aligned}
\end{align}
where
\begin{align*}
   &\widetilde{m}: (X^{2}\times \bR^{2})^\ell\rightarrow X^{2\ell}\times \bR^{\ell},\\
   &\widetilde{m}(\underline{\bq},t^1_1,t^2_1,\dots,t^1_\ell,t^2_\ell)=(\underline{\bq},t^2_\ell-t^1_1,t^2_1-t^1_2,\dots,t^2_{\ell-1}-t^1_{\ell});\\
    &\widetilde{\Delta}_X: X^\ell \times \bR^\ell\rightarrow X^{2\ell} \times \bR^\ell,\\
    &\widetilde{\Delta}_X(\bq_1,\dots,\bq_\ell,\underline{t})=(\bq_\ell,\bq_1,\bq_1,\dots,\bq_{\ell-1},\bq_{\ell-1},\bq_n,\underline{t});\\
    &\pi_{\underline{\bq}}: X^\ell \times \bR^\ell\rightarrow \bR^\ell;\\
    &i_{n}^\ell(\pt)=(n,\dots,n)\in \bR^\ell,
\end{align*}
where $\underline{\bq}=(\bq_1,\dots,\bq_\ell)$ and $\underline{t}=(t_1,\dots,t_\ell)$.
\begin{Def}With the notation above, for $\ell \in \bN$ and $n\in \bN_0$, we define the {\em contact Chiu-Tamarkin complex} as follows:
\begin{align*}
    \mathscr{C}^{\bZ/\ell}_{n\ell}({\mathscr{U}},\bK)&=\RHOM_{\bZ/\ell}\left(\alpha^c_{\ell,n,X}(\sP_{\mathscr{U}}^{\dboxtimes \ell}),\bK[-d]  \right)\\
    &\cong\RHOM_{\bZ/\ell}\left(\sP_{\mathscr{U}}^{\dboxtimes \ell},\beta^c_{\ell,n,X}\bK[-d]  \right).
\end{align*}
\end{Def}
Compare to the symplectic case, the parameter $T$ is replaced by a discrete parameter $T=n\ell$. First, let us compare $\mathscr{C}^{\bZ/\ell}_{n\ell}(U\times S^1,\bK)$ and $C^{\bZ/\ell}_{n\ell}(U,\bK )$ if $U\subset T^*X$ is symplectic admissible. By \autoref{product kernel in contact case}, the prequantized open set $U\times S^1$ is contact admissible.
\begin{Prop}\label{contact sympl comparision}For a symplectic admissible open set $U\subset T^*X$, for $\ell\in \bN$, $n\in\bN_0$, we have
\[C^{\bZ/\ell}_{n\ell}(U,\bK )\cong \mathscr{C}^{\bZ/\ell}_{n\ell}(U\times S^1,\bK).\]
\end{Prop}
\begin{proof}Since $\sP_{U\times S^1}\cong m^{-1}P_U$, we have
\[\sP_{U\times S^1}^{\dboxtimes \ell}\cong (m^{\ell})^{-1}P_U^\dboxtimes,\]
where $m^{\ell}(\underline{\bq},t^1_1,t^2_1,\dots,t^1_\ell,t^2_\ell)=(\underline{\bq},t^2_1-t^1_1,\dots,t_\ell^{2}-t_\ell^{1})$. Then we have 
\[\mathscr{C}^{\bZ/\ell}_{n\ell}(U\times S^1,\bK)\cong \RHOM_{\bZ/\ell}\left(P_U^{\dboxtimes \ell},m^\ell_*\beta^c_{\ell,n,X}\bK[-d]\right).\]
So, we only need to verify that 
\[m^\ell_*\beta^c_{\ell,n,X}\bK\cong \beta_{\ell,n\ell,X}\bK.\]
By proper base change, we only need to assume $X=\pt$ and then show that $m_*^\ell\widetilde{m}^!i_{n*}^\ell\bK[-1]\cong s_t^{\ell !}i_{n\ell*}\bK$.
Direct computation shows that both sides are isomorphic to $\bK_{\{(t_1,\cdots,t_\ell):t_1+\cdots+t_\ell=n\ell\}}[\ell-1]$.
\end{proof}

On the other hand, the constraint $T/\ell\in \bN_0$ is adapt to the problem of invariance. As the lifting of a contact isotopy is merely $\bZ$-equivariant, the sheaf quantization will only be $\bZ$-equivariant (see \autoref{contact isotopy quantization rmk}). So our discussion on invariance for symplectic version does not applies directly. However a slight modification for the proof of the symplectic invariance works.
\begin{Thm}[ {\cite[Theorem 4.7]{chiu2017}}]\label{contactinvariance1} Let ${\mathscr{U}},{\mathscr{U}}_1,{\mathscr{U}}_2$ be contact admissible open sets and let ${\mathscr{U}}_1\xhookrightarrow{i} {\mathscr{U}}_2$ be an inclusion. Then one has, for $\ell\in \bN$, $n\in \bN_0$,
\begin{enumerate}[fullwidth]
    \item There is a morphism $\mathscr{C}^{\bZ/\ell}_{n\ell}({\mathscr{U}}_2,\bK) \xrightarrow{i^*} \mathscr{C}^{\bZ/\ell}_{n\ell}({\mathscr{U}}_1,\bK)$, which is natural with respect to inclusions of admissible open sets.
    \item For a compactly supported contact isotopy $\varphi:I\times T^*X\times S^1 \rightarrow T^*X\times S^1$,. We have an isomorphism, in the equivariant category, $\Phi^{\bZ/\ell,c}_{z,n\ell}:\sC^{\bZ/\ell}_{n\ell}({\mathscr{U}},\bK) \xrightarrow{\cong} \sC^{\bZ/\ell}_{n\ell}\left(\varphi_z({\mathscr{U}}),\bK\right)$, for all $z\in I$. The isomorphism $\Phi^{\bZ/\ell,c}_{z,n\ell}$ is functorial with respect to restriction morphisms in (1). When ${\mathscr{U}}=T^*X\times S^1$, we have $\Phi^{\bZ/\ell,c}_{z,n\ell}=\id$.
\end{enumerate}
\end{Thm}
The proof for (1) is the same as the symplectic case. Let us present the proof for invariance, which is slightly different from the symplectic one. 
\begin{proof}[Proof of {\autoref{contactinvariance1} (2)}]For the contact isotopy $\varphi$, we take the GKS quantization $K(\widehat{\varphi'})$ as we discussed in \autoref{GKSsection}. Let $K=K(\widehat{\varphi'})_{z}$,  $K_{\ell}=K^{\dboxtimes \ell}$ and $K_{\ell}^{-1}=(K^{-1})^{\dboxtimes \ell}$.

Recall the proof of \autoref{invariance1} (2). In the contact case, we still have an isomorphism
\[ \sP_{\varphi_z({\mathscr{U}})} \cong K^{-1}\circ \sP_{\mathscr{U}} \circ K, \]
as well as the auto-equivalence $\kappa(F)\coloneqq K_{\ell}^{-1}\circ F \circ K_{\ell}$ of $D_{\bZ/\ell}((X\times \bR)^{2\ell})$. 

So, we only need to construct an isomorphism
\begin{equation}\label{proof contact inv. commutative.}
  \kappa(\beta^c_{n}\bK)=K_\ell^{-1}\circ \beta^c_{n}\bK \circ K_\ell \cong \beta^c_{n}\bK,  
\end{equation}
where $\beta^c_{n}=\beta^c_{\ell,n,X}$. As in \autoref{invariance1}, we only need to find an isomorphism 
\[\beta_n^c\bK \circ K_\ell \cong K_\ell \circ \beta_n^c\bK.\]
To emphasize the difference between the contact case and the symplectic case, let us present the construction precisely. Let $W=X\times\bR$, $f \colon W^\ell \to W^\ell$, $(w_1,\ldots,w_\ell) \mapsto (w_2,\ldots,w_\ell,w_1)$ where $w_i=(\bq_i,t_i)$ and $\tnT_{c}^\ell:W^\ell \rightarrow W^\ell$, $(w_1,\ldots,w_\ell)\mapsto (\tnT_{c}(w_1),\ldots,\tnT_{c}(w_\ell))$, where $c\in\bR$ and $\tnT_{c}(w_i)=\tnT_{c}(\bq_i,t_i)=(\bq_i,t_i+c)$. Set $Y=W^\ell $ and identify $Y^2 = (W^2)^\ell$ by $(w^1_1,\ldots,w^1_\ell, w^2_1,\ldots,w^2_\ell) \mapsto (w^1_1,w^2_1,\ldots,w^1_\ell,w^2_\ell)$. Then $\beta^c_{n}\bK$ is, up to orientation and shift, the
constant sheaf on the graph of the composition $f\circ \tnT^\ell_n=\tnT^\ell_n\circ f$. Precisely, we have \[\beta_n^c\bK \cong\bK_{\Gamma_f}\circ \bK_{\Gamma_{\tnT_{n}^\ell}} \circ E \cong E\circ \bK_{\Gamma_f}\circ \bK_{\Gamma_{\tnT_{n}^\ell}} ,\] where $E = \delta_{Y^2!}(\omega_Y)$, with $\omega_Y$ the dualizing sheaf and $\delta_{Y^2}$ the usual diagonal embedding. The relation $f\circ \tnT^\ell_n=\tnT^\ell_n\circ f$ implies $ \bK_{\Gamma_f}\circ \bK_{\Gamma_{\tnT_{n}^\ell}} \cong \bK_{\Gamma_{\tnT_{n}^\ell}}\circ\bK_{\Gamma_f}$. Moreover we have $E\circ- \cong -\circ E$.

Now we have the general fact $G \circ \bK_{\Gamma_g} \cong (\id_Y \times g )_!(G)$ for any $G$ and any map $g$.  This formula has the symmetric form $\bK_{\Gamma'_g} \circ G\cong (g\times \id_Y)_!(G)$ where $\Gamma'_g$ is the switched graph $\Gamma'_g =\{(g(y),y): y\in Y\}$. When $g$ is invertible, we have $\Gamma_{g^{-1}}= \Gamma'_g$. So we obtain 
\[K_{\ell} \circ \beta^c_n \bK \cong K_{\ell} \circ \bK_{\Gamma_{\tnT_{n}^\ell}}\circ \bK_{\Gamma_f}  \circ E \cong   (\id_Y \times f )_!(K_{\ell} \circ \bK_{\Gamma_{\tnT_{n}^\ell}})\circ E,\]
and
\[\beta^c_n\bK \circ K_\ell \cong E\circ \bK_{\Gamma_f}\circ \bK_{\Gamma_{\tnT_{n}^\ell}}\circ K_\ell=E\circ \bK_{\Gamma'_{f^{-1}}}   \circ \bK_{\Gamma_{\tnT_{n}^\ell}}\circ K_\ell\cong E \circ (f^{-1} \times \id_Y  )_!(\bK_{\Gamma_{\tnT_{n}^\ell}}\circ K_\ell).\]
Now, recall the GKS quantization $K(\widehat{\varphi'})$ satisfies the $\bZ$-equivariant condition \eqref{Z-equi for contact isotopy}, so the restriction on $z$-slices, $K=K(\widehat{\varphi'})_{z}$, also satisfies 
\[ K\circ \bK_{\Delta_{X^2}\times \{(t,t+n):t\in \bR\}} \cong \bK_{\Delta_{X^2}\times \{(t,t+n):t\in \bR\}}  \circ K.\]
Notice that $\Delta_{X^2}\times\{(t,t+n):t\in \bR\}=\Gamma_{\tnT_{n}}$ where $\tnT_{n}(x,t)=(x,t+n)$. Therefore we have
\[K_\ell \circ \bK_{\Gamma_{\tnT_{n}^\ell}}\cong \bK_{\Gamma_{\tnT_{n}^\ell}}\circ K_\ell.\]
Hence we have \[K_{\ell} \circ \beta^c_n \bK   \cong   (\id_Y \times f )_!(K_{\ell} \circ \bK_{\Gamma_{\tnT_{n}^\ell}})\circ E\cong E\circ (\id_Y \times f )_!( \bK_{\Gamma_{\tnT_{n}^\ell}}\circ K_{\ell} ).\]
Then the isomorphism \eqref{proof contact inv. commutative.} follows from \[(\id_Y \times f )_!( \bK_{\Gamma_{\tnT_{n}^\ell}}\circ K_{\ell} )\cong (f \times f )_!(f^{-1} \times \id_Y  )_!( \bK_{\Gamma_{\tnT_{n}^\ell}}\circ K_{\ell} )\cong (f^{-1} \times \id_Y  )_!( \bK_{\Gamma_{\tnT_{n}^\ell}}\circ K_{\ell} ).\]
The isomorphism $\Phi^{\bZ/\ell,c}_{z,n\ell}$ follows.
\end{proof}
\begin{RMK}The significant difference between $C^{\bZ/\ell}_T$ and $\mathscr{C}^{\bZ/\ell}_{n\ell}$ is that the definition of $\alpha^c$ also twist $t$ variables while $\alpha$ only twist $\bq$ variables. This is crucial for the contact invariance in \autoref{contactinvariance1}. However \autoref{contact sympl comparision} shows that when we consider the admissible sets of the form $U\times S^1$ for $U\subset T^*X$, the Chiu-Tamarkin complex itself is not affected by the difference. This is helpful for our computations.
\end{RMK}

Now, we assume $\ell\in \bO$ and ${\mathscr{U}}\subset T^*X\times S^1$ is admissible. Then $H^*\mathscr{C}^{\bZ/\ell}_{n\ell}({\mathscr{U}},\bK)$ is a module of $A=\textnormal{Ext}^*_{\bZ/\ell}(\bK,\bK)$. For an orientable manifold $X$ and a field $\bK$, the {\em fundamental class} $\eta^{\bZ/\ell,c}_{n\ell}({\mathscr{U}})$ is defined as the image of the fundamental class $[X]^{\bZ/\ell}=[X]\otimes 1$ under the morphism
$H_d^{BM}(X,\bK)\otimes \textnormal{Ext}^0_{\bZ/\ell}(\bK,\bK)\cong H^0\mathscr{C}^{\bZ/\ell}_{n\ell}(T^*X\times S^1,\bK) \xrightarrow{i_{\mathscr{U}}^*} H^0\mathscr{C}^{\bZ/\ell}_{n\ell}({\mathscr{U}},\bK)$. Similarly to \autoref{functorial fundamental class}, \autoref{contactinvariance1} shows that the fundamental class is preserved under inclusion and contact isotopy.

For the definition of capacities, it is reasonable to require a discrete spectrum. Let $\mathbb{P}$ denote the set of all prime numbers.
\begin{Def}For an admissible open set ${\mathscr{U}}\subset T^*X\times S^1$, $k\in \bN$. Define 
\begin{equation*} [\textnormal{Spec}]({\mathscr{U}},k) \coloneqq
\{p \in \mathbb{P}:\eta^{\bZ/p,c}_{p}({\mathscr{U}},\bF_{p})\in u^kH^{*}\sC^{\bZ/p}_{p}({\mathscr{U}},\bF_{p})\}
\end{equation*}
and
\[[c]_k({\mathscr{U}})\coloneqq \min [\textnormal{Spec}]({\mathscr{U}},k) \in \mathbb{P}.\]
For a general open set $O$, we define
\[[c]_k({\mathscr{U}})=\sup\{[c]_k(O): O\subset {\mathscr{U}},\,O\text{ is admissible}\}.\]
\end{Def}
Let us discuss the properties of $[c]_k$. The invariance and monotonicity are true with the same proof as in the symplectic case. The proof of representing property is invalid now. 
The positivity for open sets is obviously true by definition. However it is possible that $[c]_k$ is always $2$, which is treated as the trivial situation here. To avoid this situation, we must address some restrictions on the size of domains. Consider the constrain given by the structure theorem, we assume $\ell$ be a prime number; and moreover, the computation of ball indicate we should take $a>1$ as a necessary size constraint for $B_a\times S^1$. This fits into the framework of \cite{eliashberg2006geometry} that a small contact ball can be squeezed into smaller contact balls. Therefore, we define
\begin{Def}\label{big contact open sets}For an open set ${\mathscr{U}}\subset T^*\bR^d\times S^1$, we say it is {\em big} if there is a prequantized ball $B_a\times S^1 \xhookrightarrow{contact} {\mathscr{U}}$ such that $a>1$. 
\end{Def}
In summary, we organize our discussions as the following theorem. In the contact case, the spectrum sets could provide us more interesting obstructions. So we state results of spectrum sets as well.

\begin{Thm}\label{contact capacities} The functions $[c]_k:\text{Open}(T^*X\times S^1)\rightarrow \mathbb{P}$ satisfy the following:
\begin{enumerate}[fullwidth]
    \item $[c]_k \leq [c]_{k+1}$ and $[\textnormal{Spec}]({\mathscr{U}},k+1)\subset [\textnormal{Spec}]({\mathscr{U}},k)$, for all $k\in \bN$. 
    \item For two open sets ${\mathscr{U}}_1 \subset {\mathscr{U}}_2$, then $[c]_k({\mathscr{U}}_1) \leq [c]_k({\mathscr{U}}_2)$ and $[\textnormal{Spec}]({\mathscr{U}}_2,k)\subset [\textnormal{Spec}]({\mathscr{U}}_1,k)$.
    \item For a compactly supported contact isotopy $\varphi:I\times T^*X\times S^1 \rightarrow T^*X\times S^1$, we have $[c]_k({\mathscr{U}})=[c]_k(\varphi_z({\mathscr{U}}))$ and $[\textnormal{Spec}]({\mathscr{U}},k)= [\textnormal{Spec}](\varphi_z^H({\mathscr{U}}),k)$. 
    \item When $X=\bR^d$ and ${\mathscr{U}}\subset T^*\bR^d\times S^1$ is big, then it cannot happen that $[c]_k({\mathscr{U}})=2$ for all $k\in \bN$.
\end{enumerate}
\end{Thm}

Finally, let us discuss the prequantized toric domains, i.e., $X_{\Omega}\times S^1$ for a symplectic toric domain $X_{\Omega}$. We say that $X_{\Omega}\times S^1$ is convex if $X_{\Omega}$ is convex. 

Actually, we do not need to change the arguments much here because we already set everything up well. Using \autoref{contact sympl comparision}, we only need to slightly change the statement of the structural theorem.
\begin{Thm}\label{structure contact toric domian module}Let $X_\Omega \times S^1 \subset T^*V\times  S^1$ be a big prequantized convex toric domain (that means $\|\Omega^\circ_1\|_{\infty}<1$, see \autoref{big contact open sets}) and $\ell\in \bO$. If $n\in \bN_0$ and $n\ell\leq p_\ell/\|\Omega^\circ_1\|_{\infty}$, we have:
\begin{itemize}[fullwidth]

\item For each $Z\in \Omega^\circ_{n\ell}$, the inclusion $\overline{ZO} \subset \Omega^\circ_{n\ell}$ induces a decomposition $\eta^{\bZ/\ell,c}_{n\ell}(X_{\Omega}\times S^1,\bF_{p_\ell})=u^{I(Z)}\Lambda_{Z,\ell}$ for a non-torsion element $\Lambda_{Z,\ell}\in H^{-2I(Z)}\sC^{\bZ/\ell}_{n\ell}(X_\Omega\times S^1,\bF_{p_\ell})$. In particular, $\eta^{\bZ/\ell,c}_{n\ell}(X_{\Omega}\times S^1,\bF_{p_\ell})$ is non-zero.

\item The minimal cohomology degree of $H^*\sC^{\bZ/\ell}_{n\ell}(X_\Omega\times S^1,\bF_{p_\ell})$ is exactly $-2I(\Omega^\circ_{n\ell})$, i.e.,
    \[ H^*\sC^{\bZ/\ell}_{n\ell}(X_\Omega\times S^1,\bF_{p_\ell})\cong  H^{\geq -2I(\Omega^\circ_{n\ell})}\sC^{\bZ/\ell}_{n\ell}(X_\Omega\times S^1,\bF_{p_\ell}) \] 
    and
\[ H^{-2I(\Omega^\circ_{n\ell})}\sC^{\bZ/\ell}_{n\ell}(X_\Omega\times S^1,\bF_{p_\ell})\neq 0.  \] 
\item $H^*\sC^{\bZ/\ell}_{n\ell}(X_\Omega\times S^1,\bF_{p_\ell})$ is a finitely generated $\bF_{p_\ell}[u]$-module. The free part is isomorphic to $A=\bF_{p_\ell}[u,\theta]$, so $H^*\sC^{\bZ/\ell}_{n\ell}(X_\Omega\times S^1,\bF_{p_\ell})$ is of rank $2$ over $\bF_{p_\ell}[u]$. 

The torsion part is located in cohomology degree $[-2I(\Omega^\circ_{n\ell}),-1]$. $H^*\sC^{\bZ/\ell}_{n\ell}(X_\Omega\times S^1,\bF_{p_\ell})$ is torsion free when $X_\Omega$ is an open ellipsoid.
\end{itemize}
\end{Thm}

\begin{Thm}\label{computation of capacities of contact convex toric domian}For a big prequantized convex toric domain $X_\Omega \times S^1 \subsetneqq T^*V \times S^1$, we have:
\begin{align*}
[c]_k(X_\Omega\times S^1)=\min\left\{p \in \mathbb{P}: \,\exists  {z}\in \Omega_{p}^\circ, I( {z})\geq k\right\}
    =\min\left\{p\in\mathbb{P}: p \geq c_k(X_\Omega) \right\}.
\end{align*}
\end{Thm}
\begin{proof}If $p\in [\textnormal{Spec}](X_\Omega\times S^1,k) $, then ${p} <\frac{p}{\|\Omega^\circ_1\|_{\infty}} $ and we can use the structure theorem by the bigness condition. Then we use the minimal degree result of \autoref{structure contact toric domian module} to show that $I(z)\geq k$ for some $z\in \Omega_p^\circ$. In particular, we have $p\geq c_k(X_\Omega)$.

Conversely, if prime number $p$ satisfies the condition $p\geq c_k(X_\Omega)$. We can find a $z\in \Omega_p^\circ$, and the existence of decomposition of the fundamental class in \autoref{structure contact toric domian module} implies that $\eta^{\bZ/p,c}_{p}({X_\Omega\times S^1},\bF_{p})\in u^kH^{*}\sC^{\bZ/p}_{p}({X_\Omega\times S^1},\bF_{p})$.
\end{proof}
The result is much weaker than the symplectic case, while it is still interesting. For example, when we consider ellipsoids, we have 
\begin{align*}
    [c]_k(E(a)\times S^1)=&\min\left\{p \in \mathbb{P}: \sum_{i=1}^d \left\lfloor\frac{p}{a_i}\right\rfloor \geq k\right\}\\
    =&\min\left\{p \in \mathbb{P}: p \geq c_k(E(a)) \right\},
\end{align*}
where $a=(a_1,\dots,a_d)$ and $1<a_1\leq a_2\leq\cdots\leq a_d$.

\bibliographystyle{bingyu}
\clearpage
\phantomsection
\bibliography{bibtex}

\begin{thebibliography}{{Tam}15}
\providecommand{\url}[1]{\texttt{#1}}
\providecommand{\urlprefix}{URL }
\expandafter\ifx\csname urlstyle\endcsname\relax
  \providecommand{\doi}[1]{doi:\discretionary{}{}{}#1}\else
  \providecommand{\doi}{doi:\discretionary{}{}{}\begingroup
  \urlstyle{rm}\Url}\fi
\providecommand{\eprint}[2][]{\url{#2}}

\bibitem[AI20a]{asano2017persistence}
Tomohiro Asano and Yuichi Ike.
\newblock \emph{Persistence-like distance on Tamarkin's category and symplectic
  displacement energy}.
\newblock Journal of Symplectic Geometry, 18(2020)(3), 613–649.
\newblock \doi{10.4310/jsg.2020.v18.n3.a1}.

\bibitem[AI20b]{asano2020sheaf}
Tomohiro {Asano} and Yuichi {Ike}.
\newblock \emph{Sheaf quantization and intersection of rational Lagrangian
  immersions}.
\newblock Annales l'Institut Fourier, (2020).
\newblock \doi{10.5802/aif.3554}.

\bibitem[BL94]{BernsteinLunts}
Joseph {Bernstein} and Valery {Lunts}.
\newblock \emph{Equivariant sheaves and functors}, \emph{Lecture Notes in
  Mathematics}, volume 1578.
\newblock Springer, 1994.
\newblock \doi{10.1007/bfb0073549}.

\bibitem[{Chi}17]{chiu2017}
Sheng-Fu {Chiu}.
\newblock \emph{Non-squeezing property of contact balls}.
\newblock Duke Mathematical Journal, 166(2017)(4), 605--655.
\newblock \doi{10.1215/00127094-3715517}.

\bibitem[CHLS10]{quantitivesymplectic.2010}
Kai {Cieliebak}, Helmut {Hofer}, Janko {Latschev}, and Felix {Schlenk}.
\newblock \emph{Quantitative symplectic geometry}.
\newblock In Boris Hasselblatt, editor, \emph{Dynamics, Ergodic Theory, and
  Geometry}. Cambridge University Press, 2010.
\newblock \doi{10.1017/cbo9780511755187.002}.

\bibitem[D'A13]{FTDAgnolo}
Andrea D'Agnolo.
\newblock \emph{{On the Laplace Transform for Tempered Holomorphic Functions}}.
\newblock International Mathematics Research Notices, 2014(2013)(16),
  4587--4623.
\newblock \doi{10.1093/imrn/rnt091}.

\bibitem[EH90]{EkelandHoferII}
Ivar {Ekeland} and Helmut {Hofer}.
\newblock \emph{Symplectic Topology and Hamiltonian Dynamics II.}
\newblock Mathematische Zeitschrift, 203(1990)(4), 553--568.
\newblock \doi{10.1007/bf01221255}.

\bibitem[EKP06]{eliashberg2006geometry}
Yakov {Eliashberg}, Sang~Seon {Kim}, and Leonid {Polterovich}.
\newblock \emph{Geometry of contact transformations and domains: orderability
  versus squeezing}.
\newblock Geometry \& Topology, 10(2006)(3), 1635--1747.
\newblock \doi{10.2140/gt.2006.10.1635}.

\bibitem[{Fra}16]{fraser2016}
Maia {Fraser}.
\newblock \emph{Contact non-squeezing at large scale in $\mathbb{R}^{2n}\times
  S^1$}.
\newblock International Journal of Mathematics, 27(2016)(13), 1650107.
\newblock \doi{10.1142/s0129167x1650107x}.

\bibitem[FSZ23]{large_non_equeezing_GF}
Maia {Fraser}, Sheila {Sandon}, and Bingyu {Zhang}.
\newblock \emph{Contact non-squeezing at large scale via generating functions},
  2023.
\newblock \href{https://arxiv.org/abs/2310.11993}{arXiv:2310.11993}.

\bibitem[{Gao}17]{radonHonghao}
Honghao {Gao}.
\newblock \emph{Radon Transform for Sheaves}, 2017.
\newblock \href{https://arxiv.org/abs/1712.06453}{arXiv:1712.06453}.

\bibitem[GH18]{gutthutchings2018capacities}
Jean {Gutt} and Michael {Hutchings}.
\newblock \emph{Symplectic capacities from positive $S^1$--equivariant
  symplectic homology}.
\newblock Algebraic \& Geometric Topology, 18(2018)(6), 3537--3600.
\newblock \doi{10.2140/agt.2018.18.3537}.

\bibitem[GKS12]{GKS2012}
St{\'e}phane {Guillermou}, Masaki {Kashiwara}, and Pierre {Schapira}.
\newblock \emph{Sheaf quantization of Hamiltonian isotopies and applications to
  nondisplaceability problems}.
\newblock Duke Mathematical Journal, 161(2012)(2), 201--245.
\newblock \doi{10.1215/00127094-1507367}.

\bibitem[GPS18]{GPS3}
Sheel {Ganatra}, John {Pardon}, and Vivek {Shende}.
\newblock \emph{Microlocal Morse theory of wrapped Fukaya categories}, 2018.
\newblock \href{https://arxiv.org/abs/1809.08807}{arXiv: 1809.08807}.

\bibitem[Gro57]{Tohoku}
Alexander Grothendieck.
\newblock \emph{{Sur quelques points d'algèbre homologique, I}}.
\newblock Tohoku Mathematical Journal, 9(1957)(2), 119 -- 221.
\newblock \doi{10.2748/tmj/1178244839}.

\bibitem[{Gro}85]{gromov1985}
Mikhael {Gromov}.
\newblock \emph{Pseudo holomorphic curves in symplectic manifolds}.
\newblock Inventiones mathematicae, 82(1985)(2), 307--347.
\newblock \doi{10.1007/bf01388806}.

\bibitem[GS14]{GS2014}
St{\'e}phane {Guillermou} and Pierre {Schapira}.
\newblock \emph{Microlocal theory of sheaves and Tamarkin's non Displaceability
  theorem}.
\newblock In \emph{Homological mirror symmetry and tropical geometry}, 43--85.
  Springer, 2014.
\newblock \doi{10.1007/978-3-319-06514-4\_3}.

\bibitem[{Gui}12]{conicexactlagrangianGuillermou}
St{\'e}phane {Guillermou}.
\newblock \emph{Quantization of conic Lagrangian submanifolds of cotangent
  bundles}, 2012.
\newblock \href{https://arxiv.org/abs/1212.5818}{arXiv:1212.5818}.

\bibitem[{Gui}13]{C0rigidityGuillermou}
St{\'e}phane {Guillermou}.
\newblock \emph{The Gromov-Eliashberg theorem by microlocal sheaf theory},
  2013.
\newblock \href{https://arxiv.org/abs/1311.0187}{arXiv:1311.0187}.

\bibitem[{Gui}16]{3cuspsconjectureGuillermou}
St{\'e}phane {Guillermou}.
\newblock \emph{The three cusps conjecture}, 2016.
\newblock \href{https://arxiv.org/abs/1603.07876}{arXiv:1603.07876}.

\bibitem[{Gui}23]{guillermou2019sheaves}
St{\'e}phane {Guillermou}.
\newblock \emph{Sheaves and symplectic geometry of cotangent bundles},
  \emph{Ast\'erisque}, volume 440.
\newblock Soci\'et\'e math\'ematique de France, 2023.
\newblock \doi{10.24033/ast.1199}.

\bibitem[Hat02]{hatcher}
Allen Hatcher.
\newblock \emph{Algebraic topology}.
\newblock Cambridge University Press, 2002.

\bibitem[{Hut}11]{ECHhutchings2011}
Michael {Hutchings}.
\newblock \emph{Quantitative Embedded Contact Homology}.
\newblock J. Differential Geom., 88(2011)(2), 231--266.
\newblock \doi{10.4310/jdg/1320067647}.

\bibitem[Iri22]{irie2019symplectic}
Kei Irie.
\newblock \emph{Symplectic homology of fiberwise convex sets and homology of
  loop spaces}.
\newblock Journal of Symplectic Geometry, (2022), 417 – 470.
\newblock \doi{10.4310/JSG.2022.v20.n2.a2}.

\bibitem[KS90]{KS90}
Masaki {Kashiwara} and Pierre {Schapira}.
\newblock \emph{Sheaves on manifolds: With a short gistory. Les d{\'e}buts de
  la th{\'e}orie des faisceaux . By Christian Houzel}, volume 292.
\newblock Springer, 1990.
\newblock \doi{10.1007/978-3-662-02661-8}.

\bibitem[KS06]{KS2006}
Masaki {Kashiwara} and Pierre {Schapira}.
\newblock \emph{Categories and sheaves}, volume 332.
\newblock Springer, 2006.
\newblock \doi{10.1007/3-540-27950-4}.

\bibitem[KS18]{KS2018}
Masaki Kashiwara and Pierre Schapira.
\newblock \emph{Persistent homology and microlocal sheaf theory}.
\newblock Journal of Applied and Computational Topology, 2(2018)(1-2), 83--113.
\newblock \doi{10.1007/s41468-018-0019-z}.

\bibitem[Kuo23]{kuo2021wrapped}
Christopher Kuo.
\newblock \emph{Wrapped sheaves}.
\newblock Advances in Mathematics, 415(2023), 108882.
\newblock \doi{10.1016/j.aim.2023.108882}.

\bibitem[Lee03]{JohnLee}
John~M. Lee.
\newblock \emph{Introduction to Smooth Manifolds}.
\newblock Springer New York, 2003.
\newblock \doi{10.1007/978-0-387-21752-9}.

\bibitem[Lon21]{lonergan_2021}
Gus Lonergan.
\newblock \emph{Steenrod operators, the Coulomb branch and the Frobenius
  twist}.
\newblock Compositio Mathematica, 157(2021)(11), 2494–2552.
\newblock \doi{10.1112/S0010437X21007569}.

\bibitem[{Nad}09]{nadler2009microlocal}
David {Nadler}.
\newblock \emph{Microlocal branes are constructible sheaves}.
\newblock Selecta Mathematica, 15(2009)(4), 563--619.
\newblock \doi{10.1007/s00029-009-0008-0}.

\bibitem[{Nad}16]{Nadlerwarppedfukaya2016}
David {Nadler}.
\newblock \emph{Wrapped microlocal sheaves on pairs of pants}, 2016.
\newblock \href{https://arxiv.org/abs/1604.00114}{arXiv:1604.00114}.

\bibitem[NZ09]{nadler2009constructible}
David {Nadler} and Eric {Zaslow}.
\newblock \emph{Constructible sheaves and the Fukaya category}.
\newblock Journal of the American Mathematical Society, 22(2009)(1), 233--286.
\newblock \doi{10.1090/s0894-0347-08-00612-7}.

\bibitem[{Sie}19]{SiegelHigherCapacities}
Kyler {Siegel}.
\newblock \emph{Higher symplectic capacities}, 2019.
\newblock \href{https://arxiv.org/abs/1902.01490}{arXiv:1902.01490}.

\bibitem[Sie21]{CompSiegelHigherCapacities}
Kyler Siegel.
\newblock \emph{{Computing Higher Symplectic Capacities I}}.
\newblock International Mathematics Research Notices, (2021).
\newblock \doi{10.1093/imrn/rnaa334}.

\bibitem[SS16]{Properbasechange}
Olaf~M Schn{\"u}rer and Wolfgang Soergel.
\newblock \emph{Proper base change for separated locally proper maps}.
\newblock Rendiconti del Seminario Matematico della Universit{\`a} di Padova,
  135(2016), 223--250.
\newblock \doi{10.4171/RSMUP/135-13}.

\bibitem[{Tam}15]{mc-Tamarkin}
Dmitry {Tamarkin}.
\newblock \emph{Microlocal Category}, 2015.
\newblock \href{https://arxiv.org/abs/1511.08961}{arXiv:1511.08961}.

\bibitem[Tam18]{tamarkin2013}
Dmitry Tamarkin.
\newblock \emph{Microlocal condition for non-displaceability}.
\newblock In Michael Hitrik, Dmitry Tamarkin, Boris Tsygan, and Steve Zelditch,
  editors, \emph{Algebraic and Analytic Microlocal Analysis}, 99--223.
  Springer, 2018.
\newblock \doi{10.1007/978-3-030-01588-6\_3}.

\bibitem[{Vit}92]{viterbo1992symplectic}
Claude {Viterbo}.
\newblock \emph{Symplectic topology as the geometry of generating functions}.
\newblock Mathematische Annalen, 292(1992)(1), 685--710.
\newblock \doi{10.1007/bf01444643}.

\bibitem[{Zha}20]{zhang2020quantitative}
Jun {Zhang}.
\newblock \emph{Quantitative Tamarkin theory}.
\newblock Springer, 2020.
\newblock \doi{10.1007/978-3-030-37888-2}.

\bibitem[{Zha}23]{CyclicZHANG}
Bingyu {Zhang}.
\newblock \emph{Idempotence of microlocal kernels and the $S^1$-equivariant
  Chiu-Tamarkin invariant}, 2023.
\newblock \href{https://arxiv.org/abs/2306.12316}{arXiv:2306.12316}.

\end{thebibliography}

\noindent
\parbox[t]{28em}
{\scriptsize{
\noindent
Bingyu Zhang\\
Centre for Quantum Mathematics, University of Southern Denmark\\
Campusvej 55, 5230 Odense, Denmark\\
Email: {bingyuzhang@imada.sdu.dk}
}}

\end{document}